\documentclass[reqno,11pt]{amsart} 
\usepackage[latin1]{inputenc}
\usepackage[greek,english]{babel}
\usepackage[all]{xy}
\usepackage{amscd}
\usepackage{amssymb,amsmath,latexsym}
\usepackage{amsthm}
\usepackage{enumitem}
\usepackage{mathrsfs,bbm}
\usepackage[dvipdfmx]{graphicx}
\usepackage[latin1]{inputenc}
\usepackage{graphics}
\usepackage{hyperref}  
\usepackage{soul}
\hypersetup{
   colorlinks=true, %set true if you want colored links
    linktoc=all,     %set to all if you want both sections and subsections linked
    linkcolor=blue,
    citecolor=blue,  %choose some color if you  want links to stand out
}

\newcommand{\Z}{{\mathbb{Z}}} 
\newcommand{\Q}{{\mathbb{Q}}}  
\newcommand{\N}{{\mathbb{N}}}

\newcommand{\M}{{\mathfrak{M}}}

\newcommand{\F}{{\mathbb{F}}}

\newcommand{\lra}{\longrightarrow}

\newcommand{\Hom}{\mathrm{Hom}}

\newcommand{\cyc}{\mathrm{cyc}}

\newcommand{\La}{\Lambda}

\newtheorem{theorem}{Theorem}[section]
\newtheorem{proposition}[theorem]{Proposition}
\newtheorem{rem}[theorem]{Remark}
\newtheorem{lemma}[theorem]{Lemma}
\newtheorem{corollary}[theorem]{Corollary}
\newtheorem{defn}[theorem]{Definition}
\newtheorem{conj}[theorem]{Conjecture}
\newtheorem{example}[theorem]{Example}

\newtheorem*{maintheorem2}{Theorem 2 (Algebraic Functional Equation)}

\author{Somnath Jha, Tadashi Ochiai} \thanks{
\\
The first author acknowledges the support of  SERB ECR Grant and SERB MATRICS grant. 
The second author is  partially  supported for this work by KAKENHI (Grant-in-Aid for Exploratory Research: Grant Number 24654004, 
Grant-in-Aid for Scientific Research \rm{(Red)}: Grant Number 26287005).}\address{Department of Mathematics and Statistics, Indian Institute of Technology Kanpur,  Kanpur 208016, India} \email{jhasom@iitk.ac.in} \address{Department of Mathematics, Graduate School of Science, Osaka University, Machikaneyama 1-1, Toyonaka, Osaka 5600043, Japan}\email{ochiai@math.sci.osaka-u.ac.jp} 
\begin{document}
\title{Control theorem and functional  equation of Selmer groups over $p$-adic Lie extensions}
\begin{abstract}  
Let $K_\infty$ be a $p$-adic Lie extension of a number field $K$ which fits into the setting of non-commutative Iwasawa theory 
formulated by Coates-Fukaya-Kato-Sujatha-Venjakob. 
For the first main result, we will prove the control theorem of Selmer group associated to a motive, which 
generalizes previous results by the second author and Greenberg. 
As an application of this control theorem, we prove 
the functional equation of the {}{dual} Selmer groups, which 
generalizes previous results by Greenberg, Perrin-Riou and Z\'{a}br\'{a}di. 
Especially, we generalize the result of Z\'{a}br\'{a}di for elliptic curves to general motives. 
Note that our proof is different from the proof of Z\'{a}br\'{a}di even 
in the case of elliptic curves. We also discuss the functional equation for the analytic $p$-adic $L$-functions 
{}{and check the compatibility with the functional equation of the {}{dual} Selmer groups}. 
\end{abstract}
%\begin{keywords}
%\end{keywords}
\maketitle
\tableofcontents 
%\tableofcontents%{hyperref}{links}
\section*{Introduction}
 Let us fix a prime number $p$ throughout the paper. 
Let $K$ be a number field and let $\mathcal{K}$ be a finite extension of $\mathbb{Q}_p$ whose ring of integers is  denoted by 
$\mathcal{O}$. For any $p$-adic Lie group $G$, we define $\Lambda (G):=  
\underset{U}{\varprojlim}~ 
\Z_p[G/U]$ to be the completed group ring where $U$ varies over open normal subgroups of $G$. We denote by $\Lambda_{\mathcal{O}} (G)$ 
the base extension $\Lambda (G) \otimes_{\mathbb{Z}_p} \mathcal{O}$. 
\par 
Let $A \cong (\mathcal{K} /\mathcal{O} )^{\oplus d}$ be a discrete Galois representation
of the absolute Galois group $G_K$ of $K$ and 
a $p$-adic Lie extension $K_\infty /K$. 
In the first part of the paper, we prove the control theorem of the Selmer group of $A$ over $K_\infty /K$. 
Control theorem was originally studied by Mazur \cite{mazur} for the behavior of the Selmer groups of ordinary elliptic curves when 
$K_\infty /K$ is the cyclotomic $\mathbb{Z}_p$-extension. 
The work was generalized in two directions: from elliptic curves to general Galois representations \cite{oc00} and 
from the cyclotomic $\mathbb{Z}_p$-extension to general non-commutative $p$-adic Lie-extensions \cite{gr2}.  
Here, we study a bi-product of these two different generalizations of \cite{oc00} and \cite{gr2}. 
To study the most general setting, we will prepare the notation and fix the setting below. Then, we will formulate and prove a control theorem  
with precise description of the kernel and the cokernel of the restriction map. 
\par 
First, we introduce error  terms {}{$E_0^A$ and $E_1^A$} which will appear in the control theorem and in the theorem of functional equation. 
Let $\mathcal{U}$ be the set of open subgroup of $G =\mathrm{Gal}(K_\infty /K )$. 
For each $U \in \mathcal{U}$, we denote by $K_U$ the fixed field $(K_\infty )^U $. We will denote by $u$ a prime in $K_U$ and by $v$ a prime in $K$.
We denote by $D_{U,u}$ (resp. $I_{U,u}$) the decomposition group (resp. inertia group) of $G_{K_U}$ at $u$. 
\par 
Let $\Gamma_{\mathrm{cyc}}$ be the Galois group  of the cyclotomic $\Z_p$-extension $ \mathrm{Gal} (K_\mathrm{cyc} /K)$ of $K$. 
For any continuous character $\rho: \Gamma_\mathrm{cyc} \rightarrow \mathbb{Z}_p^\times$, 
we denote by $A_\rho$ the twist $A\otimes \rho$ of the Galois representation $A$ by $\rho$. 
Let $P_U$ be the set of primes $u$ of $K_U$ such that the image of $I_{U,u} \subset G_{K_U} $ via $G_{K_U} \twoheadrightarrow U=\mathrm{Gal} (K_\infty /K_U)$ is infinite. 
{}{Under the assumption that $V= \varprojlim_n A_{p^n}  \otimes_{\mathbb{Z}_p} \mathbb{Q}_p$ satisfies 
the condition \rm{(Pan)} which will be stated in Definition \ref{s22ax}, }
we define error terms $E^{A_\rho}_{0,U}= 
\displaystyle{\bigoplus_{u\in P_U} }E^{A_\rho}_{U,u}$ and $E^{A_\rho}_{1,U}= 
\displaystyle{\bigoplus_{u\in P_U, u \nmid p}} E^{A_\rho}_{U,u}$ by putting: 

\begin{equation}\label{error-term-E}
E^{A_\rho}_{U,u} =
    \begin{cases}
     \mathrm{Ker} \left[ 
H^1 (I_{U,u} ,A_\rho )^{D_{U,u }} \longrightarrow H^1 (I_{\infty ,w} ,A_\rho )^{D_{\infty,w }} \right]   & \text{if }  u \nmid p.  \\
     \mathrm{Ker} \left[ 
H^1 (I_{U,u} ,A_\rho/F^+_v A_\rho )^{D_{U,u}} \longrightarrow H^1 (I_{\infty ,w} ,A_{\rho}/F^+_v A_{\rho} )^{D_{\infty,w }}
\right]       &  \text{if }  u \mid p 
    \end{cases}
  \end{equation}
{}{where $F^+_v A_\rho $ is the filtration induced by the filtration $F^+_v V_\rho = F^+_v V \otimes \rho $ given by the condition \rm{(Pan)} 
through the surjection $V_\rho \twoheadrightarrow A_\rho$. }
When the twisting  representation $\rho$ is trivial, 
we denote $E^{A_\rho}_{U,u}$, $E^{A_\rho}_{0,U}$ and $E^{A_\rho}_{1,U}$ by 
$E^{A}_{U,u}$, $E^{A}_{0,U}$ and $E^{A}_{1,U}$ respectively. 
We further define:
  \begin{equation}\label{e1ande2}
 E_0^{A_\rho}: = \underset{U}{\varprojlim} E^{A_\rho}_{0,U}= \underset{U}{\varprojlim} \underset{u \in P_U}{\oplus}E^{A_\rho}_{U,u},
 \end{equation}
\begin{equation}\label{e1ande23}
 E_1^{A_\rho}: = \underset{U}{\varprojlim} E^{A_\rho}_{1,U}= \underset{U}{\varprojlim} \underset{u \in P_U, u \nmid p}{\oplus}E^{A_\rho}_{U,u}.
  \end {equation}

Under the condition \rm{(Pan)} {}{ on $V= \varprojlim_n A_{p^n}  \otimes_{\mathbb{Z}_p} \mathbb{Q}_p$}, we will define the Selmer group $\mathrm{Sel}^{\mathrm{Gr}}_{A_\rho} (K_U )$ for $A$ over $K_U$ whose precise definition is recalled in 
Definition \ref{definition:Selmergroup}. We denote the inductive limit $\underset{U\in \mathcal{U}}{\varinjlim} \mathrm{Sel}^{\mathrm{Gr}}_{A_\rho} (K_U )$ by 
$\mathrm{Sel}^{\mathrm{Gr}}_{A_\rho} (K_\infty )$.  
We are interested in the following functorial restriction map:  
\begin{equation}\label{basic-res-map}
\mathrm{res}^A_{\rho,U} : \mathrm{Sel}^{\mathrm{Gr}}_{A_\rho} (K_U ) \longrightarrow \mathrm{Sel}^{\mathrm{Gr}}_{A_\rho} (K_\infty ) ^{G_{K_U}}. 
\end{equation}
and we have the following control theorem for these restriction maps: 
%%%%%%%%%%%%%%%%%%%%%%%%%%%%%%%%%%%%%%%%
%\begin{maintheorem1}
\begin{theorem}[{\bf Control Theorem}]\label{propisition:control_theorem}
Let $K_\infty$ be a $p$-adic Lie Galois extension of $K$ which fits into the setting $\mathrm{(G)}$, $\mathrm{(K)}$ of Definition \ref{definition:setting} 
$($see \S $\ref{s1})$. 
Let $V\cong \mathcal{K}^{\oplus d}$ be a $p$-adic representation of $G_K$ 
satisfying the condition $\mathrm{(Pan)}$ of Definition \ref{s22ax} $($see \S $\ref{s1})$. 
We fix a $G_K$-stable $\mathcal{O}$-lattice $T \subset V$ and we put $A:=V/T$. 
\par 
Then, the following statements hold for any continuous character $\rho: \Gamma_\mathrm{cyc} \rightarrow \mathbb{Z}_p^\times$ and 
for the restriction map $\mathrm{res}_{\rho,U}$: 
\begin{enumerate}
\item[\rm{(1)}] 
\begin{enumerate}
\item[\rm{(a)}]
 Assume the condition $\mathrm {(A)}$ of Definition \ref{set2b} $($see \S $\ref{s1})$ for $A = V/T$. Then the group $\mathrm{Ker} (\mathrm{res}^A_{\rho,U})$ is a finite group whose order is bounded independently of $U \in \mathcal{U}$. 
\item[\rm{(b)}]
Assume the condition $\mathrm{(Red)}$ of Definition \ref{definition:setting}. Then the group $\mathrm{Ker} (\mathrm{res}^A_{\rho,U})$ is a finite group 
when $U$ varies in $\mathcal{U}$. Further, the inverse limit $\varprojlim_U \mathrm{Ker} (\mathrm{res}^A_{\rho,U})$ is a finitely generated $\mathbb{Z}_p$-module.
\end{enumerate}
\item[\rm{(2)}] 
\begin{enumerate}
\item[\rm{(a)}]
Assume the condition $\mathrm {(A)}$ of Definition \ref{set2b}  
and the condition $(\mathrm{V}_q)$ of Definition \ref{s22ax}  for each $q$ ($\neq p$) dividing $P_U$. 
Assume one of the following situations 
\begin{enumerate}
\item[\rm{(i)}] The condition $(\mathrm{A}_p)$ of Definition \ref{set2b} holds.
\item[\rm{(ii)}] The condition $(\mathrm{Sol}_p)$ of Definition \ref{definition:setting} and the conditions $\mathrm{(Ord)}$ and $(\mathrm{V}_p)$ 
of Definition \ref{s22ax} hold.
\end{enumerate}
Then the group $\mathrm{Coker} (\mathrm{res}^A_{\rho,U})$ is a finite group when $U $ moves in $\mathcal{U}$. 
Further, we have a natural map $E^{A_\rho}_{0,U} \longrightarrow \mathrm{Coker} (\mathrm{res}^A_{\rho,U})$
whose kernel and cokernel are finite groups and their orders are bounded independently of $U \in \mathcal{U}$. 
\item[\rm{(b)}]
Assume the condition $\mathrm{(Red)}$ of Definition \ref{definition:setting} and the condition $(\mathrm{V}_q)$ of Definition \ref{s22ax} 
 for each $q $ ($\neq p$) dividing $P_U$. 
Assume one of the following situations 
\begin{enumerate}
\item[\rm{(i)}] The condition $(\mathrm{A}_p)$ of Definition \ref{set2b}  holds.
\item[\rm{(ii)}] The condition $(\mathrm{Sol}_p)$ of Definition \ref{definition:setting} and the conditions $\mathrm{(Ord)}$ and $(\mathrm{V}_p)$ 
of Definition \ref{s22ax} hold. 
\end{enumerate}
Then the group $\mathrm{Coker} (\mathrm{res}^A_{\rho,U})$ is a finite group when $U $ moves in $\mathcal{U}$. 
Further, we have a natural map $E_0^{A_\rho}=\varprojlim_U E^{A_\rho}_{0,U} \longrightarrow \varprojlim_U \mathrm{Coker} (\mathrm{res}^A_{\rho,U})$ 
whose kernel and cokernel are finitely generated $\mathbb{Z}_p$-modules. 
\end{enumerate}
\end{enumerate}
\end{theorem}
%\end{maintheorem1}
The proof of Theorem \ref{propisition:control_theorem} will be given in Section \ref{s2}. 
%\begin{rem}{\color{red}
%In a non-commutative $p$-adic Lie extension $K_\infty /K$, not only 
%primes  over $p$ but also primes away from $p$ can be infinitely ramified. Hence it is important to calculate contributions of local terms outside $p$ in the Control Theorem. 
%\par 
 %The   Control Theorem  above plays an important role in the proof of the algebraic functional equation. Note that both primes  above $p$ and primes above $\ell \neq p$ can be infinitely decomposed. Thus we need to calculate and estimate the growth of the kernel of the local terms in the Control Theorem carefully in the proof of the algebraic functional equation because, 
%even if these local terms are finite at every intermediate field, they might not be universally bounded.
%}
\begin{rem} In a non-commutative $p$-adic Lie extension $K_\infty /K$, we need to calculate and estimate the growth of the kernel and cokernel of the natural restriction map in the Control Theorem carefully in the proof of the algebraic functional equation because, 
even if these kernel and cokernel  are finite at every intermediate field, they might not be universally bounded. Also, to calculate the contribution of the local terms in the Control Theorem, we need to keep in mind that the 
primes  over $p$ as well as the  primes  above $\ell \neq p$ can be infinitely ramified.
\end{rem}
In the latter part of the paper, we prove the algebraic functional equation (the functional equation of the Selmer group of $A$ over $K_\infty /K$) by applying Theorem 1. {}{Let us set $H:=  \mathrm{Gal} (K_\infty /K_\mathrm{cyc} )$ and}  
let $\mathfrak M_H(G)$ be the category of finitely generated $\La_{\mathcal{O}}(G)$-modules $M$ such that $M/M(p)$ are finitely generated over 
$\La_{\mathcal{O}}(H)$ where $M(p)$ is a union of all $p^n$-torsion subgroups $M[p^n]$ of $M$. We denote by $K_0 (\mathfrak M_H(G))$ the 
Grothendieck group of $\mathfrak M_H(G)$. For an object $M$ in $\mathfrak M_H(G)$, the class of $M$ in $K_0 (\mathfrak M_H(G))$ is denoted by $[M]$. % and we will call $[M]$ the characteristic class of $M$.  
Given a   $\La_{\mathcal{O}}(G)$-modules $M$,  $M^\iota$ is a  $\La_{\mathcal{O}}(G)$-module where $ M^\iota =M$ as an 
 $\mathcal O$-module but $G$ acts via  the involution $m.g = g^{-1}m$ for all $g\in G$.

%%%%%%%%%%%%%%%%%%%%%%%%%%%%%%%%%%%%%%%%%%%%%%%%%%%%%%%%%%%%%%%%

%\begin{maintheorem2}
\begin{theorem}[{\bf Algebraic Functional Equation}]\label{function-equation-thm} 
 Let $K_\infty/K$, $V$, $A$ be as in the setting of Control Theorem $($Theorem \ref{propisition:control_theorem}$)$. Let $\mathrm{Sel}^{\mathrm{Gr}}_A (K_\infty )^\vee $denote the {}{dual} Greenberg Selmer group of $A$ over $K_\infty$. %Let $\dagger \in  \{ \mathrm{Gr}, \mathrm{BK}\} $. 
Assume that $\mathrm{Sel}^{\mathrm{Gr}}_A (K_\infty )^\vee$ and $\mathrm{Sel}^{\mathrm{Gr}}_{A^\ast(1)} (K_\infty )^\vee $ are in $\mathfrak M_H(G)$ 
{}{where we set $A^\ast = \mathrm{Hom} (A , \mathcal{K} /\mathcal{O} ) \otimes_{\mathcal{O}} \mathcal{K} /\mathcal{O}$}. 
Assume that all of the following conditions are satisfied simultaneously.
\begin{enumerate}
\item For every continuous character $\rho: \Gamma_\mathrm{cyc} \rightarrow \mathbb{Z}_p^\times$, the groups $\mathrm{Ker} (\mathrm{res}^{A^*(1)}_{\rho, U})$ and $\mathrm{Coker} (\mathrm{res}^{A^*(1)}_{\rho, U})$ are  finite groups for each $U \in \mathcal{U}$.
\item Either one of the following two conditions are satisfied.

\begin{enumerate} 
\item The order of $\mathrm{Ker} (\mathrm{res}^{A^*(1)}_{U})$ is bounded independently of $U \in \mathcal{U}$. 
Further, we have a natural map $E^{A^*(1)}_{0,U} \longrightarrow \mathrm{Coker} (\mathrm{res}^{A^*(1)}_{U})$
whose kernel and cokernel are finite groups and their orders are bounded independently of $U \in \mathcal{U}$. 
\item The inverse limit $\varprojlim_U \mathrm{Ker} (\mathrm{res}^{A^*(1)}_{U})$ is a finitely generated $\mathbb{Z}_p$-module.
Further, we have a natural map $E_0^{A^*(1)} \longrightarrow \varprojlim_U \mathrm{Coker} (\mathrm{res}^{A^*(1)}_{U})$ 
whose kernel and cokernel are finitely generated $\mathbb{Z}_p$-modules.
Moreover, hypothesis   $\mathrm{(Van)}$  of Definition \ref{definition:setting} holds.
\end{enumerate}
\item Either the hypothesis (a) or the  hypothesis (b) of  Proposition \ref{proposition:vanishing_higher_ext}  holds.
\item Either the condition $(\mathrm{A}_p)$ of Definition \ref{set2b} or the condition  $(\mathrm{Van}_p)$ of Definition \ref{definition:setting} holds.
\end{enumerate}
Then we have the following equality in $K_0 (\mathfrak{M}_H (G))$: 
$$
\left[ \mathrm{Sel}^{\mathrm{Gr}}_A (K_\infty )^\vee  \right] + \Bigl[ 
E_1^{A^\ast(1)}\Bigr] = 
\left[\left( \mathrm{Sel}^{\mathrm{Gr}}_{A^\ast (1)} (K_\infty )^\vee \right) ^\iota \right] 
$$ 
where $E_1^A: = \underset{U}{\varprojlim}E^A_{1,U}$ is an exceptional term defined by the equation \eqref{e1ande23}. 
\end{theorem}
%\end{maintheorem2}
A quite interesting phenomenon in the non-commutative case observed in the above theorem is that 
the functional equation is not really symmetric as in the commutative case and it has some error terms, which we call the exceptional divisor. 
The proof of Theorem \ref{function-equation-thm} will be given in Section \ref{s3}. 
%%%
\par 
{}{We denote by $\mathrm{Sel}^{\mathrm{BK}}_{A_\rho} (K_U )$ the Bloch-Kato Selmer group for $A$ over $K_U$ 
whose definition is similar to $\mathrm{Sel}^{\mathrm{Gr}}_{A_\rho} (K_U )$ replacing the local condition 
$H^1_{\mathrm{Gr}} (K_{U,u} ,A)$ of Definition \ref{definition:Selmergroup} by the local condition $H^1_{f} (K_{U,u} ,A)$ 
of Bloch-Kato given in \cite[Section 3]{bk}. We denote the inductive limit $\underset{U\in \mathcal{U}}{\varinjlim} \mathrm{Sel}^{\mathrm{BK}}_{A_\rho} (K_U )$ by 
$\mathrm{Sel}^{\mathrm{BK}}_{A_\rho} (K_\infty )$. It is known that $\mathrm{Sel}^{\mathrm{BK}}_{A_\rho} (K_U )$ (resp. $\mathrm{Sel}^{\mathrm{BK}}_{A_\rho} (K_\infty )$) is a subgroup of $\mathrm{Sel}^{\mathrm{Gr}}_{A_\rho} (K_U )$ (resp. $\mathrm{Sel}^{\mathrm{Gr}}_{A_\rho} (K_\infty )$). 
}
\begin{rem}[Algebraic Functional equation for Bloch-Kato Selmer group]\label{remark0.2}
Let us keep the hypotheses $($and the notation$)$ of Theorem \ref{function-equation-thm} above. 
 Then $\mathrm{Sel}^{\mathrm{BK}}_A (K_\infty )^\vee$ and $\mathrm{Sel}^{\mathrm{BK}}_{A^\ast(1)} (K_\infty )^\vee $ are in $\mathfrak M_H(G)$ and we have the following equality in $K_0 (\mathfrak{M}_H (G))$:

$$
\left[ \mathrm{Sel}^{\mathrm{BK}}_A (K_\infty )^\vee  \right] + \Bigl[ 
E_1^{A^\ast(1)}\Bigr] = 
\left[\left( \mathrm{Sel}^{\mathrm{BK}}_{A^\ast (1)} (K_\infty )^\vee \right) ^\iota \right] .
$$ 
This duality  for Bloch-Kato Selmer groups over $K_\infty$ follows from  the proof of Theorem \ref{function-equation-thm}. (see Section \ref{s3}).
\end{rem} 

\medskip

We give an example where all the conditions of Theorem \ref{propisition:control_theorem} and Theorem \ref{function-equation-thm}  are satisfied simultaneously. Also see Example \ref{example:51}.

\begin{example}\label{example:totalexample}
We consider the false-Tate curve extension $K_\infty = \mathbb{Q} (\mu_{p^\infty}, a^{1/p^\infty})$ appearing in Example \ref{example:2}. 

Let $f$ be {}{a normalized} eigen cuspform of weight $k\geq 2$ whose conductor is prime to $p$ and
assume that $a_p(f)$ is a $p$-adic unit as in Example \ref{example:2}(2). We take a lattice $T \subset V_f(j)$ 
{}{for an intger $j$ satisfying $1\leq j \leq k-1$} and put $A =T\otimes \mathbb{Q}_p /\mathbb{Z}_p$. 
Then, from the discussion of  Example \ref{example:firstone}(3), Example \ref{example:2}(3) and Example \ref{example2-2}, we have the following:

\begin{enumerate}
\item For any continuous character $\rho: \Gamma_\mathrm{cyc} \rightarrow \mathbb{Z}_p^\times$, $\mathrm{Ker} (\mathrm{res}^{A^\ast(1)}_{\rho,U})$ is a finite group whose order is bounded independently of $U \in \mathcal{U}$ {}{and} $\mathrm{Coker} (\mathrm{res}^{A^\ast(1)}_{\rho,U})$ is a finite group for any $U $. \par \par Further, the kernel and cokernel of the natural map $E_0^{A^\ast(1)} \longrightarrow \varprojlim_U \mathrm{Coker} (\mathrm{res}^{A^\ast(1)}_{U})$ are  finite groups of bounded order, independent of $U$.
\item Moreover, if $\mathrm{Sel}^{\mathrm{Gr}}_A (K_\infty )^\vee$ and $\mathrm{Sel}^{\mathrm{Gr}}_{A^\ast(1)} (K_\infty )^\vee $ are in $\mathfrak M_H(G)$, then 
$$
\left[ \mathrm{Sel}^{\mathrm{Gr}}_A (K_\infty )^\vee  \right] + \Bigl[ 
E_1^{A^\ast(1)}\Bigr] = 
\left[\left( \mathrm{Sel}^{\mathrm{Gr}}_{A^\ast (1)} (K_\infty )^\vee \right) ^\iota \right]  
\text{ in } K_0 (\mathfrak{M}_H (G)) $$ 
where $E_1^A$ is defined in \eqref{e1ande2}. 
\item Whenever $\mathrm{Sel}^{\mathrm{Gr}}_A (\Q(\mu_{p^\infty} ))^\vee$ is a finitely generated $\Z_p$-module, then  
$\mathrm{Sel}^{\mathrm{Gr}}_A (K_\infty )^\vee$ and $\mathrm{Sel}^{\mathrm{Gr}}_{A^\ast(1)} (K_\infty )^\vee $ are in $ \mathfrak M_H(G).$
\end{enumerate}
\end{example}

\begin{rem}[Historical comments]
Our result is a generalization of \cite{gr}, \cite{pr} and that of \cite{zab}, \cite{z2} and  \cite{bz}. The last three papers establish 
functional equation of Selmer groups of $p$-ordinary elliptic curves over certain $p$-adic Lie extensions  
$K_\infty /K$ including false-Tate extensions and $GL_2$-extensions. We allow more general Galois representations and 
more general $p$-adic Lie extensions compared to the above mentioned articles. There is also a related work of Hsieh \cite{mlh} on algebraic functional equation over CM fields.
\end{rem}
\begin{rem}[Technical comments]
The method of proof in our work is different from preceding results \cite{zab}, \cite{z2} and   \cite{bz} even in the case of Selmer groups 
for elliptic curves which were already studied by them. 
In their works, functional equations are proved {}{as follows}: %by two-step argument; 
for each intermediate extension $k$ of $K_\infty /K $ which is finite over $K$, 
we have the algebraic functional equation over the cyclotomic $\mathbb{Z}_p$-extension $k_{\cyc}$, obtained by Greenberg and Perrin-Riou. Then they take the limit when $k$ varies and  obtain the desired functional equation over $K_\infty$.   
\par 
In our method, in some sense, the argument becomes simpler. %since we do not take such a two-step argument. 
For each finite intermediate extension $k$ of $K_\infty /K $, 
we have a natural duality pairing of Selmer group thanks to the global duality theorems of Galois cohomology 
and this natural pairing is perfect assuming that these Selmer groups are of finite order \cite{flach}.  
By a technique called a twisting lemma developed in our previous article with Z\'{a}br\'{a}di \cite{jo}, 
we are reduced to the situation where Selmer groups are of finite order at every finite intermediate extension $k$ of $K_\infty /K$. 
Under this situation of finiteness, we can just take the limit of this pairing with respect to $k$ with help of 
Control theorem of Selmer group (Theorem \ref{propisition:control_theorem}) and we obtain the desired result for Theorem \ref{function-equation-thm}.
\end{rem}
In the philosophy of Iwasawa theory, the characteristic class of the Selmer group is expected to be an algebraic counterpart of 
the analytic $p$-adic $L$-function whose existence is hypothetical in general. Iwasawa Main Conjecture predicts 
an equality in $K_0 (\mathfrak{M}_H (G))$ between the characteristic class of the Selmer group and the class defined by the analytic $p$-adic $L$-function. 
\par

{}{Now, let us consider a specific situation, where we study the Selmer group and $p$-adic $L$-function associated to a normalized,   $p$-ordinary eigen cuspform $f$ of even weight $k\geq 2$ and level $\Gamma_0 (N)$ %whose $p$-th Fourier coefficient $a_p(f)$ is a $p$-adic unit 
such that  $N$ is square-free and the conductor $N_f$ of $f$ is not divisible by $p$ as discussed in Example 3 of Section \ref{s5}.  We take $V=V_{f,p}(\frac{k}{2})$,  the $p$-adic Galois representation for $f$ at the crtical value $k/2$ and  take a lattice $T$ of $V=V_{f,p}(\tfrac{k}{2})$ and we set $A =T\otimes \mathbb{Q}_p /\mathbb{Z}_p$. Now consider a  `false-Tate curve' extension $K_\infty = \mathbb{Q} (\mu_{p^\infty}, a^{1/p^\infty})$ for some  $p$-power free natural number $a$ (see Example \ref{example:2}).  
Let $\mathcal{L}_p(V_{f,p}(\tfrac{k}{2}))$ be the analytic $p$-adic $L$-function, whose existence is conjectural. 
 At every Artin representation 
$\eta$ of $\mathrm{Gal}(K_\infty /\mathbb{Q})$, it makes sense to evaluate $\mathcal{L}_p(V_{f,p}(\tfrac{k}{2}))$ at $\eta$  but 
the evaluated value $\eta(\mathcal{L}_p(V_{f,p}(\tfrac{k}{2})))$ is only well-defined modulo multiplication by a $p$-adic unit. Similarly, for any $M$ in $\mathfrak M_H(G)$,   $\eta([M])$  makes sense up to a $p$-adic unit. (see Section \ref{s6}.)} %for details.)
\par 
In Appendix A, we discuss a general conjectural framework concerning the conjectural existence and the interpolation property 
for expected $p$-adic $L$-functions (see \eqref{equation:interpolation_property}) associated to general $p$-ordinary motives. As for the interpolation property satisfied by 
$\mathcal{L}_p(V_{f,p}(\tfrac{k}{2}))$, we refer the reader to the conjectural interpolation in \eqref{p-adiclfn-ellipticcurve}.
%\eqref{equation:interpolation_property}.  
\begin{theorem}[{\bf Compatibility between the algebraic side and the analytic side}]\label{theorem3}
Let us assume the setting above of 
{}{an ordinary normalized eigen elliptic cuspform $f$ of even weight $k\geq 2$ and square-free level $\Gamma_0 (N)$}  with a false-Tate extension $K_\infty /\mathbb{Q}$. 
Assume that a conjectural $p$-adic $L$-function {}{$\mathcal{L}_p(V_{f,p}(\tfrac{k}{2})) \in K_1(\La_\mathcal{O}(G)_{{S^*}})$} with the interpolation property \eqref{p-adiclfn-ellipticcurve} exists. 
Then for any Artin representation $\eta$ of $\mathrm{Gal}(K_\infty /\mathbb{Q})$ we have the equality 
\begin{equation}\label{equation:evaluated_functional_equation1}
\frac{\eta \left( \left[ \mathrm{Sel}^{\mathrm{BK}}_A (K_\infty )^\vee  \right]\right) }{\eta \left( 
\left[ \left( \mathrm{Sel}^{\mathrm{BK}}_{A} (K_\infty )^\vee \right) ^\iota \right] \right)} 
={}{\frac{\eta \left( \mathcal{L}_p(V_{f,p}(\tfrac{k}{2})) \right)}{\eta \left( \mathcal{L}_p(V_{f,p}(\tfrac{k}{2}))^\iota\right)}}
\end{equation}
modulo multiplication by a $p$-adic unit. 
\end{theorem}
Theorem \ref{theorem3} is restated at Theorem \ref{1100} and we prove it there. A key point of the proof is to calculate 
the evaluations of the exceptional divisor at every Artin character $\eta$ of $\mathrm{Gal}(K_\infty /\mathbb{Q})$ in this 
specific situation. Recall that Iwasawa Main Conjecture mentioned above (see Conjecture \ref{773} for details)  predicts equalities 
$[ \mathrm{Sel}^{\mathrm{BK}}_A (K_\infty )^\vee  ] =\delta\big(\mathcal{L}_p(V_{f,p}(\tfrac{k}{2}))\big)$ and 
$\big[ \big( \mathrm{Sel}^{\mathrm{BK}}_{A^\ast (1)} (K_\infty )^\vee \big) ^\iota \big]= \delta\big(\mathcal{L}_p(V_{f,p}(\tfrac{k}{2})\big)^\iota)$ 
{}{ in $K_0(\mathfrak M_H(G))$ where $\delta=\delta_G$ is the standard map 
$K_1(\La_{\mathcal{O}}(G)_{S^*})\lra K_0(\La_{\mathcal{O}}(G), \La_{\mathcal{O}}(G)_{S^*})= K_0(\mathfrak M_H(G))$ 
(see \eqref{fundamental-exact-k-theory3})}. 
Thus the equality in \eqref{equation:evaluated_functional_equation1} holds true  
if Iwasawa Main Conjecture holds true. 
A remarkable point of Theorem \ref{theorem3} is that we prove \eqref{equation:evaluated_functional_equation1} without assuming 
Iwasawa Main Conjecture, thus supporting the validity of Iwasawa Main Conjecture. 
\\ 
\\ 
{\bf Outline of the paper:} 
\par 
In Section \ref{s1}, we fix the setting of our work. We list up some conditions on 
$p$-adic Lie extensions $K_\infty /K$ and on its Galois group $\mathrm{Gal} (K_\infty /K)$. 
We also list up some conditions on Galois representations $V$ with geometric origin with which we define the Selmer group. 
For these conditions, we give some examples to explain when they are satisfied and when they are not. 
\par 
In Section \ref{s2}, we prove Control Theorem (Theorem \ref{propisition:control_theorem}). We remark that Control theorem in non-commutative case 
is more complicated than its commutative counterpart. Sometimes it happens that kernels and cokernels are finite 
but unbounded. 
\par 
As remarked after Theorem \ref{function-equation-thm}, 
the exceptional divisor plays an important role in the statement of the algebraic functional equation. 
In Section \ref{s5}, we calculate some examples of the exceptional divisor. 
Using general results obtained in previous section, we determine 
the exceptional divisor for some non-commutative extensions  where $V$ is associated to an elliptic curve or a modular form. We also prove that 
the exceptional divisor is trivial when the extension $K_\infty /K$ is commutative.   
\par 
In Section \ref{sext}, we prove the vanishing of some higher extension groups associated to Selmer groups in the category 
$\mathfrak M_H(G)$. Main result of the section is Proposition \ref{proposition:vanishing_higher_ext} and we use some criterion 
of Ardakov and Wadsley for the triviality of the class of a module of finite cardinality in certain Grothendieck groups.   
\par 
In Section \ref{s3}, we finally give the proof of the algebraic functional equation (Theorem \ref{function-equation-thm}) using the vanishing result of 
the previous section. 
\par 
In Section \ref{s6}, we first prove the analytic functional equation for the conjectural analytic $p$-adic $L$-function 
of  an elliptic curve over the false-Tate curve extension. Then we prove that the algebraic functional equation 
is compatible with the analytic functional equation, as predicted by the Iwasawa Main Conjecture. 
The most important technical point is to calculate  
of the evaluation of the error terms on the algebraic functional equation at every Artin representation of 
$\mathrm{Gal}(K_\infty/\Q)$ and to check that they match with the `error terms '
of the analytic functional equation, given by the Euler factors at certain bad primes. 
\par 
In Appendix \ref{appendix}, we first recall some basic conjectures on the complex $L$-function associated to 
a motive $M$. Then we formulate the conjecture on the existence of analytic $p$-adic $L$-function $\mathcal L_p(V)$ in a non-commutative situation when $V$ is associated to a $p$-ordinary motive. 
\\ 
\\
{\bf Notation:} %Throughout the paper we fix a topological generator $\gamma$ of $\Gamma$. 
Given a  left  $\La_{\mathcal{O}}(G)$-modules $M$,  $M^\iota$ is a right $\La_{\mathcal{O}}(G)$-module where $ M^\iota =M$ as an 
 $\mathcal O$-module but $G$ acts via  the involution $m.g = g^{-1}m$ for all $g\in G$. We extend this action of $G$,  $\mathcal O$-linearly to make $M^\iota$ a right $\La_{\mathcal{O}}(G)$-module. {}{For a discrete module $A$ which is isomorphic to a direct sum of a finite copies of 
 $\mathcal{K} /\mathcal{O}$, we set $A^\ast = \mathrm{Hom} (A , \mathcal{K} /\mathcal{O} ) \otimes_{\mathcal{O}} \mathcal{K} /\mathcal{O}$. 
For a finite dimensional $\mathcal{K}$-vector space $V$, we denote its $\mathcal{K}$-linear dual by $V^\ast$. For a discrete $\mathcal{O}$-module $ M$, 
we denote the Pontryagin dual Hom$_{\mathrm{cont}}(M, \Q_p /\Z_p)$ by $M^\vee$.} 
\par Unless otherwise specified, all modules over $\La_{\mathcal{O}}(G)$ are considered as left modules.  However, given a left $\La_{\mathcal{O}}(G)$-module $M$, we view Ext$^i_{\La_{\mathcal{O}}(G)}(M, \La_{\mathcal{O}}(G))$ as a right $\La_{\mathcal{O}}(G)$-module and Ext$^i_{\La_{\mathcal{O}}(G)}(M^\iota, \La_{\mathcal{O}}(G))$ again as a left $\La_{\mathcal{O}}(G)$-module. 
\par  %Given a left $\La_{\mathcal{O}}(G)$-module $M$ in (left) $\mathfrak M_H(G)$ and a right $\La_{\mathcal{O}}(G)$-module $N$ in (right) $\mathfrak M_H(G)$, we can compare the class of the left module $M$ and the right module $N$ in the same  $K_0(\mathfrak M_H(G))$. Indeed,  to be precise, 
Note that, there are two different categories for $\mathfrak M_H(G)$; that of left $\La_{\mathcal{O}}(G)$-modules and that of right $\La_{\mathcal{O}}(G)$-modules. By abuse of notation, we denote both to them by the same notation $\mathfrak M_H(G)$. 
The abelian group $K_0(\mathfrak M_H(G))$ for the category $\mathfrak M_H(G)$ for left $\La_{\mathcal{O}}(G)$-modules and the category $\mathfrak M_H(G)$ for right $\La_{\mathcal{O}}(G)$-modules 
are canonically identified. This last fact is verified by the isomorphism 
$K_0(\mathfrak M_H(G)) \cong K_1({\La_\mathcal O(G)}_{S^*})/{\text{Image}(K_1({\La_\mathcal O(G)}))}$ (see \eqref{fundamental-exact-k-theory3} for details) and this  is implicitly used in Theorem \ref{function-equation-thm} and Proposition \ref{ext-groups-in-k0}.
\\ 
\\ 
{\bf Acknowledgement:} 
\\ 
Some parts of this work progressed a lot during the stay of the authors at TIFR Mumbai
and at ISI Kolkata. We express our thanks for the hospitality of TIFR 
and ISI Kolkata.   We also thank Gergely Z\'{a}br\'{a}di for  discussions. {}{We also thank the referee for comments and suggestions which helped us to improve the article.}
%%%%%%%%%%%%%%%%%%%%%%%%%%%%%%%%%%%%%
%%%%%%%%%%%%%%%%%%%%%%%%%%%%%%%%%%%%%
\section{Setting}\label{s1}
We fix the setting of our global Galois extension.  We denote by  $\chi_{\mathrm{cyc}}: \mathrm{G}_\Q \lra \Z_p^\times$, 
the $p$-adic cyclotomic character.
%%%%%%%%%%%%%%%%%%%%%%
\begin{defn}\label{definition:setting}
We consider the following conditions for a compact $p$-adic Lie group $G$ and a Galois extension $K_\infty /K$ of $K$: 
\begin{enumerate}
\item[\rm{(K)}]
Only  finitely many primes of $K$ are  ramified in the extension $K_\infty$ of $K$ and $\mathrm{Gal} (K_\infty /K)$ 
is isomorphic to $G$. The group $G$ does not have any elements of order $p$. 
\end{enumerate}
\begin{enumerate}
\item[\rm{(G)}] 
The cyclotomic $\Z_p$-extension $K_\mathrm{cyc}$ of $K$ is contained in $K_\infty$ so that $H: = \mathrm{Gal} (K_\infty/K_\mathrm{cyc})$ is a closed subgroup of $G$ and  $G/H$ is isomorphic to 
$\Gamma_{\mathrm{cyc}}= \mathrm{Gal} (K_\mathrm{cyc} /K)$.
\item[\rm{(Red)}] 
The Lie algebra $\mathrm{Lie} (G)$ attached to the $p$-adic Lie group $G$ by Lazard is reductive.
\item[\rm{(Van)}] \label{zp0}
Let $G$ and $H$ be as in the condition \rm{(G)} above. Then $[\Z_p] = 0 \text{ in }K_0(\M_H(G)).$
\end{enumerate}
Under the condition (K), let $\mathcal{U}$ be the set of all open subgroups of $G$ and we denote by $K_U$ 
the fixed subfield of $K_\infty$ by $U \in\mathcal{U}$. For every finite prime $v$ of $K$, we denote the decomposition group of $G$ at 
$v$ by $G_v$, %{}{the inertia subgroup of $G$ at $v$ by $G'_v$. Is this used anywhere in the paper ?} 
We will denote by $u$ a prime in $K_U$ lying above $v$.  
\begin{enumerate}
 \item[\rm{(Van$_p$)}] \label{zpp0}
Let $G$ and $H$ be as in the condition \rm{(G)} above. For any prime $v$ in $K$ dividing $p$, set $H_v : = H \cap G_v$. Then $[\Z_p] = 0 \text{ in }K_0\big(\M_{H_v}(G_v)\big).$
%\item[$(\mathrm{Reg}_p)$] 
%For every prime $v$ of $K$ dividing $p$ and for every open normal subgroup $U$ of $G$, Iwasawa algebra $\Lambda (U_v )$ 
%for $U_v :=U \cap G_v$ has a regular sequence in the sense that  
%we have a sequence $x_{U,1} , \ldots , x_{U,r} \in \Lambda (U_v )$
%satisfying three conditions as follows: 
%\begin{enumerate}
%\item[\rm{(i)}] 
%The multiplication by $x_{U,i}$ on $\Lambda (U_v ) /(x_{U,1} , \ldots , x_{U,i-1} )\Lambda (U_v )$
%is injective for every $i$ such that $1\leq i \leq r$.
%\item[\rm{(ii)}] 
%We have $\Lambda (U_v ) /(x_{U,1} , \ldots , x_{U,r} )\Lambda (U_v )=\mathbb{Z}_p$.  
%\item[\rm{(iii)}] 
%When $U_v /U'_v$ is infinite, we have $\Lambda (U_v ) /(x_{U,1} , \ldots , x_{U,r-1} )=\mathbb{Z}_p
%[U_v /U'_v ] $ where $U'_v = U \cap G'_v$. 
%\end{enumerate}
\end{enumerate} 
\end{defn}

\begin{rem}
An example of infinite extension $K_\infty /K$ whose Galois group is isomorphic to a compact $p$-adic 
Lie group is often obtained as a trivializing extension of $p$-adic representation of $G_K$ which has geometric origin.  
For such an extension $K_\infty /K$, the statement of the condition $\mathrm{(K)}$ saying that only finitely many prime ramifies, is known by proper-smooth base change theorem of etale cohomology. 
The statement of the condition $\mathrm{(K)}$ saying that $\mathrm{Gal}(K_\infty/K)$ is a $p$-adic Lie group, 
is known by the fact that a closed subgroup 
of $GL_d (\mathbb{Z}_p)$ is a compact $p$-adic Lie group.  In many arithmetic situations (see Examples \ref{example:2}, \ref{example:3} and \ref{example:4}) the condition that $\mathrm{Gal}(K_\infty/K)$  has no element of order $p$ is also satisfied.
\par There are some examples of $K_\infty /K$ which do not have geometric origin. 
For example, a false-Tate extension $K (\mu_{p^\infty} ,a^{1/p^\infty})$ seems to be not obtained as the trivializing extension of 
a $p$-adic Galois representation of geometric origin. But, 
it is also easy to see that false-Tate curve extensions satisfy the condition $\mathrm{(K)}$. 
\end{rem}
We then introduce some conditions for the setting of our $p$-adic Galois representation of the absolute Galois group $G_K$ of $K$. 

\begin{defn}\label{s22ax}
We consider the following conditions for a $p$-adic representation $V \cong \mathcal{K}^{\oplus d}$ on which we have a continuous and $\mathcal{K}$-linear action of the absolute Galois group $G_K$:  
\begin{enumerate} 
\item[\rm{(Geo)}] The representation $V$ is geometric in the sense that is is obtained by $p$-adic realization of certain pure motive of homogeneous weight over $K$.
%\item[\rm{(SS)}] The representation $V$ is semi-simple as a representation of $G_K$. 
\item[\rm{(Pan)}] For any prime $v$ of $K$ over $p$, $V$ restricted to  $G_{K_v}$ is a Hodge-Tate representation which has a $G_{K_v}$-stable  $\mathcal{K}$-subspace $F^+_v V \subset V$ such that the following isomorphisms hold: 
\begin{small}{\begin{align*}
& F^+_v V \otimes_{\mathcal{K}} \mathbb{C}_p \cong \underset{n\geq 1 }{\oplus}\mathbb{C}_p (n)^{\oplus r_n }, \\ 
& (V/F^+_v V) \otimes_{\mathcal{K}} \mathbb{C}_p \cong \underset{n\leq 0 }{\oplus}\mathbb{C}_p (n)^{\oplus r_n } ,
\end{align*}}\end{small}
where $r_n$ are non-negative integers almost all of which are zero.  
\item[\rm{(Ord)}] For every prime $v$ of $K$ dividing $p$, the action $\rho \vert_{G_{K_v}}$ is quasi-ordinary in the sense of \cite[Definition 3.1]{oc}. That is, we have a $G_{K_v}$-stable decreasing filtration $\mathrm{Fil}^i_v V \subset V$ such that the action of $G_{K_v}$ on $(\mathrm{Fil}^i_v V /\mathrm{Fil}^{i+1}_v V )\otimes \chi^{-i}_{\mathrm{cyc}}$ is trivial on an open subgroup of the inertia subgroup $I_v$.  
\item[$(\mathrm{V}_p)$] Assume the condition $\mathrm{(Ord)}
$. For every prime $v$ of $K$ dividing $p$, the generalized eigenvalues of the Frobenius element in $G_{K_v} /I_v$ 
acting on $((\mathrm{Fil}^i_v V /\mathrm{Fil}^{i+1}_v V )\otimes \chi^{-i}_{\mathrm{cyc}})^{I_v}$ are not roots of unity. 
\item[$(\mathrm{V}_q )$] For every prime $v$ of $K$ dividing a prime number $q$   $ \neq p$, and for any open subgroup $W_v \subset G_{K_v}$, we have $V^{W_v}=V^\ast (1)^{W_v} =0$.
\end{enumerate}
\end{defn}

\begin{rem}
The condition $\mathrm{(Ord)}$ for $V$ implies the condition $ \mathrm{(Pan)}$ for $V$ by setting $\mathrm{F}^+_v:= \mathrm{Fil}^1_v$.
\end{rem}
{}{
Under the assumption that the extension extension $K_\infty /K$ satisfies the condition $ \mathrm{(K)}$ and 
the $p$-adic representation $V$ satisfies the condition $\mathrm{(Ord)}$, we consider the following condition:  
\begin{enumerate}
\item[$(\mathrm{Sol}_p)$]\label{conditionsolp}
(a) For every prime $v$ of $K$ dividing $p$ and for every open normal subgroup $U$ of $G$, $U_u :=U \cap G_v$ is a solvable $p$-adic Lie group. 
That is, $G_v$ is expressed as a successive extension 
of $p$-adic Lie group of dimension one.\\
(b) Moreover, for each open subgroup $U$ of $G$, for each prime $u$ of $K_U=(K_\infty)^U $ dividing $p$ and for each $i>0$ and integer $j$, we have 
\begin{equation}\label{equation:conjugation}
\Hom (I^{(i)}_{U,u}/ I^{(i+1)}_{U,u}, \mathrm{Fil}^{j} _v V /  \mathrm{Fil}^{j+1} _v V )
^{D_{U,u}}=0 
\end{equation} 
 where $D_{U,u}$ and $I_{U,u}$ are respectively the decomposition subgroup and the inertia subgroup of  $G_{K_U}$   and $D_{U,u}$ acts on $I^{(i)}_{U,u}/ I^{(i+1)}_{U,u}$ 
by conjugation.  
 \end{enumerate}
}

In the example below, let us discuss the validity of the conditions for some Galois representations $V$ listed {}{in Definition \ref{s22ax}}. 
%%%%%%%%%%%%%%%%%%%%%%%%%%%%%%%%%%%%%%%%%%%%
\begin{example}\label{example:firstone}
\begin{enumerate}
\item[{\rm (1)}] 
Let $B$ be an abelian variety of dimension $d$ over $K$ and we define a $p$-adic 
Galois representation $V \cong \mathbb{Q}^{\oplus d}_p$ to be $V=T_p B \otimes_{\mathbb{Z}_p} \mathbb{Q}_p$ 
where $T_p B$ is the $p$-Tate module of $B$. By definition, $V$ satisfies the condition $\mathrm{(Geo)}$.  
The condition $\mathrm{(Pan)}$ is satisfied if $B$ has semistable reduction and the abelian variety 
part of the special fiber $B_v$ of the N\'{e}ron model is ordinary for any prime $v$ over $p$.  
The condition $(\mathrm{V}_p)$ is satisfied if and only if $B$ satisfies the condition $\mathrm{(Pan)}$ and $B$ has good reduction at every prime $v$ of 
$K$ dividing $p$. The condition $(\mathrm{V}_q)$ for $q \neq p$, holds always true for  this example. 
\item[{\rm (2)}] 
Let $f$ be a normalized eigen elliptic cuspform of weight $k\geq 2$ and level $\Gamma_1 (N)$ for a natural number $N$ divisible 
by $p$. Let $\mathcal{K}$ be a finite extension $\mathbb{Q}_p (\{ a_n (f)\})$ of $\mathbb{Q}_p$ obtained by adjoining Fourier coefficients of $f$. By Deligne and Shimura, 
we have a continuous $p$-adic Galois representation $V_f \cong \mathcal{K}^{\oplus 2}$ of $G_{\mathbb{Q}}$ 
which is unramified outside $N\infty$ 
on which the trace of the action of $\mathrm{Frob}_q$ for $q \nmid N $ is equal to $a_q (f)$. 
We take $V$ to be an appropriate Tate twist $V_f (j)$ of $V_f$. 
By the work of Scholl, $V$ satisfies the condition $\mathrm{(Geo)}$. The condition $\mathrm{(Pan)}$ is always satisfied for $j\leq 0$ and for 
$j\geq k$. For $1\leq j \leq k-1$, 
The condition $\mathrm{(Pan)}$ is satisfied if $a_p (f) \in \mathcal{K}$ is a 
$p$-adic unit $($see \cite[Thm 2.1.4]{wiles}$)$. 
The condition $(\mathrm{V}_p)$ is satisfied if $V$ satisfies the condition $\mathrm{(Pan)}$ and the local automorphic representation $\pi_{f,p}$ 
associated to $f$ at $p$ is an unramified principal series representation. 
The condition $(\mathrm{V}_q)$, for $q \neq p$, holds always true for $1\leq j \leq k-1$. 
\item[{\rm (3)}] 
 Let $B$ be an elliptic curve over $\Q$ with good ordinary reduction at a prime $p$. For an odd integer $d\geq 1$, let $V_d$ be the $d$-th symmetric power of $V_p B$. The critical Tate twists $V_d (j)$ are listed in  \cite[Lemma 3.3]{oc00}, for every $d \in \N$. In particular, when $d$ is odd, then $V_d(\frac{-d+1}{2})$ is the only critical twist. %On the other hand, if $d$ is even with $d/2$ odd,  then  $V_d(j)$ is critical for $j=\frac{-d}{2}, \frac{-d}{2} +1$. Finally, if $4 \mid d$, then $V_d$ has no critical twist. 
 Accordingly, we set $V$ to be the critical Tate twist $V_d(\frac{-d+1}{2})$, where $d$ is odd. As explained in the proof of \cite[Prop. 3.4]{oc00}, the representation $V$ satisfies $\mathrm{(Ord)}$  and hence the $\mathrm{(Pan)}$ condition. Let $q $ be an integer prime to $p$. The elliptic curve $B$ has either potentially good reduction or potentially 
multiplicative reduction at $q$.
When $B$ has potentially good reduction at $q$, the eigenvalues of the action of frobenius element at a prime $v\vert q$ on $V$ has complex absolute value $\neq 1$. Hence we have $\mathrm{V}^{G_{\mathbb{Q}_q}}=0$ for $V=V_d(j)$. When $B$ has potentially multiplicative reduction at $q$, there is a 
prime $v \vert q$ such that $V^{I_v}$ is of dimension one. The eigenvalues of the action of frobenius element at $v$ on $V^{I_v}$ 
has complex absolute value $\neq 1$. 
Hence we also have $\mathrm{V}^{G_{\mathbb{Q}_q}}=0$ for $V=V_d(j)$. Thus the condition $(\mathrm{V}_q)$ holds for $V=V_d(j)$.

\end{enumerate}
\end{example}
We introduce some conditions concerning both the setting of our $p$-adic Galois representation $V$ and the setting of our $p$-adic Lie extension $K_\infty$. When $V$ satisfies (Pan), we define $F^+_v A$ and  $A/ F^+_v A$ to be the image of  $F^+_v V$ and   $V/ F^+_v V$, respectively  in $A=V/T$.
%%%%%%%%%%%%%%%%%%%%%%
\begin{defn}\label{set2b}
Let $K_\infty$ a fixed $p$-adic Lie extension of $K$. 
We consider the following conditions for a $p$-adic representation $V \cong \mathcal{K}^{\oplus d}$ on which we have a continuous and $\mathcal{K}$-linear action of the absolute Galois group $G_K$:  
\begin{enumerate} 
\item[\rm{(A)}]
The groups $A^{G_{K_\infty}}$ and $(A^\ast)^{G_{K_\infty}}$ are finite. 
\item[$(\mathrm{A}_p)$]
Assume the condition $\mathrm{(Pan)}$ {}{of Definition \ref{definition:setting}}. 
For every prime $v$ of $K$ dividing $p$ and for every prime $w$ of $K_\infty$ over $v$, the groups $A^{G_{K_{\infty ,w}}}$, 
$(A^\ast )^{G_{K_{\infty ,w}}}$, $(A/ F^+_v A )^{G_{K_{\infty ,w}}}$, $(A^*/ F^+_v A^* )^{G_{K_{\infty ,w}}}$ are finite. 
For every prime $v$ of $K$ dividing $p$, we have $H^1_f (K_v ,A)=H^1_g (K_v ,A)$ where $H^1_f$ and $H^1_g$ is the local condition defined in 
\cite[\S 3]{bk}.
\end{enumerate}
\end{defn}
In the example below, let us discuss the validity of the conditions for some Galois representations $V$ listed in Definition \ref{s22ax} restricting ourselves to a case where $G=\mathrm{Gal} (K_\infty /K)$ is a semidirect product $\mathbb{Z}_p$ by $\mathbb{Z}_p$. 

%%%%%%%%%%%%%%%%%%%%%%%%%%%%%%%%%%%%%%%%%%%%%%%%%%%%%%%%
\begin{example}\label{example:2}
Let us consider the case 
$K=\Q$ and $K_\infty = \underset{n \geq 1}{\cup} K(\mu_{p^n} ,\sqrt[p^n]{a})$ where 
$a$ is a $p$-power free integer which is prime to $p$. Any prime $q$ of $\Q$ dividing  $ap$  is finitely decomposed in $K_\infty /\Q$  
and the decomposition group $G_q$ of $G$ is open in $G$. (see \cite[Lemma 3.9]{hv}.)
The extension $\mathrm{Gal} (K_\infty /K)$ satisfies {\rm (K)}, {\rm (G)}  {}{of Definition \ref{definition:setting}} and {\rm (Sol$_p$){\color{red}(a)}} {}{of Definition \ref{s22ax}}, but 
does not satisfy {\rm (Red)}  {}{of Definition \ref{definition:setting}}. 
\begin{enumerate}
\item[{\rm (1)}]
Let $B$ be an elliptic curve over $\Q$ which is good, ordinary  at $p$. We set $T= T_p B$, $V=T \otimes \mathbb{Q}_p$, $A=T \otimes \mathbb{Q}_p /\mathbb{Z}_p$. Then $G_{\Q_p}$ acts on $A/F^+_v A$ by a certain unramified character $\alpha$ of infinite order. 
The group $(A/F^+_v A )^{G_{K_{\infty ,w}}}$ is finite for every prime $w$ of $K_{\infty}$ dividing $p$ since 
the residue field of $K_{\infty ,w}$ is finite. 
Also, $G_{\Q_p}$ acts on $F^+_v A$ by $\chi_{\mathrm{cyc}}\alpha^{-1}$. Since 
$K_{\infty ,w}$ contains $\Q(\mu_{p^\infty})$ for every prime $w$ of $K_\infty$, $(F^+_v A )^{G_{K_{\infty ,w}}}$ is finite by the same reason as 
$(A/F^+_v A )^{G_{K_{\infty ,w}}}$. Hence we have $\# A ^{G_{K_{\infty ,w}}} < \infty$. 
We thus verified the first statement of the condition  {\rm  $(\mathrm{A}_p)$} {}{of Definition \ref{set2b}}. We can also verify that the second statement of the condition {\rm  $(\mathrm{A}_p)$} 
holds true by using results on the comparisons of $H^1_f (\Q_p ,A)$ and $H^1_g (\Q_p ,A)$ in \cite[\S 3]{bk}. Since we have 
$ A ^{G_{K_{\infty}}} \subset A ^{G_{K_{\infty ,w}}} $, we conclude $\# A ^{G_{K_{\infty }}} < \infty$, which implies 
the condition {\rm (A)} {of Definition \ref{set2b}}. 
\item[{\rm (2)}]
Let $f$ be a normalized eigen elliptic cuspform of weight $k\geq 2$ and level $\Gamma_1 (N)$ for a natural number $N$ divisible 
by $p$ such that $a_p(f)$ is a $p$-adic unit. 
The Galois representation $V_f$ associated to $f$ satisfies the condition {\rm (Pan)}. If we take a lattice $T \subset V_f$ and 
put $A =T\otimes \mathbb{Q}_p /\mathbb{Z}_p$, the action of $G_{K_v}$ on $F^+_v A$ and $A/F^+_v A$ is explicitly known by \cite[Thm 2.1.4]{wiles}. 
By a similar argument as in (1), we can verify the  the condition {\rm  $(\mathrm{A}_p)$} as well as the condition {\rm (A)} {}{of  Definition \ref{set2b}}
if we assume that the conductor of $f$ is not divisible by $p$. 
\item[{\rm (3)}]  Let us again consider the example $V_d$ with $d$ odd, positive as in the setting of Example \ref{example:firstone}(3) and  $V=V_d (j)$ 
be a Tate twist. We take a Galois stable lattice $T$ of $V$ and put $A=V/T$. Note that the weight of $V$ is odd, hence in particular non-zero. Since the residue field of $K_{\infty ,w}$ is finite, by a similar argument as in  this Example \ref{example:2}(1), we verify that $A^{G_{K_\infty,w}}$ and $(A/{F_v^+A})^{G_{K_\infty,w}}$ are both finite for $w\vert p$ and $H^1_f(\Q_p,A) =H^1_g(\Q_p,A)$.  Thus {\rm (A$_p$)} and  the condition  {\rm (A)}  {}{of Definition \ref{set2b}} both hold.

%a continuous character ρ : Γ −→ Z×p such that the largest U- coinvariant quotient M(ρ)U of M(ρ) is finite for every open normal subgroup U of G.
\end{enumerate}
Now, for either one of {\rm (1)}, {\rm (2)}or {\rm (3)} above, it is important to twist by a continuous character $\rho: \Gamma_\mathrm{cyc} \longrightarrow \mathbb{Z}^\times_p$ as given in Introduction.  
We note that we have $(A_\rho /F^+_v A_\rho )^{G_{K_{\infty ,w}}}\cong (A/F^+_v A )^{G_{K_{\infty ,w}}} \otimes \rho$ and 
$A_\rho ^{G_{K_{\infty }}}\cong (A ^{G_{K_{\infty }}})\otimes \rho$ thanks to the condition {\rm (G)} of Definition \ref{definition:setting}. 
Hence, for any continuous character $\rho: \Gamma_\mathrm{cyc} \longrightarrow \mathbb{Z}^\times_p$, the conditions {}{$(\mathrm{A}_p)$} and {\rm (A)} {}{of Definition \ref{set2b}} 
also hold true for $A_\rho$ in all of {\rm (1)}, {\rm (2)} and {\rm (3)} above. 
\end{example}

%%%%%%%%%%%%%%%%%%%%%%%%%%%%%%%%%%%%%%%%%%%%%%%%%%%%%%%%
\begin{example}\label{example:3}
Let $K$ be an arbitrary number field and $K_\infty = \underset{n \geq 1}{\cup} K(E_{p^n})$ where 
$E$ is a non-CM elliptic curve over $K$. 
The group $G=\mathrm{Gal}(K_{\infty} /K) $ is known to be an open subgroup of $GL_2 (\mathbb{Z}_p)$ and 
the extension $K_\infty /K$ always satisfies  {\rm (G)} of Definition \ref{definition:setting} and also satisfies {\rm (K)} of Definition \ref{definition:setting} for $p \geq 5$. We have different situations according to the reduction of $E$ at $p$. 
\begin{enumerate}
\item[{\rm (S)}]
Suppose that $E$ has good supersingular reduction at a prime $v$ over $p$. Then decomposition group $G_v$ is open subgroup of $G$. 
Hence the extension $K_\infty /K$ satisfies {\rm (Red)}. 
Let us take a Galois representation $A$ of $G_K$ as in either  Example \ref{example:2} (1)  or Example \ref{example:2} (2). 
In this case, it is known that the residue field at the primes of $K$ over $p$ does not extend in $K_\infty /K$. 
Hence we verify the conditions $(\mathrm{A}_p)$ and {\rm (A)} {}{of Definition \ref{set2b}} by the argument as in the previous example. 
\par Moreover, for any continuous character $\rho: \Gamma_\mathrm{cyc} \longrightarrow \mathbb{Z}^\times_p$, the conditions $(\mathrm{A}_p)$ and {\rm (A)} 
also hold true for $A_\rho$ in both cases of Example  \ref{example:2} above.
\item[{\rm (O)}]
Suppose that $E$ has ordinary reduction at a prime $v$ over $p$. The decomposition group $G_v$ of $G$ is contained in a Borel subgroup of $G$. 
Hence the extension $K_\infty /K$ satisfies  {\rm (Sol$_p$){(a)}} of Definition \ref{s22ax}. 
Let us take a Galois representation {}{$A=B_{p^\infty}$} of $G_K$  {}{for 
another elliptic curve $B$ over $\Q$ as in Example \ref{example:2} (1)}. 
In this case, the condition $(\mathrm{A}_p)$ of Definition \ref{set2b} does not hold. The condition {\rm (A)} {}{of Definition \ref{set2b} is} satisfied if $B$ is non-CM and 
not isogenous to $E$. 
\par 
Let $v$ be a prime over $p$ of $K$. 
In this case, the action of $G_{\mathbb{Q}_p}$ on graded pieces on 
{}{$V =T_p B \otimes \mathbb{Q}_p $} are given by characters $\eta , \eta'$ 
such that $\eta'$ is unramified and $\eta \eta'$ is equal to the 
$p$-adic cyclotomic character. For every prime $w$ over $p$ of $K_\infty$, 
we have successive extensions 
$K_w =K^{(0)} \subset K^{(1)} \subset K^{(2)} \subset K_{\infty ,w} ,$
where $K^{(i)}$ is the field corresponding to a successive filtration 
$D^{(0)} \supset D^{(1)} \supset D^{(2)} \supset D^{(3)}$ of $G_v$ which exists according to {\rm (Sol$_p$){(a)}}. 
We can take $K^{(i)}$ so that $K^{(1)}/K^{(0)}$ is an unramified isogenous
$\mathbb{Z}_p$-extension, $K^{(2)}/K^{(1)}$ is the local cyclotomic 
$\mathbb{Z}_p$-extension. In this case, $I^{(0)}/I^{(1)}$ is finite and 
the conjugate action of the decomposition group $G_{K_v}$ 
on $I^{(1)}/I^{(2)}$ is trivial since $K^{(2)}$ is defined on $K_v$.
The conjugate action of the decomposition group $G_{K_v}$ on 
$I^{(2)}/I^{(3)}$ is given by $\eta (\eta')^{-1}$.  
\par 
Now, if we take $V =T_p B \otimes \mathbb{Q}_p$ with $B$ a non CM 
elliptic curve which is isogenous to $E$, the action of 
$G_{K_v}$ on $\mathrm{Fil}^0 V/\mathrm{Fil}^1 V$ and $\mathrm{Fil}^1 V/\mathrm{Fil}^2 V$ are given by characters $\eta$ and $\eta'$ respectively. 
Now, since the value of the unramified character $\eta'$ is not a root of unity, we have 
\begin{equation*}
\Hom (I^{(i)}_{U,v}/ I^{(i+1)}_{U,v}, \mathrm{Fil}^{j} _v V /  \mathrm{Fil}^{j+1} _v V )
^{D_{U,v}}=0 
\end{equation*} 
for every $i=0,1,2$ and $j=0,1$, which verifies  {\rm (Sol$_p$){(b)}} of Definition \ref{s22ax}. 
\end{enumerate}
 
\end{example}

\begin{example}\label{example:4}
Let $F$ be a totally real field of degree $d$ over $\Q$ and $K$ be a CM field which is a quadratic extension of $F$. % in which primes above $p$ in $F$ splits. 
Let us consider a  Galois extension $K_\infty$ of $K$ with $G= \mathrm{Gal}(K_\infty/K) \cong  \Z_p^{\oplus d+1}.$ The group $G$ is abelian and satisfies {\rm (K)}, {\rm (G)}, {\rm (Red)}  in Definition \ref{definition:setting}. 
Let $E$ be an elliptic curve defined over $\Q$ with good ordinary reduction at all primes above $p$ in $K$. Then the condition $(V_p)$ of Definition \ref{s22ax} is satisfied. 
\par 
If the condition $(\mathrm{A}_p)$ $($respectively {\rm (A)}$)$  of Definition \ref{set2b} holds true for {}{$A= B_{p^\infty}$}, then  for any continuous character $\rho: \Gamma_\mathrm{cyc} \longrightarrow \mathbb{Z}^\times_p$, the condition $(\mathrm{A}_p )$ $($resp. {\rm (A)}$)$ also holds true for $A_\rho$.  Moreover $(V_p)$ of Definition \ref{s22ax} holds for $A_\rho$ as well. 
\par 
Though the condition $(\mathrm{A}_p)$ for {}{$A = B_{p^\infty}$} of Definition \ref{set2b} was always satisfied for the cyclotomic $\Z_p$ extension of $K$, 
the condition $(\mathrm{A}_p)$ for {}{$A = B_{p^\infty}$} may not be satisfied for $\Z_p^{\oplus {d+1}}$-extension $K_\infty$ of $K$. However the condition $(\mathrm{Sol}_p)$ in Definition  \ref{s22ax} is satisfied in our situation. To check  {\rm (Sol$_p$){(b)}}, we note that the action of the decomposition subgroup of $K_v$ on successive quotients of inertia subgroups by conjugation is trivial for a prime $v$ in $K$ above $p$ in this case. 
Hence we have  \begin{equation*}
\Hom (I^{(i)}_{U,v}/ I^{(i+1)}_{U,v},  V /  \mathrm{Fil}^{1} _v V )
^{D_{U,v}}= \Hom (I^{(i)}_{U,v}/ I^{(i+1)}_{U,v},  \mathrm{Fil}^{1} _v V  )^{D_{U,v}}=0
\end{equation*} 
that for {}{$V= T_pB \otimes \Q_p$} and for every $i$ and for every $U$ 
by a similar argument as in Example \ref{example:3}. 
\end{example}

We end examples on this section by explicitly citing another example where Theorem 1 and Theorem 2  hold.
\begin{example}\label{example:51}

We consider the false-Tate curve extension $K_\infty = \mathbb{Q} (\mu_{p^\infty}, a^{1/p^\infty})$ appearing in Example \ref{example:2}. 
Let us again consider the example of the $d$-th symmetric power $V_d$ with $d$ odd positive as in the setting of Example \ref{example:firstone}(3) and let $V$ 
be the unique critical twist $V_d (j)$ where $j= \frac{-d+1}{2}$. 
Set and $A=V/T$ as in the setting of Example \ref{example:firstone}(3) and Example \ref{example:2}(3). Then, by discussions in Example \ref{example:firstone}(3), Example \ref{example:2}(3) and Example \ref{example2-2}, we deduce

\begin{enumerate}
\item For any continuous character $\rho: \Gamma_\mathrm{cyc} \rightarrow \mathbb{Z}_p^\times$, $\mathrm{Ker} (\mathrm{res}^{A^\ast(1)}_{\rho,U})$ is a finite group whose order is bounded independently of $U \in \mathcal{U}$.  $\mathrm{Coker} (\mathrm{res}^{A^\ast(1)}_{\rho,U})$ is a finite group for any $U $. \par  \noindent Further, the kernel and cokernel of the natural map $E_0^{A^\ast(1)} \longrightarrow \varprojlim_U \mathrm{Coker} (\mathrm{res}^{A^\ast(1)}_{U})$  are  finite groups of bounded order, independent of $U$.
\item Moreover, if $\mathrm{Sel}^{\mathrm{Gr}}_A (K_\infty )^\vee$ is  in $\mathfrak M_H(G)$, then 
$$
\left[ \mathrm{Sel}^{\mathrm{Gr}}_A (K_\infty )^\vee  \right] + \Bigl[ 
E_1^{A^\ast(1)}\Bigr] = 
\left[\left( \mathrm{Sel}^{\mathrm{Gr}}_{A^\ast (1)} (K_\infty )^\vee \right) ^\iota \right]  
\text{ in } K_0 (\mathfrak{M}_H (G)) $$ 
where $E_1^A$ is defined in \eqref{e1ande2}. 
\item Whenever $\mathrm{Sel}^{\mathrm{Gr}}_A (\Q(\mu_{p^\infty} ))^\vee$ is a finitely generated $\Z_p$-module, then  
$\mathrm{Sel}^{\mathrm{Gr}}_A (K_\infty )^\vee$ is  in $ \mathfrak M_H(G).$
\end{enumerate}

\end{example}
Now, we define a Selmer group. Let us fix a $G_K$-stable $\mathcal{O}$-lattice $T$ of $V$ and denote $V/T$ by $A$. 
%%%%%%%%%%%%%%%%%%%%%%%%%%%%%%%%%%%%%%%%%
\begin{defn}\label{definition:Selmergroup}
Let $K_\infty$ be a $p$-adic Lie Galois extension of $K$ which fits into the setting $\mathrm{(G)}$ and $\mathrm{(K)}$ 
of Definition \ref{definition:setting}. 
Let $V\cong \mathcal{K}^{\oplus d}$ be a $p$-adic representation of $G_K$ 
satisfying the conditions $\mathrm{(Geo)}$ and $\mathrm{(Pan)}$ of Definition \ref{s22ax}. 
We fix a $G_K$-stable $\mathcal{O}$-lattice $T \subset V$ and we put $A:=V/T$. 
Let $\Sigma$ be a finite set of primes of $K$ which contain 
all primes of $K$ dividing $p \infty$, all primes of $K$ where the representation $V$ is ramified and 
all primes of $K$ where the infinite extension $K_\infty/K$ is ramified. For every $U \in\mathcal{U}$, 
we denote by $\Sigma_U$ (respectively  $\Sigma^{(\infty )}_U$) the set of primes of $K_U$ which are above primes of $K$ in $\Sigma$ (respectively  the set of infinite primes of $K_U$). 
\begin{enumerate}
\item 
For prime $u \in \Sigma_U \setminus \Sigma^{(\infty )}_U $, we define the local condition 
$H^1_{\mathrm{Gr}} (K_{U,u} ,A)$ as follows. 
\\ 
If $u \in \Sigma_U \setminus \Sigma^{(\infty )}_U$ does not divide $p$, we define $H^1_{\mathrm{Gr}} (K_{U,u} ,A)$ to be: 
$$
H^1_{\mathrm{Gr}} (K_{U,u} ,A) = \mathrm{Ker} 
\left[ H^1 (K_{U,u} ,A) \longrightarrow H^1 (K^{\mathrm{ur}}_{U,u} ,A)^{\mathrm{Gal} (K^{\mathrm{ur}}_{U,u} /K_{U,u})} \right]. 
$$ 
If $u \in \Sigma^{(p )}_U $, we define $H^1_{\mathrm{Gr}} (K_{U,u} ,A)$ to be: 
$$
H^1_{\mathrm{Gr}} (K_{U,u} ,A) = \mathrm{Ker} 
\left[ H^1 (K_{U,u} ,A) \longrightarrow H^1 (K^{\mathrm{ur}}_{U,u} ,A/\mathrm{F}^+_v A )^{\mathrm{Gal} (K^{\mathrm{ur}}_{U,u} /K_{U,u})} \right]. 
$$ 
\item 
We define the Greenberg's Selmer group $\mathrm{Sel}^{\mathrm{Gr}}_A (K_U)$ as follows. 
$$
\mathrm{Sel}^{\mathrm{Gr}}_A (K_U ) =\mathrm{Ker} \left[ H^1 (K_{\Sigma} /K_U  , A) \longrightarrow 
\displaystyle{\prod_{u \in \Sigma_U \setminus \Sigma^{(\infty )}_U} \frac{H^1 (K_{U,u} ,A)}
{H^1_{\mathrm{Gr}} (K_{U,u},A)}}
\right].
$$ 
%\item 
%Finally for the infinite extension $K_\infty/K$, we define $$\mathrm{Sel}^{\mathrm{Gr}}_A (K_\infty) := \underset{U}{\varinjlim}~\mathrm{Sel}^{\mathrm{Gr}}_A (K_U).$$
\end{enumerate}
\end{defn}

\section{Proof of the control theorem}\label{s2}
In this section, we prove the control theorem of Selmer group in a non-commutative $p$-adic Lie extension (Theorem \ref{propisition:control_theorem}). 
\begin{proof}[Proof of Theorem \ref{propisition:control_theorem}]
To ease the burden of notation, throughout the proof of this theorem, we will write $\mathrm{res}_{\rho ,U}:=\mathrm{res}^A_{\rho ,U}$. 
Let $v$ be a  prime of $K$ in $ \Sigma$ and $u$ be a prime of $K_U$ in $\Sigma_U$ such that $u \mid v$. Also $w$ will denote a prime in $K_\infty$ such that $w \mid u \mid v$. We denote a decomposition subgroup of $G$ at $v$ by $G_v \subset G$. %and an inertia subgroup of $G$ at $v$ by $G'_v \subset G$.  
For each $v \in \Sigma$ and for each $U \in \mathcal{U}$, we denote $U \cap G_v$ %(resp. $U \cap G'_v$) 
by $U_u$. %(resp. $U'_u$).  
\\ 
{}{[General strategy for both the assertion (1)  and the assertion (2)]}
\par 
We have the following commutative diagram: 
\begin{equation}\label{equation:control}
\small 
\text{
$
\begin{CD}
0 @>>> 
\mathrm{Sel}^{\mathrm{Gr}}_{A_\rho } (K_U ) @>>> H^1 (K_{\Sigma} /K_U  , A_\rho) 
@>{\mathrm{loc}_{\rho, U}}>> 
\displaystyle{\prod_{u \in \Sigma_U  } \frac{H^1 (K_{U, u} ,A_\rho)}
{H^1_{\mathrm{Gr}} (K_{U, u} ,A_\rho )}} \\ 
@. @V{\mathrm{res}_{\rho ,U}}VV @V{\mathrm{res}'_{\rho ,U}}VV @VV{\mathrm{res}''_{\rho ,U}}V \\ 
0 @>>> \mathrm{Sel}^{\mathrm{Gr}}_{A_\rho} (K_\infty )^{G_{K_U }} @>>> H^1 (K_{\Sigma} /K_\infty  , A_\rho)^{G_{K_U}}
@>>{\mathrm{loc}_{\rho, \infty } }>
 \left(
\displaystyle{\prod_{w \in \Sigma_\infty }}  \frac{H^1 (K_{\infty,w} ,A_\rho)}
{H^1_{\mathrm{Gr}} (K_{\infty,w} ,A_\rho )}\right)^{G_{K_U}}.\\  
\end{CD}
$
}
\normalsize 
\end{equation}
By the snake lemma applied to the above diagram, 
the group $\mathrm{Ker} (\mathrm{res}_{\rho ,U})$ is a subgroup of $\mathrm{Ker} (\mathrm{res}'_{\rho ,U})$. 
For each $U\in \mathcal{U}$, if the group $\mathrm{Ker} (\mathrm{res}'_{\rho ,U})$ is a finite group, 
$\mathrm{Ker} (\mathrm{res}_{\rho ,U})$ is a finite group whose order is bounded by $\# \mathrm{Ker} (\mathrm{res}'_{\rho ,U})$. 
Thus, when the inverse limit $\varprojlim_U \mathrm{Ker} (\mathrm{res}'_{\rho,U})$ is a finitely generated $\mathbb{Z}_p$-module, 
so is the group $\varprojlim_U \mathrm{Ker} (\mathrm{res}_{\rho,U})$. When the oder of the group $\mathrm{Ker} (\mathrm{res}'_{\rho,U})$ 
is bounded independently of $U \in \mathcal{U}$, so is the order of $\mathrm{Ker} (\mathrm{res}_{\rho,U})$. 
\par 
Also, by the snake lemma applied to the above diagram, 
the group $\mathrm{Coker} (\mathrm{res}_{\rho ,U})$ is an extension of $\mathrm{Coker} (\mathrm{res}'_{\rho ,U})$
by a subquotient of $\mathrm{Ker} (\mathrm{res}''_{\rho ,U})$. 
Under the assumption 
that $ \mathrm{Sel}^{\mathrm{Gr}}_{A^\ast (1)} (K_\infty )^\vee $ is a finitely generated 
torsion module over $\mathcal{O} [[\mathrm{Gal} (K_\infty /K) ]]$, the map 
$\underset{U}{\varinjlim} ( \mathrm{loc}_{\rho ,U})$ is known to be surjective. We omit the proof of this surjectivity, 
which follows from the global duality theorem of Galois cohomology and some standard arguments and we refer to 
Lemma 4.11 and Corollary 4.12 of \cite{oc06} for the argument of this technique for Greenberg type Selmer group. 
Hence, in order to have the desired conclusion on $\mathrm{Coker} (\mathrm{res}_{\rho ,U})$, it suffices 
to study $\mathrm{Coker} (\mathrm{res}'_{\rho ,U} )$ and $\mathrm{Ker}(\mathrm{res}''_{\rho ,U})$. 
\par 
Thus, from now on, we concentrate on the study of $\mathrm{Ker} (\mathrm{res}'_{\rho ,U} )$, 
$\mathrm{Coker} (\mathrm{res}'_{\rho ,U} )$ and 
$\mathrm{Ker}(\mathrm{res}''_{\rho ,U})$. It is sufficient to show the following statements to complete the proof of 
the theorem. 
\begin{enumerate} 
\item[{}{($\alpha$)}] 
For any $U \in \mathcal{U}$, $\mathrm{Ker} (\mathrm{res}'_{\rho ,U} )$ is a finite group 
whose order is uniformly bounded when $U$ varies in $\mathcal{U}$ in the case (A) in Definition \ref{set2b} and 
$\mathrm{Ker} (\mathrm{res}'_{\rho ,U} )$ is a finite group when $U$ moves in $\mathcal{U}$ in the case \rm{(Red)} in Definition \ref{definition:setting}. Moreover, in the case \rm{(Red)},  $\varprojlim_U \mathrm{Ker} (\mathrm{res}'_{\rho ,U} )$  is a finitely generated $\Z_p$-module.
\item[{}{($\beta$)}] 
For any $U \in \mathcal{U}$, $\mathrm{Coker} (\mathrm{res}'_{\rho ,U} )$ is a finite group 
whose order is uniformly bounded when $U$ varies in $\mathcal{U}$ in the case (A) and 
$\mathrm{Coker} (\mathrm{res}'_{\rho ,U} )$ is a finite group when $U$ moves in $\mathcal{U}$ in the case \rm{(Red)}.  Moreover, in the case \rm{(Red)},  $\varprojlim_U \mathrm{Coker} (\mathrm{res}'_{\rho ,U} )$  is a finitely generated $\Z_p$-module.
\item[{}{($\gamma$)}] 
For any $U \in \mathcal{U}$, $\mathrm{Ker} (\mathrm{res}''_{\rho ,U} )$ is a finite group 
which is isomorphic to {}{($E^{A_\rho}_{0,U} )$}  modulo uniformly bounded errors when $U$ varies in $\mathcal{U}$. %{}{in the case (A) and we have a natural map from   $E_0^{A_\rho}=\varprojlim_U E^{A_\rho}_U \lra  \varprojlim_U \mathrm{Ker} (\mathrm{res}''_{\rho ,U} )$  whose kernel and cokernel are finitely generated $\mathbb{Z}_p$-modules in the case \rm{(Red)}.}
\end{enumerate} 
\ \\ 
{}{[Argument for the points ($\alpha$) and ($\beta$)]}
\par  
Firstly, by the Inflation-Restriction sequence, we have 
\begin{equation}\label{equation:1res'}
\mathrm{Ker} (\mathrm{res}'_{\rho ,U}) \cong H^1 (K_\infty /K_U , (A\otimes \rho )^{G_{K_\infty}}) \quad \text{ and}
\end{equation}
\begin{equation}\label{equation:2res'}
\mathrm{Coker} (\mathrm{res}'_{\rho ,U}) \cong H^2 (K_\infty /K_U , (A\otimes \rho )^{G_{K_\infty}}). \ \ \quad  
\end{equation}

Suppose that we are in the case (A). 
Let us consider the order $d_i (U) := \# H^i (U ,\mathbb{Z}/p\mathbb{Z})$. Since $(A\otimes \rho )^{G_{K_\infty}}$ 
is a finite $p$-group and $U$ is a pro-$p$ group, $U$-module 
$(A\otimes \rho )^{G_{K_\infty}}$ is obtained by taking successive extensions of 
$\mathbb{Z} /p\mathbb{Z}$ for $\mathrm{ord}_p \# (A\otimes \rho )^{G_{K_\infty}}$ times. 
Hence we have $\# H^i (U ,(A\otimes \rho )^{G_{K_\infty}})= d_i (U) \cdot \mathrm{ord}_p \# (A\otimes \rho )^{G_{K_\infty}}$ for 
$U \in \mathcal{U}$ and $i \in \{ 1,2 \}$. 
\par 
Recall the following lemma from \cite[Lemma 2.1]{gr2}.  
The proof can be found in \cite{gr2} and it essentially follows from the results of \cite{dsms}.
%%%%%%%%%%%%%
\begin{lemma}\label{lemma:greenberg_cohomology_bound}
Let $G$ be a $p$-adic Lie group of dimension $d$. 
Then  $H^1 (Z,\mathbb{Z} /p\mathbb{Z} ) \text{ and } H^2 (Z,\mathbb{Z} /p\mathbb{Z} )$ are finite and their 
orders are bounded when $Z$ varies in closed subgroups of $G$. 
%
%
%
%\begin{enumerate}
%\item[{\rm (1)}]
%Let $Z$ be a closed subgroup of $G$. Then $H^1 (Z,\mathbb{Z} /p\mathbb{Z} )$ is finite and its 
%order is bounded when $Z$ varies. 
%There exists an open subgroup $U$ of $G$ such that $\# H^1 (Z,\mathbb{Z} /p\mathbb{Z} )\leq p^d $
%for any closed subgroups $Z$ of $U$.
%\item[{\rm (2)}] 
%Let $Z$ be a closed subgroup of $G$. Then $H^ 2 (Z,\mathbb{Z} /p\mathbb{Z} )$ is finite and its order is
%bounded when $Z$ varies.
%\end{enumerate}
\end{lemma}
By Lemma \ref{lemma:greenberg_cohomology_bound}, the number $d_i (U)$ is finite and bounded when $U \in \mathcal{U}$ varies. Thus, 
by \eqref{equation:1res'} and \eqref{equation:2res'}, we showed that $\mathrm{Ker} (\mathrm{res}'_{\rho ,U})$ and $\mathrm{Coker} (\mathrm{res}'_{\rho ,U})$ 
are finite groups whose orders are bounded independently of $U \in \mathcal{U}$ when we are in the case (A). 
\par 
Suppose that we are in the case \rm{(Red)}. Recall that we have the following exact sequence 
$$
0 \longrightarrow ((A\otimes \rho )^{G_{K_\infty}})_\mathrm{div} \longrightarrow (A\otimes \rho )^{G_{K_\infty}}
\longrightarrow (A\otimes \rho )^{G_{K_\infty}} /((A\otimes \rho )^{G_{K_\infty}})_\mathrm{div} \longrightarrow 0.
$$ 
Taking a part of the long exact sequence of Galois cohomology associated to this short exact sequence, we have the following exact sequence 
\begin{multline}\label{equation:exacst_sequence_for_res'}
H^i  (K_\infty /K_U ,((A\otimes \rho )^{G_{K_\infty}})_\mathrm{div} ) \longrightarrow 
H^i (K_\infty /K_U , (A\otimes \rho )^{G_{K_\infty}}) 
\\ \longrightarrow 
H^i (K_\infty /K_U , (A\otimes \rho )^{G_{K_\infty}} /((A\otimes \rho )^{G_{K_\infty}})_\mathrm{div}) 
\end{multline}
for $i=1,2$. 
In order to show the desired statements for $\mathrm{Ker} (\mathrm{res}'_{\rho ,U})$, $\mathrm{Coker} (\mathrm{res}'_{\rho ,U})$, 
$\varprojlim_U\mathrm{Ker} (\mathrm{res}'_{\rho ,U})$ and $\varprojlim_U \mathrm{Coker} (\mathrm{res}'_{\rho ,U})$ in the case (Red), 
it suffices to show similar statements for groups in the left end and the right end of \eqref{equation:exacst_sequence_for_res'}. 
\par 
Since the group $(A\otimes \rho )^{G_{K_\infty}} /((A\otimes \rho )^{G_{K_\infty}})_\mathrm{div}$ is a finite $p$-group, we 
can show that the group $H^i (K_\infty /K_U , (A\otimes \rho )^{G_{K_\infty}} /((A\otimes \rho )^{G_{K_\infty}})_\mathrm{div})$ 
in the right end is finite and bounded by the same argument using Lemma \ref{lemma:greenberg_cohomology_bound} as the case (A) above. 
Let us handle the group $H^i  (K_\infty /K_U ,((A\otimes \rho )^{G_{K_\infty}})_\mathrm{div} )$ 
of the left end in \eqref{equation:exacst_sequence_for_res'}. 
By taking a part of the long exact sequence of Galois cohomology associated to the short exact sequence
$$
0 \longrightarrow ((A\otimes \rho )^{G_{K_\infty}})_\mathrm{div}[p] \longrightarrow ((A\otimes \rho )^{G_{K_\infty}})_\mathrm{div}
\longrightarrow ((A\otimes \rho )^{G_{K_\infty}})_\mathrm{div} \longrightarrow 0 , 
$$ 
we have an surjection $H^i  (K_\infty /K_U ,((A\otimes \rho )^{G_{K_\infty}})_\mathrm{div}[p] ) \twoheadrightarrow 
H^i  (K_\infty /K_U ,((A\otimes \rho )^{G_{K_\infty}})_\mathrm{div})[p]
$ for $i=1,2$. 
Since the argument using Lemma \ref{lemma:greenberg_cohomology_bound} as in the case (A) also works for $H^i  (K_\infty /K_U ,((A\otimes \rho )^{G_{K_\infty}})_\mathrm{div}[p] )$, 
we prove that the group $H^i  (K_\infty /K_U ,((A\otimes \rho )^{G_{K_\infty}})_\mathrm{div})[p] $ 
is finite and bounded when $U$ varies in $\mathcal{U}$. 
\par 
The $\mathbb{Z}_p$-rank the cofree part of $H^i (K_\infty /K_U ,((A\otimes \rho )^{G_{K_\infty}})_\mathrm{div})$ 
is equal to the $\mathbb{Q}_p$-rank of the Lie algebra cohomology 
$H^i (\mathfrak{g},V^{G_{K_\infty}})$ where $\mathfrak{g}$ is the Lie algebra of the $p$-adic Lie group $G$. 
Since $\mathfrak{g}$ is reductive by the condition (Red) and $V^{G_{K_\infty}}$ is semisimple as a representation of $\mathfrak{g}$, 
we have $H^i (\mathfrak{g},V^{G_{K_\infty}}) =0$ by \cite[Thm 10]{hs}. 
This shows that $H^i  (K_\infty /K_U ,((A\otimes \rho )^{G_{K_\infty}})_\mathrm{div})$ is finite for $i=1,2$ and 
that $\varprojlim_U  H^i  (K_\infty /K_U ,((A\otimes \rho )^{G_{K_\infty}})_\mathrm{div})$ is a finitely generated $\mathbb{Z}_p$-module 
for $i=1,2$. 
\\ 
{}{[Argument for the point ($\gamma$)\ :\ contribution outside $p$]}
\par  
Let us calculate $\mathrm{Ker} (\mathrm{res}''_{\rho ,U})$ when $U$ varies. First, we take a prime $u \in \Sigma_U$ which is not over $p$ and we calculate the local contribution to $\mathrm{Ker} (\mathrm{res}''_{\rho ,U})$ at $u$.   
Since we have an isomorphism 
\begin{equation}\label{equation:h^1modh1gr_nonp}
\dfrac{H^1 (K_{U,u} , A_\rho )}{H^1_{\mathrm{Gr}} (K_{U,u}  , A_\rho)} \cong H^1 (I_{U,u} ,A_\rho)^{D_{U,u}} ,
\end{equation}
it suffices to calculate 
\begin{equation}\label{equation:restriction_local_nonp}
E^{A_\rho}_{U,u}:=\mathrm{Ker}\left[ H^1 (I_{U,u} ,A_\rho )^{D_{U,u}} \longrightarrow H^1 (I_{\infty ,w} ,A_\rho )^{D_{\infty,w }}\right] 
\end{equation}
where $w $ is a prime of $K_{\infty}$ which is over $u$. 
\par
We note that, if $u$ is not in the set $P_U$ defined in Introduction, 
$E^{A_\rho}_{U,u}$ is finite and bounded when $U$ varies since 
we have $I_{U,u} =I_{\infty ,w}$ for sufficiently small $U$. 
From now on, we assume that $u \nmid p$ is in the set $P_U$. 
\par 
In order to bound the group $E^{A_\rho}_{U,u}$, 
we consider the following commutative diagram: 
\begin{equation}\label{equation:CDoutsidep}
\begin{CD}
H^1 (D_{U,u} ,A_\rho ) @>>> H^1 (D_{\infty ,w} ,A_\rho  ) \\ 
@VVV @VVV \\ 
H^1 (I_{U,u} ,A_\rho  )^{D_{U,v }} @>>> H^1 (I_{\infty ,w} ,A_\rho  )^{D_{\infty,w }} . \\ 
\end{CD}
\end{equation}
We calculate the kernels and the cokernels of the maps 
in the diagram \eqref{equation:CDoutsidep} by 
Hochschild-Serre spectral sequence. 
First, the group $E^{A_\rho}_{U,u}$ is the kernel of the lower horizontal map of the diagram by definition. 
On the other hand, the kernel of the upper horizontal map 
of the diagram is isomorphic to $H^1 (U_u , (A_\rho )^{D_{\infty,w }})$ where $U_u =\mathrm{Gal}(K_{\infty ,w}/K_{U,v})$. 
The kernel of the left vertical map of the diagram \eqref{equation:CDoutsidep} is 
$H^1 (D_{U,u}/I_{U,u}, (A_\rho )^{I_{U,u}})$. 
The kernel of the right vertical map of the diagram \eqref{equation:CDoutsidep} is 
$H^1 (D_{\infty,w}/I_{\infty,w},(A_\rho )^{I_{\infty,w}})$, which is 
trivial since the profinite group $D_{\infty,w}/I_{\infty,w}$ has no pro-$p$-part but the coefficient $(A_\rho )^{I_{\infty,w}}$ is $p$-primary. 
\par 
Thus we have the following exact sequence 
\begin{equation}\label{equation:boundE}
0 \longrightarrow H^1 (D_{U,u}/I_{U,u},(A_\rho )^{I_{U,u}}) 
\longrightarrow H^1 (U_u , (A_\rho )^{D_{\infty ,w}}) 
\longrightarrow E^{A_\rho}_{U,u} 
\longrightarrow 0 . 
\end{equation}
Now, we have the following lemma. 
\begin{lemma}
Assume that the condition $(V_q)$ of Definition \ref{s22ax} holds for a prime $q$ not dividing $p$. Then the group $E^{A_\rho}_{U,u}$ is finite for every open subgroup 
$U$ of $G$ and for every prime $u$ of $K_U$ which does not divide $p$. 
\end{lemma}
\begin{proof}
Note that since the kernel of Artin representation $\rho$ is open in $G_K$, 
the condition $(V_q)$ of Definition \ref{s22ax} holds for $A$ if and only if 
the condition $(V_q)$ holds for $A_\rho$. 
By the condition $(V_q)$ and by the local Tata duality, 
$H^0 (D_{U,u} ,A_\rho  )$ and $H^2 (D_{U,u} ,A_\rho  )$ is 
are finite groups. 
Then, by Euler-Poincar\'{e} characteristic formula, 
the group $H^1 (D_{U,u} ,A_\rho )$ must be finite. Since 
we have $U_u = D_{U,u} /D_{\infty ,w}$ by definition, we have 
an Inflation-Restriction exact sequence:  
\begin{equation}\label{equation:boundE2} 
0 \longrightarrow 
H^1 (U_u , (A_\rho)^{D_{\infty ,w}}) 
\longrightarrow 
H^1 (D_{U,u} ,A_\rho )
\longrightarrow 
H^1 (D_{\infty ,w} ,A_\rho) ^{D_{U,u}}
\end{equation} 
Thus $H^1 (U_u , (A_\rho)^{D_{\infty ,w}})$ is finite. 
By \eqref{equation:boundE}, this implies that $E^{A_\rho}_{U,u}$ is finite. 
This completes the proof. 
\end{proof}
\ \\ 
{}{[Argument for the point ($\gamma$)\ :\ contribution over $p$]}
\par  
Next, we take a prime $u \in P_U$ which is over $p$ and we will show that the local contribution to $\mathrm{Ker} (\mathrm{res}''_{\rho ,U})$ from 
all primes $u$ of $K_{U}$ dividing $p$ is finite %and bounded independently of $U \in \mathcal{U}$ 
assuming either  the condition $(\mathrm{A}_p)$ of Definition \ref{set2b} or the condition (Sol$_p$) of Definition \ref{s22ax}.  %Moreover we will show if for a prime $v$ in $K$ above $p$, if $[G:G_v]$ is finite i.e. $v$ is finitely decomposed then the local contribution to $\mathrm{Ker} (\mathrm{res}''_{\rho ,U})$ from all primes $v$ of $K_{U}$ dividing $v$ is bounded independently of $U \in \mathcal{U}$.

We choose a prime $u \in \Sigma_U$ which is over $p$. Since we have an isomorphism 
\begin{equation}\label{equation:h^1modh1gr_p}
\dfrac{H^1 (K_{U,u} , A_\rho )}{H^1_{\mathrm{Gr}} (K_{U,u}  , A_\rho)} \cong H^1 (I_{U,u} ,A_\rho /F^+_v A_\rho )^{D_{U,u}} ,
\end{equation}
it suffices to calculate 
the kernel of 
\begin{equation}\label{equation:restriction_local_p}
H^1 (I_{U,u} ,A_\rho /F^+_v A_\rho )^{D_{U,u}} \longrightarrow H^1 (I_{\infty ,w} ,A_\rho /F^+_v A_\rho )^{D_{\infty,w }}
\end{equation}
where $w $ is a prime of $K_{\infty}$ which is over $u\mid v$. 
\\ 
{\bf Case $(\mathrm{A}_p)$:} Let us consider the following commutative diagram:
$$ 
\begin{CD}
H^1 (D_{U,u} ,A_\rho /F^+_v A_\rho ) @>>> H^1 (D_{\infty ,w} ,A_\rho /F^+_v A_\rho ) \\ 
@VVV @VVV \\ 
H^1 (I_{U,u} ,A_\rho /F^+_v A_\rho )^{D_{U,u }} @>>> H^1 (I_{\infty ,w} ,A_\rho /F^+_v A_\rho )^{D_{\infty,w }} . \\ 
\end{CD}
$$ 
The vertical homomorphisms of the diagram are surjective by Inflation-Restriction sequence and by the fact that 
$D_{U,u}/I_{U,u}$ and $D_{\infty,w }/I_{\infty,w }$ are pro-cyclic groups. 
By the condition $(\mathrm{A}_p)$ of Definition \ref{set2b}, the kernel of the left vertical map is finite and bounded independently of $U \in \mathcal{U}$.  
Hence, the kernel of \eqref{equation:restriction_local_p} is finite and bounded independently of $U \in \mathcal{U}$
if and only if the kernel of the map \eqref{equation:restriction_local_p_2} below is finite and bounded independently 
of $U \in \mathcal{U}$: 
\begin{equation}\label{equation:restriction_local_p_2}
H^1 (D_{U,u} ,A_\rho /F^+_v A_\rho ) \longrightarrow H^1 (D_{\infty ,w} ,A_\rho /F^+_v A_\rho ) . 
\end{equation}
By definition, $D_{U,u}/D_{\infty ,w}$ is isomorphic to the decomposition group $U_u$ of $U \subset G$ and 
$D_{\infty ,w}$ is identified with $G_{K_{\infty ,w}}$. 
The kernel of the map \eqref{equation:restriction_local_p_2} is hence isomorphic to $H^1 (U_u , (A_\rho /F^+_v A_\rho )^{G_{K_{\infty ,w}}})$ by 
Inflation-Restriction sequence.  
Since the group $(A_\rho /F^+_v A_\rho )^{G_{K_{\infty ,w}}}$ is finite by the condition $(\mathrm{A}_p)$, 
we prove that $H^1 (U_u,  (A_\rho /F^+_v A_\rho )^{G_{K_{\infty ,w}}})$ is finite and bounded independently 
of $U \in \mathcal{U}$ similarly as the earlier argument using Lemma \ref{lemma:greenberg_cohomology_bound}. 
\\ 
{\bf Case (Sol$_p$):} In this case, let us consider a decreasing filtration on $U_u =\mathrm{Gal}(K_{\infty ,w}/K_{U,u})$ starting from 
$U^{ (0)}_u =U_u $: 
$$
U^{ (0)}_u \supset U^{ (1)}_u\supset \cdots \supset U^{ (r-1)}_u \supset U^{ (r)}_u=0 
$$ 
such that $U^{(i)}_u/U^{(i+1)}_u$ is isomorphic to $\mathbb{Z}_p$ 
up to an extension by finite groups for $i= 0, \ldots , r-1$, 
which is required by the condition (Sol$_p$) of Definition \ref{s22ax}. 
By our hypothesis that $G$ has no element of order $p$, order of these  finite groups are prime to $p$, which do not contribute to the cohomology with $p$-primary coefficients. 
Via the surjection $D^{ (0)}_{U,u} :=D_{U,u} \twoheadrightarrow U_v$, the pull-back of the above filtration induces the following filtration: 
\begin{equation}
D^{ (0)}_{U,u} \supset D^{ (1)}_{U,u} \supset \cdots \supset D^{ (r-1)}_{U,u}  \supset D^{ (r)}_{U,u} =D_{\infty ,w}  
\end{equation}
such that we have $D^{(i)}_{U,u} /D^{(i+1)}_{U,u}  \cong U^{(i)}_u/U^{(i+1)}_u$ for $i= 0, \ldots , r-1$. Via the injection 
$I_{U,u} \hookrightarrow D_{U,u}$, the pull-back of this filtration induces 
the following filtration:
\begin{equation}\label{equation:filtration_on_inertia1}
I^{ (0)}_{U,u}   \supset I^{ (1)}_{U,u} \supset \cdots \supset I^{ (r-1)}_{U,u}  \supset I^{ (r)}_{U,u}=I_{\infty ,w}  
\end{equation}
such that we have $I^{(i)}_{U,u} /I^{(i+1)}_{U,u} \cong D^{(i)}_{U,u} /D^{(i+1)}_{U,u}  $ for $i= 1, \ldots , r-1$.  
As is explained below, only the initial graded piece $I^{(0)}_{U,u} /I^{(1)}_{U,u}$ depends on the situation, either a finite group of order prime to $p$ or isomorphic to $\mathbb{Z}_p$ 
up to an extension by finite groups. 
For the first case where $I_{U,u}$ is of finite index in $D_{U,u}$, we have $I^{(0)}_{U,u} /I^{(1)}_{U,u}$ is isomorphic to $\mathbb{Z}_p$ 
up to an extension by finite groups. 
For the second case where $I_{U,u}$ is of infinite index in $D_{U,u}$, 
$I^{(0)}_{U,u} /I^{(1)}_{U,u}$ is finite group of order prime to $p$. 
Note that the kernel of \eqref{equation:restriction_local_p} 
is finite and bounded when $U \in \mathcal{U}$ varies if the kernel of 
\begin{equation}\label{sucsseccive_extension1}
H^1 (I^{(i)}_{U,u} ,A_\rho /F^+_v A_\rho )^{D_{U,v}} \longrightarrow H^1 (I^{(i+1)}_{U,u} ,A_\rho /F^+_v A_\rho )^{D_{U,u}}
\end{equation} 
is finite and bounded when $U \in \mathcal{U}$ varies for each $i$. By Inflation-Restriction sequence, 
the kernel of \eqref{sucsseccive_extension1} is isomorphic to 
\begin{equation}
H^1 (I^{(i)}_{U,u} /I^{(i+1)}_{U,u} , (A_\rho /F^+_v A_\rho )^{I_{U,u}^{(i+1)}} )^{D_{U,u}}. 
\end{equation}
It is clear that the condition \rm{(Ord)} (resp. $(\mathrm{V}_p)$) in  Definition \ref{s22ax} holds true for $A$ if and only if it holds true for $A_\rho$. 
By the conditions \rm{(Ord)} and $(\mathrm{V}_p)$, 
$A_\rho /F^+_v A_\rho $ is a successive extension of the representation of type 
$(\mathcal{K}/\mathcal{O}) (\alpha \chi^m_{\mathrm{cyc}} \psi)$ where $\alpha$ is an unramified character of $G_{K,v}$, $m$ an integer, 
$\psi$ a finite order character of $I_{U,u}$. 
\par 
Let us discuss the case $i=0$, 
where we only need to discuss the first case where $I_{U,u}$ is of 
finite index in $D_{U,u}$ since $I^{(0)}/I^{(1)}$ is a finite group of order prime to $n$ for the second case. 
Since $I^{(0)}_{U,u} /I^{(1)}_{U,u}$ is the Galois group of the local cyclotomic $\mathbb{Z}_p$-extension of $K_{U,u}$ in this case, 
we have $(A_\rho /F^+_v A_\rho )^{I_{U,u}^{(i+1)}}=A_\rho /F^+_v A_\rho $ for every $i\geq 0$. 
Let us study the group $H^1 (I^{(i)}_{U,u} /I^{(i+1)}_{U,u} , (A_\rho /F^+_v A_\rho ) )^{D_{U,u}}$ for $i=0$. 
By looking at graded pieces, we will study the group 
\begin{equation}
H^1 (I^{(0)}_{U,u} /I^{(1)}_{U,u} , (\mathcal{K}/\mathcal{O}) (\alpha \chi^m_{\mathrm{cyc}} \psi) )^{D_{U,u}}
\end{equation}
If $m\not=0$, the pro-cyclic group $I^{(0)}_{U,u} /I^{(1)}_{U,u}$ acts non-trivially on 
$ (\mathcal{K}/\mathcal{O}) (\alpha \chi^m_{\mathrm{cyc}} \psi)$ and hence 
$H^1 (I^{(0)}_{U,u} /I^{(1)}_{U,u} , (\mathcal{K}/\mathcal{O}) (\alpha \chi^m_{\mathrm{cyc}} \psi) )^{D_{U,u}}=0$. 
If $m=0$, $I^{(0)}_{U,u} /I^{(1)}_{U,u}$ acts trivially on 
$ (\mathcal{K}/\mathcal{O}) (\alpha  \psi)$ and hence we have 
$$H^1 (I^{(0)}_{U,u} /I^{(1)}_{U,u} , (\mathcal{K}/\mathcal{O}) (\alpha  \psi) )^{D_{U,u}}=
\mathrm{Hom} (I^{(0)}_{U,u} /I^{(1)}_{U,u} , (\mathcal{K}/\mathcal{O}) (\alpha  \psi) )^{D_{U,u}}.
$$ 
The action of $D_{U,u}$ on $I^{(0)}_{U,u} /I^{(1)}_{U,u}$ is a conjugate action. 
Since the action of $D_{U,u}$ on $I^{(0)}_{U,u} /I^{(1)}_{U,u}$ is trivial, the $D_{U,u}$-invariant part of this group is finite and bounded 
{}{when $U \in \mathcal{U}$ varies} by 
the assumption (V$_p$) in Definition \ref{s22ax}. We thus prove that,
{}{for $i=0$,} the kernel of the map in \eqref{sucsseccive_extension1} is finite and bounded when 
$U \in \mathcal{U}$ varies . 
\par 
Next, we discuss the case $i>0$. The condition \eqref{equation:conjugation} implies that, {}{ for every $i$ and $j$,} 
$\Hom (I^{(i)}_{U,u}/ I^{(i+1)}_{U,u}, \mathrm{Fil}^{j} _v A /  \mathrm{Fil}^{j+1} _v A )
^{D_{U,u}}$ is finite {}{and bounded when $U \in \mathcal{U}$ varies}. Hence, {}{ for every $i$,} 
$\Hom (I^{(i)}_{U,u}/ I^{(i+1)}_{U,u}, A /F^+_v A )
^{D_{U,u}}$ is finite {}{and bounded when $U \in \mathcal{U}$ varies}. 
Since the kernel of the Artin representation $\rho$ is open in $G_K$, 
this implies that the group 
$$
\mathrm{Hom} (I^{(i)}_{U,u} /I^{(i+1)}_{U,u} , 
A_\rho /F^+_v A_\rho )^{D_{U,u}}.
$$
is finite {}{and bounded when $U \in \mathcal{U}$ varies}. This completes the proof of the control theorem. 
\end{proof}
\begin{rem}
Only the part of the condition $\rm({A}_{\it p})$ of Definition \ref{set2b}, stating the groups $(A/ F^+_v A )^{G_{K_{\infty ,w}}}$ are finite, is used in the proof of control theorem.
\end{rem}
%\begin{rem}\label{error-term-over-p}
 {}{
In \eqref{e1ande2} and \eqref{e1ande23}, 
we defined terms  $E_0^A$ and  $E_1^A$ which contributes to the functional equation of Theorem \ref{function-equation-thm}. 
Later, we will calculate these terms in some specific cases. At the moment, we will show that these two terms represent the same class in $K_0(\mathfrak M_H(G))$ under suitable hypotheses. 
%\par 

%\end{rem}

\begin{proposition}\label{cor-error-term-over-p}
We keep the hypotheses and setting of Theorem \ref{propisition:control_theorem} (Control Theorem). 
Assume further that either $\rm({A}_{\it p})$ of Definition \ref{set2b} or  $ \rm{(Van_{\it p})}$ of Definition  \ref{definition:setting} is satisfied. Then we have $[E_0^A]=[E_1^A]$ in $K_0(\mathfrak M_H(G))$. 

\end{proposition}

\begin{proof}   Recall that $P_U$ is the set of primes $u$ of $K_U$ such that the image of $I_{U,u} \subset G_{K_U} $ via $G_{K_U} \twoheadrightarrow U=\mathrm{Gal} (K_\infty /K_U)$ is infinite by the equation \eqref{error-term-E}. Recall also that we defined the terms $E_0^A$ and $E_1^A$ as follows: 
 $$
 E_0^A: = \underset{U}{\varprojlim} E^{A}_{0,U} = \underset{U}{\varprojlim} \underset{u \in P_U}{\oplus}E^{A}_{U,u}
 \text{ and }E_1^A: = \underset{U}{\varprojlim} E^{A}_{1,U} = \underset{U}{\varprojlim} \underset{u \in P_U, u \nmid p}{\oplus}E^{A}_{U,u}.
 $$ Let us fix a prime $v$ in $K$ such that  $v \mid p$. Let $U$ be an open subgroup of $G$ and $u$ a prime in $K_U$ dividing $v$. 
By the definition of $P_U$, we have $u \in P_U$, $v \in P_G$. %\ref{function-equation-thm} . 
\par 
First, let us assume that $\rm({A}_{\it p})$ is satisfied. In this case, the proof of control theorem shows that $E^{A}_{U,u}$ is finite and uniformly bounded independent of $U$ and $u$. Using the induced module,  we derive $\varprojlim_U E^{A}_{U,u} \cong \La_{\mathcal{O}}(G) \otimes_{\La_{\mathcal{O}}(G_{v})} M^{A}_{v}$, where $M^{A}_{v}$ is a  finite  module  of $p$-power cardinality. If $v$ is finitely decomposed in the extension $K_\infty$ i.e. if  $[G:G_v] < \infty$, then we have $\La_{\mathcal{O}}(G) \otimes_{\La_{\mathcal{O}}(G_{v})} M^{A}_{v}$ is finite and uniformly bounded. Hence $[\varprojlim_U E^{A}_{U,u}]= [\La_{\mathcal{O}}(G) \otimes_{\La_{\mathcal{O}}(G_{v})} M^{A}_{v}] =0$ in $K_0\big(\mathfrak M_H(G)\big)$ by Lemma \ref{h-level-vanishing} (a). On the other hand, if $v$ is infinitely decomposed in $K_\infty$ i.e. if $G_v$ is of infinite index in $G$, 
then, we have $[M^{A}_{v}] =0 $ in $K_0\big(\mathfrak M_{H_{v}}(G_{v})\big)$ by Lemma \ref{h-level-vanishing}(a) where $H_{v} = H\cap G_{v}$. As  %$K_\cyc$ is finitely decomposed at the prime $v$,
 $\La(G) \otimes_{\La(G_{v})} -$ is a well-defined map from $\M_{H_{v}}(G_{v}) \lra \M_H(G)$. Hence  we deduce that $[\La_{\mathcal{O}}(G) \otimes_{\La_{\mathcal{O}}(G_{v})} M^{A}_{v}] =[\varprojlim_U E^{A}_{U,u}]= 0$ in $K_0\big(\mathfrak M_H(G)\big)$. Thus from the definition of $E_0^A, E_1^A$,  we deduce that $[E_0^A]=[E_1^A]$  in $K_0(\mathfrak M_H(G))$. \\
\par 
Next, let us assume that $\rm({A}_{\it p})$ is not satisfied but the assumption $ \rm{(Van_{\it p})}$ is satisfied. Using the same notation as above,  we have $\varprojlim_U E^{A}_{U,u} \cong \La_{\mathcal{O}}(G) \otimes_{\La_{\mathcal{O}}(G_{v})} M^{A}_{v}$, where $M^{A}_{v}$ is now a finitely generated $\Z_p$-module. However by the assumption $ \rm{(Van_{\it p})}$ and Lemma \ref{h-level-vanishing} (a), we have $[M^{A}_{v}] =0 $ in $K_0\big(\mathfrak M_{H_{v}}(G_{v})\big)$. Thus, by an argument similar to the above case, we deduce that $[\varprojlim_U E^{A}_{U,u}] =0$ when (i) $v$ is finitely decomposed in $K_\infty$ and also  when (ii) $v$ is infinitely decomposed in $K_\infty$ thanks to the hypothesis   $ \rm{(Van_{\it p})}$. Hence once again from the definition of $E_0^A, E_1^A$, we deduce that $[E_0^A]=[E_1^A]$  in $K_0(\mathfrak M_H(G))$. This completes the proof of the proposition.
\end{proof}}
%\begin{corollary}\label{cor-error-term-over-p}
%is %\begin{enumerate}
%%\item[(i)]  If the conditions 1(a) and 2(a) of Control Theorem are satisfied then we have  $[E_0^A] =[E_1^A]$ in $K_0(\mathfrak M_H(G))$. 
%%\item[(ii)] If the conditions 1(b) and 2(b) of Control Theorem are satisfied and also $ \rm{(Van_{\it p})}$  of Definition \ref{definition:setting} is satisfied,  then again we have  
%$[E_0^A] =[E_1^A]$ in $K_0(\mathfrak M_H(G))$. 
%%\end{enumerate}
%\end{corollary}

%%%%%%%%%%%%%%%%%%%%%%%%%%%%%%%%%%%%%%%%%%%%%%%%%%%%%%%%
%\begin{example}\label{example:5}

\section{Examples of the error term of the algebraic functional equation}\label{s5}
In this section, we calculate the exceptional divisor $\Bigl[ 
\varprojlim_U E^{A}_{0,U} \Bigr] 
= \Bigl[ \varprojlim_{U} \underset{v \in P_U}{\oplus} E^{A}_{U,u} \Bigr] $ in the setting of Example \ref{example:2}, Example \ref{example:3} and Example \ref{example:4}.
%\begin{enumerate}
%\item 
\par 
{\it Example $1$: }  First, we consider the case of Example \ref{example:2} i.e. $K_\infty = \underset{n \geq 1}{\cup} \mathbb{Q}(\mu_{p^n} ,\sqrt[p^n]{a})$, $K=\Q$ and 
{}{$A= B_{p^\infty}$ where} {}{$B$} is an elliptic curve defined over $\mathbb{Q}$ with good ordinary reduction at $p$. Recall that all primes of $\mathbb{Q}$ are finitely decomposed in 
the false-Tate extension $K_\infty$ as stated in Example \ref{example:2}.  Note that in this case  the condition $\rm({A}_{\it p})$ is satisfied. Hence we only need to calculate  error terms corresponding primes not dividing $p$, as by Proposition \ref{cor-error-term-over-p} we have, $[E_0^A] =[E_1^A]$ in $K_0(\mathfrak M_H(G))$ where $E_1^{A}: = \underset{U}{\varprojlim} E^{A}_{1,U}= \underset{U}{\varprojlim} \underset{u \in P_U, u \nmid p}{\oplus}E^{A}_{U,u}.$ Following notation of \cite{hv}, we define 
\begin{align*}
& P_0 = \{  \text{primes }  q  \text{ in } \Q:  q \mid a  \text{ but } q \nmid p  \},
\\
& P_1 =   \{  q \in P_0 : \text{{}{$B$} has split multiplicative reduction above $q$ 
over $\mathbb{Q}(\mu_{p^\infty})$}   \}, \\ 
& P_2 = 
 \{  q \in P_0:  \text{{}{$B$} has good  reduction at } q \text{ and  {}{$B_{p^\infty}(K_v) \neq 0$}} \}. 
\end{align*}
In addition, for each $U$ and $ i =0, 1,2$, we define $P_{i,U}$ to be the set of primes of $K_U$ dividing $P_i$. 
Note that we have $
P_0 = P_G \setminus \{p\}$ and   
$P_{0,U} =P_U \setminus \{  \text{primes $u$ in } K_U: u \mid p \} $ by definition 
and by \cite[Lemma 3.9]{hv}. 
Recall that we have 
\begin{equation}
\Bigl[ 
\varprojlim_U E^{A}_{1,U} \Bigr] 
= \Bigl[ \varprojlim_{U} \underset{u \in P_U, u \nmid p}{\oplus} E^{A}_{U,u} \Bigr] 
= \Bigl[\underset{v \in P_0}{\oplus} \mathrm{Ind}_{G_v}^G  \varprojlim_{U}  \underset{u \in P_U, u \mid v}{\oplus} E^{A}_{U,u} \Bigr] . 
\end{equation}
Hence, it suffices to calculate $\Bigl[ \varprojlim_{U}   \underset{u \in P_U, u \mid v}{\oplus} E^{A}_{U,u}  \Bigr] $. 
We also define 
\begin{equation*}
P_{00} = \{  q  \in P_0 :  A^{D_{\infty,w}} \not= 0 
\text{ for any prime $w$ of $K_\infty$ over $q$} \}.
\end{equation*}
By the exact sequence \eqref{equation:boundE}, we have 
\begin{equation}\label{JCHdncjkbcC}
E^A_{U,u}=0 
\text{ for each prime $u \in P_{0,U} \setminus P_{00,U} $}. 
\end{equation}
\cite[Proposition 5.1(iii)]{hm2} proves that 
$P_{00,U} =P_{1,U} \cup P_{2,U}$.  
\par 
For a prime $u \in P_{2,U}$, the inertia group $I_{U,u}$ acts trivially on 
{}{$A$}  and the eigenvalues of Frobenius element in 
$D_{U,u}/I_{U,u}$ are non-trivial. 
Hence the first term $H^1 (D_{U,u}/I_{U,u},A^{I_{U,u}})$ of the exact sequence \eqref{equation:boundE} is zero and we have 
 $E^A_{U,u} \cong 
H^1 (U_u , A^{D_{\infty ,w}}) $. Since we have   
{}{$A^{D_{\infty,w}} = B_{p^\infty}(K_{\infty,w}) = B_{p^\infty} $} $($see \cite[Proposition 5.1(i)]{hm2}$)$, 
we obtain $E^A_{U,u} \cong 
H^1 (U_u , A ) $. By a similar argument, the restriction map of Galois cohomology induces an isomorphism $H^1 (U_u , A ) \cong  H^1 (I_{U,u} ,A)^{D_{U,u}}$. Since $I_{U,u}$ acts trivially on $A$, we have 
$H^1 (I_{U,u} ,A)^{D_{U,u}} \cong \Hom ( I_{U,u} ,A)^{D_{U,u}}$. 
Thus we obtained 
$$
E^A_{U,u} 
 \cong \Hom ( I_{U,u} ,A)^{D_{U,u}}.
$$  
Note that, the conjugate action of $D_{U,u}$ on $I_{U,u}$ is given by 
the $p$-adic cyclotomic character. Hence we have 
{}{
\begin{equation}\label{equation:calculationP_2}
\underset{q \in P_2}{\bigoplus}  \varprojlim_{U} \bigoplus_{\substack{u\in P_{2,U} \\ 
u \vert q }} E^{A}_{U,u} \cong \underset{q \in P_2}{\bigoplus} \varprojlim_{U}\bigoplus_{\substack{u\in P_{2,U} \\ 
u \vert q }}  B_{p^\infty} (-1)^{D_{U,u}}  \cong \underset{q \in P_2}{\bigoplus} \mathrm{Ind}^{G}_{G_q} 
\mathrm T_p B(-1 ) .
\end{equation}   }
\par 
For a prime $q \in P_{1,U}$, we have  {}{$A^{D_{\infty,w}} = B_{p^\infty}(K_{\infty,w}) =B_{p^\infty}$} by a similar argument as the previous case. Since $D_{\infty,w}$ has no non-trivial pro-$p$ quotient, 
the third term of \eqref{equation:boundE2} is zero. 
Hence \eqref{equation:boundE} gives the following exact sequence: 
\begin{equation}\label{equation:boundE3}
0 \longrightarrow H^1 (D_{U,u}/I_{U,u},H^0 ( {I_{U,u}} ,A) ) 
\longrightarrow H^1 (D_{U,u} , A ) 
\longrightarrow E^{A}_{U,u} 
\longrightarrow 0 . 
\end{equation}
Since our representation {}{$A= B_{p^\infty}$} satisfies 
the condition $(V_q)$  of  Definition \ref{s22ax} for any prime $q \neq p$, 
$H^0 (D_{U,u} , A)$ and $H^2 (D_{U,u} , A)$ are finite. 
%Further, $H^2 (D_{U, u} , A )$ must be $p$-divisible since the cohomological dimension of $D_{U_u}$ is $2$. Thus we see that  $H^2 (D_{U, u} , A ) =0$. 
Thus we have 
\begin{equation}\label{equation:boundE4}
H^1 (D_{U, u} , A ) \cong H^1 (D_{U,u} , T )^\vee 
\cong H^0 (D_{U,u} , A)^\vee 
\end{equation}
Putting these together, we can calculate the inverse limit of the middle term of \eqref{equation:boundE3}, 
{}{
\begin{equation}\label{equation:calculationP_1}
 \underset{q \in P_1}{\bigoplus}  \varprojlim_{U}\bigoplus_{\substack{u\in P_{1,U} \\ 
u \vert q }} H^0 (D_{U,u} , A)^\vee 
 \cong \underset{q \in P_1}{\bigoplus} \varprojlim_{U}\bigoplus_{\substack{u\in P_{1,U} \\
u \vert q }} B_{p^\infty}(K_{U,u})^\vee 
 \cong \underset{q \in P_1}{\bigoplus} \mathrm{Ind}^{G}_{G_q}  T_pB(-1) .
\end{equation}   
}
On the other hand, we calculate the projective limit of the left most term of \eqref{equation:boundE3} 
with respect to the corestriction maps when $U$ is changing:
$$
 \underset{q \in P_1}{\bigoplus} \varprojlim_{U} \bigoplus_{\substack{u\in P_{1,U} \\ 
u \vert q }}H^
1 (D_{U,u}/I_{U,u}, A^ {I_{U,u}} ) .
$$
By assumption, {}{$B$} has split multiplicative reduction above $q$ 
over $\mathbb{Q}(\mu_{p^\infty})$. Hence, for any sufficiently small $U$, {}{$B$} has split multiplicative reduction above $q$ 
over $K_U$. Since we calculate the projective limit of some modules related to $U$, we may and we do assume without loss of generality, that  
{}{$B$} has split multiplicative reduction above $q$ over $K_U$ for any $U$ below. 
Since $q \in P_1$, the action of $G_q$ on $A$ has a non-split filtration 
$$ 
0\lra F_q^+A \lra A \lra A/{F_q^+A} \lra 0
$$
where $F_q^+ A $ is cofree $\mathbb{Z}_p$-module of corank one on which 
$G_q$ acts by the $p$-adic cyclotomic character $\chi_{cyc}$. 
Taking the $I_{U,u}$-invariant of the above short exact sequence, we obtain 
\begin{equation}\label{gyiuwehdknd}
 0\lra F_q^+A \lra A^{I_{U,u}} \lra F_{U,u} \lra 0, 
\end{equation} where $F_{U,u}$  is finite. 
Since $F_q^+A =\Q_p/{\Z_p}(1)$  as $G_q$-module, we have 
$H^1(D_{U,u}/I_{U,u}, F_q^+A)=0$. 
We also have $H^2(D_{U,u}/I_{U,u}, F_q^+A)=0$ since $D_{U,u}/I_{U,u}$ is infinite procyclic. 
From equation \eqref{gyiuwehdknd}, we obtain 
$$
H^1(D_{U,u}/I_{U,u}, A^{I_{U,u}})\cong  H^1(D_{U,u}/I_{U,u}, F_{U,u})\cong \mathrm{Hom}(
D_{U,u}/I_{U,u}, F_{U,u}). 
$$ 
The last isomorphism is true as $D_{U,u}/I_{U,u}$ acts trivially on {}{$A/{F_q^+A}=\tilde{B}_{p^\infty}$}. 
Recall that $A$ is isomorphic to $\left(\overline{\mathbb{Q}}^\times_q /q^{\mathbb{Z}}\right)[p^\infty ]$ and 
{}{$B_{p^n}$} fits into the exact sequence 
{}{
$$
0 \longrightarrow \mu_{p^n}  \longrightarrow B _{p^n}  \longrightarrow \mathbb{Z} /p^n \mathbb{Z}  \longrightarrow 0.
$$
}
When $K_{U,u}$ is equal to $k_n=\mathbb{Q}_q (\mu_{p^n},a^{1/{p^n}})$ we denote $F_{U,u}$ by $F_n$, which is a subgroup of the quotient $\mathbb{Z} /p^n \mathbb{Z}$. There exists a constant $\delta$ such that the order of $F_n$ is $p^{n -\delta} $ when $n$ is sufficiently large. 
The degree of the extension $k_{n+1}/k_n$ is $p^2$ and the norm map for the extension $k_{n+1}/k_n$ maps $F_{n+1}$ to a subgroup of order 
$p^2 \# F_{n+1}$ in $F_n$. From this observation, we see that see that  \begin{equation}\label{jzbcdda;lpkkjhh2}
 \underset{q \in P_1}{\bigoplus} \varprojlim_{U}\bigoplus_{\substack{u\in P_{1,U} \\ 
u \vert q }} H^1 (D_{U,u}/I_{U,u}, A^ {I_{U,u}} ) \cong  \underset{q \in P_1}{\bigoplus}  \varprojlim_{U} \bigoplus_{\substack{u\in P_{1,U} \\ 
u \vert q }}F_{U,u} \cong  0 .
\end{equation}
Using \eqref{jzbcdda;lpkkjhh2} in \eqref{equation:boundE3}, the  contribution  of the split multiplicative primes  to the error term is given by
{}{
\begin{equation}\label{ekjjefnbhjgyf}
\underset{q \in P_1}{\bigoplus} \varprojlim_{U} \bigoplus_{\substack{u\in P_{1,U} \\ 
u \vert q }} E^{A}_{U,u} \cong  \underset{q \in P_1}{\bigoplus} \mathrm{Ind}^{G}_{G_q}  T_p B(-1)
\end{equation}}
\par 
By \eqref{JCHdncjkbcC}, \eqref{equation:calculationP_2} and 
\eqref{equation:calculationP_1}, we have the following description of error terms in  $K_0(\mathfrak M_H(G)):$
{}{
$$
[E_0^A] =
[E_1^A] \cong  
\Bigl[ 
\underset{q \in P_1\cup P_2}{\bigoplus} \mathrm{Ind}^{G}_{G_q}   T_pB(-1)\Bigr] .
$$
}
\par 
%\item 
{}{In this case of Example \ref{example:2}(1), our error term is exactly the same error term as that of \cite[Theorem 6.2, Equation (6.29)]{zab}. Note in this case, $B$ has good ordinary reduction at $p$.  Then, using Coates-Greenberg's theory of deeply ramified extension \cite[Proposition 4.3 and  Proposition 4.8]{cg}, it follows that 
%as explained in \cite[\S2, Page 448]{hv}, 
$\mathrm{Sel}^{\mathrm{Gr}}_{A } (K_\infty )  =S_{p^\infty}(B/K_\infty)$, where $S_{p^\infty}(B/K_\infty)$ is the classical Selmer group of $B$ over $K_\infty$.  Thus our result is consistent with that of \cite{zab}.} 

\medskip

\par {\it Example $2$: } Next, let us consider the case of Example \ref{example:3} (O) i.e. $K_\infty = \underset{n \geq 1}{\cup} \mathbb{Q}(E_{p^n})$ where 
$E$ is a non-CM elliptic curve over $\mathbb{Q}$ and $V =T_p B \otimes \mathbb{Q}_p$
is the Galois representation associated to an elliptic curve $B$. {}{We assume that $B$ is isogenous to $E$}. 
As in the previous example, we define 
\begin{align*}
& P_0 = \{  \text{primes }  q :  q \neq p \text{ and the $q$-adic valuation of {}{$j(B)$} is negative}  \},
\\
& P_1 =   \{ q \in P_0 : \text{{}{$B$}  has split multiplicative reduction  above $q$ 
over $\mathbb{Q}(\mu_{p^\infty})$}\}.
\end{align*} 
The primes in $P_0$ are precisely those prime where {}{$B$} has potentially multiplicative reduction. Thus $P_0$ is nothing but primes not divisible by $p$ at which the inertia subgroup of $\mathrm{Gal}(K_\infty /\mathbb{Q})$ is 
infinite. Note that there is no analogue of $P_2$ of the false-Tate case of Example 1 in this case. %and in fact,   
%For simplicity we assume that, for all primes $q$ where the $q$-adic valuation of $j(E)$ is negative, $E$ as a split multiplicative reduction. 
 %the set $P_0$ coincides with $P_1$.   
\par 
We note that primes in $P_1$ are not finitely decomposed in this example. However, the arguments and conclusions of the false-Tate case above for a prime of $q$ in $P_1$ remain valid even if $q$ is infinitely decomposed. Thus, following the same argument as in the false-Tate case above, we deduce the following description of error terms in  $K_0(\mathfrak M_H(G)):$
{}{
$$%[E_0^A] =
[E_1^A] \cong \Bigl[ 
\underset{q \in P_1}{\bigoplus} \mathrm{Ind}^{G}_{G_q}  T_p B(-1) \Bigr]. 
$$
}
%\cong \left[ \underset{v \in P_1}{\bigoplus} \mathrm{Ind}^{G}_{G_v}  \Z_p\right] .$$
 Note by Example \ref{rmk33}(5), the condition $(\mathrm{Van}_p)$ of Definition \ref{definition:setting} is satisfied in this case. Hence by Proposition \ref{cor-error-term-over-p}, the error term is given by 
 {}{
$$[E_0^A] =
[E_1^A] \cong \Bigl[ 
\underset{q \in P_1}{\bigoplus} \mathrm{Ind}^{G}_{G_q}  T_p B(-1) \Bigr]. 
$$
}
{}{In this case of Example \ref{example:3}(O), our error term  above is the same error term as that of \cite[Theorem 5.2, Equation (5.23)]{z2}. Since {}{$B$} has good, ordinary reduction at $p$, again using \cite[Proposition 4.3 and Proposition 4.8]{cg}, 
%Coates-Greenberg's theory of deeply ramified extension (see  \cite[Equation (60) and Proposition 5.15]{ch}), 
it follows that  $\mathrm{Sel}^{\mathrm{Gr}}_{A } (K_\infty )$   is the same as the  classical Selmer group $S_{p^\infty}(B/K_\infty)$.  Hence, our result is consistent with that of \cite{z2}.} 

\medskip

\par {\it Example $3$: } Let $f$ be a normalized eigen elliptic cuspform of even weight $k\geq 2$ and level $\Gamma_0 (N)$ such that  $N$ is square-free and the conductor $N_f$ of $f$ is not divisible by $p$. Let us assume that $a_p(f)$ is a $p$-adic unit.  {}{We define $V$ to be the Tate-twist 
$V_f(\frac{k}{2})$ and} we take a lattice $T \subset V$. We set $A =T\otimes \mathbb{Q}_p /\mathbb{Z}_p$. 
Let $K_\infty/\Q$ be the false-Tate curve extension as in Example 1 above. In this case, we define 
\begin{align*}
& P_0 =\{q \text{ prime in } \Q: q \neq p, \ q \mid a \}, \\ 
& P_{00} = \{  q  \in P_0 :  A^{D_{\infty,w}} \not= 0 
\text{ for any prime $w$ of $K_\infty$ over $q$} \}.
\end{align*}
By the same reason as \eqref{JCHdncjkbcC}, we have 
\begin{equation}\label{JCHdncjkbcC3}
E^A_{U,u}=0 
\text{ for each prime $u \in P_{0,U} \setminus P_{00,U} $}. 
\end{equation}
In order to give the analogues of $P_1$ and $P_2$ of Example 1, 
we prepare some notation. Let us denote by $\pi_{f,q}$ the local automorphic representation at the prime $q$ associated to $f$. Since we assume that 
the Nebentypus character of $f$ is trivial and the conductor $N_f$ of $f$ is square-free, 
when  prime $q$ divides  $N_f$, $\pi_{f,q}$ is a special representation and we have $\pi_{f,q} \cong \pi 
(\delta \vert \ \vert^{\frac{k}{2}-1} , \delta \vert \ \vert^{\frac{k}{2}})$ where 
$\delta$ is a unramified quadratic character of $\Q^\times_q$ and $\vert \ \vert$ 
is the valuation on $\Q_q$. Now, we define 
\begin{align*}
& P_1 =   \{  q \in P_0 : q \vert N_f  \text{  and the quadratic character $\delta$ associated to $\pi_{f,q}$ is trivial on $\Q_q (\mu_p)^\times$}   \}, \\ 
& P_2 = 
 \{  q \in P_0: q\nmid N_f  
  \text{  and $A^{D_{\infty,w}} \not= 0 $ for any prime $w$ of $K_\infty$ over $q$ }  \}. 
\end{align*}
By definition, for any sufficiently small $U$, $\delta$ is  trivial  regarded as a Galois character of $D_{U,u}$. 
 Since we calculate the projective limit of some modules related to $U$, we may and we do assume without loss of generality, that  
$\delta$ is  trivial  regarded as a Galois character of $D_{U,u}$ for any $U$ below. 
For $q \in P_1$, the action of any open subgroup $D_{U,u}$ of $G_q$ on $A$ has a non-split filtration 
$$ 
0\lra F_q^+A \lra A \lra A/{F_q^+A} \lra 0
$$
where $F_q^+ A $ (resp.  $A/F_q^+A$) is a cofree $\mathbb{Z}_p$-module of corank one on which 
$D_{U,u}$ acts by the $p$-adic cyclotomic character $\chi^{\frac{k}{2} }_{\mathrm{cyc}}$ (resp. $\chi^{\frac{k}{2} -1}_{\mathrm{cyc}}$ ). 
Hence, we have 
$P_{00,U} =P_{1,U} \cup P_{2,U}$. 
By a similar argument as in Example 1, we prove the same type of results as 
\eqref{equation:calculationP_2} and 
\eqref{equation:calculationP_1}. 
Hence, the error term is given by 
{}{
$$
[E_0^{A}]= [E_1^{A}]  =[\underset{q\in P_{1} \cup P_{2}}{\bigoplus} \text{Ind}_{G_q}^G T (-1)] \text{ in } K_0(\mathfrak M_H(G)).
$$
}
%, the action of $G_{K_v}$ on $F^+_v A$ and $A/F^+_v A$ is explicitly known by \cite[Thm 2.1.4]{wiles}. 
%By a similar argument as in (1), we can verify the  the condition {\rm  $(\mathrm{A}_p)$} as well as the condition {\rm (A)} 
%if we assume that the conductor of $f$ is not divisible by $p$. 
%
%Let $f \in S_k(\Gamma_0(N))$ be a newform where $k \geq 2$, , $p\nmid N$. 

\par {\it Example $4$: }Let us discuss  commutative examples. Let us keep the same setting as Example \ref{example:4} where $G$ is isomorphic to $\mathbb{Z}^{d+1}_p$ and 
$A$ is isomorphic to $E_{p^\infty}$ with $E$ an elliptic curve with good ordinary reduction at all primes of $K$ above $p$. 
Now we have two cases regrading the decomposition subgroup $G_v$ of $G$ at a prime $v$ in $K$ dividing $p$. Note as $K_\cyc \subset K_\infty$, we necessarily have dimension of $G_v$, as a $p$-adic Lie group, is at least $1$.

In the first case, if $K_v$ is equal to $\mathbb{Q}_p$, then for a prime $w$ in $K_\infty$ over $w$, $E_{p^\infty}(K_{\infty,w})= E_{p^\infty}(K_{\cyc,w}) \subset E_{p^\infty}(K_{v}(\mu_{p^\infty}))$ is finite by Imai's theorem \cite{i} (Note that Imai's theorem can be applied since we assume that $E$ has good reduction at $v$). Thus in this case condition (A$_p$) is satisfied.
 
On the other hand, if $K_v$ is a non-trivial extension of $\mathbb{Q}_p$, then the dimension of $H_v:= H \cap G_v$ is at least $1$. Thus $\Z_p[[G_v]] \cong \Z_p[[T_1, T_2, ..., T_r]]$ for some $r$ with $r \geq 2$. Thus $\Z_p$ is a pseudo-null $\Z_p[[G_v]]$-module and consequently, $\mathrm{(Van}_p)$ hypothesis is satisfied. 

Hence, by Proposition \ref{cor-error-term-over-p}, we have $[E_0^A] =[E_1^A]$ in $K_0(\mathfrak M_H(G))$ 
no matter how $K_v$ is equal to $\mathbb{Q}_p$ or not. 
Also, since $\Z_p^{d+1}$ extension of a number field is unramified outside $p$, we have $E^A_{U,u}=0$ for sufficiently small $U$ any $u \in K_U$ with $u \nmid p$. It follows that $E_1^A =0.$ Thus the functional equation takes the expected simple form $$[\mathrm{Sel}^{\mathrm{Gr}}_{A}(K_\infty)^\vee ] = [\big(\mathrm{Sel}^{\mathrm{Gr}}_{A}(K_\infty)^\vee\big)^\iota] \quad \text{ in } K_0 (\mathfrak{M}_H (G)).$$

\section{Higher extension groups of Selmer groups}\label{sext}
In this section, we study the higher extension groups 
$\mathrm{Ext}^i_{\Lambda (G)} ( \mathrm{Sel}^{\mathrm{Gr}}_A (K_\infty )^\vee  , \Lambda (G) )$, which appear in the proof of our functional equation given in the next section. Recall $G $ has a closed normal subgroup $H$ such that $G/H \cong \Gamma \cong \Z_p$. We prepare a few lemmas and a proposition in the beginning which is used later.

%%%%%%%%%%%%%%%%%%%%%%%%%%%%%%%%%%
\iffalse{
\begin{lemma}\label{finite0}
Let $G$ be a compact $p$-adic Lie group G without any element of order $p$  such that $G$ has a quotient $\Gamma$ isomorphic to $\mathbb{Z}_p$. 
Then, we have $[M] =0 $ in $K_0(\M_H(G))$ for any $\La_{\mathcal{O}}(G)$-module $M$ which is of finite cardinality. 
\end{lemma}
\begin{proof} 
Note that any $\La_{\mathcal{O}}(G)$-module of finite cardinality is a 
$\La (G)$-module of finite cardinality. Hence it suffices to 
show the lemma when $O=\mathbb{Z}_p$. 
Since any $\La (G)$-module of finite cardinality is 
successive extension of $\La (G)$-modules isomorphic to $\F_p$, 
we reduce the proof to the case $M= \F_p$. Notice that as $G$ has the quotient $\Gamma$, we can prove $[\F_p] =0$ in $K_0(\Omega(G))$ (see for example \cite[Lemma 2.2]{bz}). Further  \cite[\S 5.5]{aw2} shows 
that $K_0(\Omega(G)) \cong K_0(\mathcal D)$, where $\mathcal D$ denote the category of all finitely generated $p$-torsion $\La (G)$-modules. The result follows by considering the  natural homomorphism (see  \cite[\S 1.4]{aw2}) from $K_0(\mathcal D) \lra K_0(\M_H(G))$. 
\end{proof}}
\fi
\begin{lemma}\label{fp-0-in-sl2zp}
For any  $X \in \mathrm{SL}_2(\Z_p)$, its centralizer 
$$
C_{\mathrm{SL}_2(\Z_p)}(X) = \{Y \in \mathrm{SL}_2(\Z_p)| YXY^{-1} =X\}
$$ is infinite.  Consequently, the dimension of $C_{\mathrm{SL}_2(\Z_p)}(X)$ as $p$-adic Lie group is $\geq 1$.

\end{lemma}
\begin{proof} Take any matrix $X$ in $\mathrm{SL}_2(\Z_p)$. For
a diagonal matrix  \begin{small}{$X= \begin{pmatrix} a  & 0 \\ 0  & a^{-1}  \end{pmatrix}\in \mathrm{SL}_2(\Z_p)$}\end{small}, we have  
$\begin{small}{\left\{ \begin{pmatrix} \lambda  & 0 \\ 0  & \lambda^{-1}  \end{pmatrix} \middle| \lambda \in \Z^\times_p \right\}}\end{small} \subset C_{\mathrm{SL}_2(\Z_p)}(X)$, which shows that $C_{\mathrm{SL}_2(\Z_p)}(X)$ is infinite 
in this case. 
Next, we consider the case where $X = \begin{pmatrix} a  & b \\ c  & d  \end{pmatrix} \in \mathrm{SL}_2(\Z_p)$ is not a diagonal matrix, that is, 
at least one of $b $ or $c$ is nonzero. Then for each $n$ consider 
\begin{small}{$$
X_n: = I_{2\times 2} +p^nX = \begin{pmatrix} 1+p^na  & p^nb \\ p^nc  & 1+p^nd  \end{pmatrix}.$$}\end{small} Then $\mathrm{det}(X_n) \in 1+p\Z_p$ and  set $Y_n= \dfrac{1}{\mathrm{det}(X_n)}X^2_n \in \mathrm{SL}_2(\Z_p)$. 
For all $n$, we have $Y_nX =XY_n$ and we have the presentation 
\begin{small}{$$Y_n = \begin{pmatrix} \dfrac{(1+p^na)^2 +p^{2n}bc}{\mathrm{det}(X_n)} & \dfrac{p^nb(2+p^na+p^nd)}{\mathrm{det}(X_n)} \\ \dfrac{p^nc(2+p^na+p^nd)}{\mathrm{det}(X_n)}  & \dfrac{(1+p^nd)^2 + p^{2n}bc}{\mathrm{det}(X_n)}  \end{pmatrix}.
$$}\end{small} 
Since the $p$-adic valuation of at least one of the non-diagonal entries of $Y_n$ goes to infinity as $n$ goes to infinity, the set $\{Y_n\}_{n \in \N}$ is an infinite subset of $\mathrm{SL}_2(\Z_p)$. This completes the proof.  
\end{proof}

Let $\varpi$ be a uniformizer of $\mathcal{O}$. 
For any $p$-adic Lie group $G$, we define $\Omega (G):=  
\underset{U}{\varprojlim}~ 
\mathcal{O} /(\varpi)  [G/U]$ to be the completed group ring where $U$ varies over open normal subgroups of $G$. We denote by $K_0 (\Omega(G))$ the Grothendieck group of the ring $\Omega(G )$.

%%%%
\begin{proposition}\label{h-level-vanishing}
(a) Let $G$ be a compact $p$-adic Lie group G without any element of order $p$  such that $G$ has a quotient $\Gamma$ isomorphic to $\mathbb{Z}_p$.  Let  $M$  be a $\La_{\mathcal{O}}(G)$-module which is of finite cardinality. 
Then, $[M] =0 $ in $K_0(\M_H(G))$.
\par (b) Let $H$ be a $p$-adic Lie group without any element of order $p$. Assume that either one of the following two conditions hold: 
 \begin{enumerate}
 \item[\rm{(i)}] The centralizer of every element in $H$ is infinite. 
 \item[\rm{(ii)}] The group $H$ is pro-$p$.
 \end{enumerate}
Then we have $[\mathcal{O} /(\varpi) ] =0 $ in $K_0(\Omega(H))$. 
\end{proposition}
{\it Proof of Part (a)}: 
This statement is proved in \cite[Lemma 2.2]{bz}. 
However, we will give a detailed proof. 
%Note that any $\La_{\mathcal{O}}(G)$-module of finite cardinality is a 
%$\La (G)$-module of finite cardinality. Hence it suffices to 
%show the lemma when $\mathcal{O}=\mathbb{Z}_p$. 
Since any $\La_{\mathcal{O}} (G)$-module of finite cardinality is isomorphic to 
a successive extension of the $\La_{\mathcal{O}} (G)$-module $\mathcal{O}/(\varpi )$, 
we reduce the proof to the case $M= \mathcal{O}/(\varpi )$. As $G$ has no element of order $p$, $\Omega(G)$ has finite global dimension \cite[Proposition(e) \S 3.3]{aw2}. It follows that every finitely generated $\Omega(G)$-module can be represented in the Grothendieck group $K_0(\Omega(G))$. In particular, $[\mathcal{O}/(\varpi )[[\Gamma]]] \in K_0(\Omega(G))$. Thus $[\mathcal{O}/(\varpi )] =0$ in $K_0(\Omega(G))$. Further,  \cite[\S 5.5]{aw2} shows 
that $K_0(\Omega(G)) \cong K_0(\mathcal D)$, where $\mathcal D$ denote the category of all finitely generated $p$-torsion $\La (G)$-modules. The result follows by considering the  natural homomorphism (see  \cite[\S 1.4]{aw2}) from $K_0(\mathcal D) \lra K_0(\M_H(G))$.

\begin{proof}[Proof of part (b)] First, we discuss the case with the condition (i). 
Using works of Serre \cite{se}, it is shown in \cite[\S 1.2 \& \S1.3, Page 32]{aw} that   $[\mathcal{O}/(\varpi )] =0 $ in $K_0(\Omega(H))$ if  the dimension (as a $p$-adic Lie group) of the centralizer of every element of $H$ is at least one.  Thus we deduce $[\mathcal{O}/(\varpi ) ] =0 $ in $K_0(\Omega(H))$ in the first case. 
\par 
Next, we discuss the case with the condition (ii). 
If $H$ is also pro-$p$, then Grothendieck group $K_0(\Omega(H))$ is isomorphic to $\Z$ and the class of any finitely generated torsion $\Omega(H)$-module in $K_0(\Omega(H))$ is zero. In particular, we obtain $[\mathcal{O}/(\varpi ) ] = 0 $. 
\end{proof}

By using Lemma \ref{fp-0-in-sl2zp} and Proposition \ref{h-level-vanishing}, 
we immediately obtain the following corollary. 
\begin{corollary}\label{fp-0-criterion}
Let $H = \mathrm{SL}_2(\Z_p)$. Then we have $[\mathcal{O}/(\varpi ) ] =0$ in $K_0(\Omega(H)).$  
\end{corollary}

\begin{rem}
 In  \cite[Example 9.6]{aw2}, they considered a pro-Dihedral group $H = \Z_p \rtimes \Z/{2\Z} $ and showed that  $[\mathcal{O}/(\varpi ) ] \neq 0 $ in $K_0(\Omega(H))$.
\end{rem}

\begin{lemma}\label{jannsen-lemma} \cite[Lemma 2.4]{ja}
Let $G$ be a $p$-adic Lie group and $U$ be an open subgroup of $G$. Then the restriction map induces a functorial isomorphism of $\La_{\mathcal{O}}(U)$-modules, $$\mathrm{Ext}_{\La_{\mathcal{O}}(G)}^r(M, \La_{\mathcal{O}}(G)) \cong  \mathrm{Ext}_{\La_{\mathcal{O}}(U)}^r(M, \La_{\mathcal{O}}(U)) $$ for every $\La_{\mathcal{O}}(G)$-module $M$. 
\end{lemma}

Let $ P_{H \cap U }$ be the set of primes in $K_UK_\cyc = K_\infty^{H\cap U}$ lying above $P_U$ and let $D_{H \cap U ,\tilde{u}}$ (resp. $I_{H \cap U ,\tilde{u}} $) denote the decomposition (resp. inertia subgroup) of $H\cap U =\mathrm{Gal}(K_\infty/K_UK_\cyc)$ at $\tilde{u}$ where $\tilde{u}$ is a prime in $K_UK_\cyc$ over $u$. Set
\begin{equation}\label{cyclotomic-level-error-term}
E^{A^\ast(1)}_{H \cap U ,
\tilde{u}} = 
\begin{cases} \mathrm{Ker} \left[ 
H^1 (I_{H \cap U ,\tilde{u}} ,A^\ast(1) )^{D_{H \cap U ,\tilde{u} }} \longrightarrow H^1 (I_{\infty ,w} ,A^\ast(1))^{D_{\infty,w }} \right] & \text{if }  u \nmid p.  \\
\mathrm{Ker} \left[ 
H^1 (I_{H \cap U ,\tilde{u}} ,\frac{A^\ast(1)}{F^+_v A^\ast(1)} )^{D_{H \cap U ,\tilde{u} }} \longrightarrow H^1 (I_{\infty ,w} ,\frac{A^\ast(1)}{F^+_v A^\ast(1)})^{D_{\infty,w }} \right] & \text{if }  u \mid p.
\end{cases}
\end{equation}
\par 
Recall that, for a finitely generated $\mathcal O[[\Gamma_U]]$ 
module  $M$, we denote the maximal $\mathcal O[[\Gamma_U]]$ pseudonull (i.e. finite) submodule of $M$ by $M_\mathrm{null}$.  Then we have the following lemma.
\begin{lemma}\label{ext2-van}
 
 Let $U$ vary over the open normal subgroups $G$, then 
 \begin{small}{$$\Big[\underset{U}{\varprojlim}\Big( \bigoplus_{\tilde{u}\in P_{H \cap U }} (E^{A^\ast(1)}_{H \cap U ,\tilde{u}})^\vee\Big)_\mathrm{null}\Big]  = 0  \text{ in } K_0(\mathfrak M_H(G)),$$}\end{small} if either the condition $(A_p)$ or the condition $(\mathrm{Van}_p)$ is satisfied. 
\end{lemma}
\begin{proof} 
First, we discuss $E^{A^\ast(1)}_{H \cap U ,\tilde{u}}$ for $\tilde{u} \nmid p$. %In this case of $u \nmid p$, for our goal, it suffices to show that the groups $E^{A^\ast(1)}_{H \cap U , \tilde{u}}$ are divisible. 
We study  the following diagram. 
\begin{equation}\label{equation:CDoutsidepH}
\begin{CD}
H^1 (D_{H \cap U ,\tilde u} ,A^\ast(1)) @>>> H^1 (D_{\infty ,w} ,A^\ast(1)  ) \\ 
@VVV @VVV \\ 
H^1 (I_{H \cap U , \tilde u} ,A^\ast(1)  )^{D_{H \cap U,\tilde u }} @>>> H^1 (I_{\infty ,w} ,A^\ast(1)  )^{D_{\infty,w }} . \\ 
\end{CD}
\end{equation}
By the same reason as that we gave on the diagram \eqref{equation:CDoutsidep}, 
the left vertical homomorphism of the diagram are surjective and 
the right vertical homomorphism of the diagram is isomorphic. 
Since the group $E^{A^\ast(1)}_{H \cap U ,
\tilde{u}}$ is the kernel of the lower horizontal homomorphism of the 
diagram, we have a surjection map 
\begin{equation}\label{equation:l-part_surjection_to_E}
\mathrm{Ker} [H^1 (D_{H \cap U ,\tilde u} ,A^\ast(1)) \longrightarrow  H^1 (D_{\infty ,w} ,A^\ast(1)  )]
\twoheadrightarrow 
E^{A^\ast(1)}_{H \cap U ,
\tilde{u}}, 
\end{equation}
applying the snake lemma to the diagram 
\eqref{equation:CDoutsidepH}. 
By the Inflation-Restriction sequence, 
we have 
\begin{small}{\begin{equation}\label{equation:E_l-part}
H^1 (D_{H \cap U ,\tilde u}/ D_{\infty ,w}, (A^\ast(1))^{G_{K_{\infty ,w}}})
%\\ 
\cong 
\mathrm{Ker} [H^1 (D_{H \cap U ,\tilde u} ,A^\ast(1)) \longrightarrow  H^1 (D_{\infty ,w} ,A^\ast(1)  )]. 
\end{equation}}\end{small}
By the fact that $v$ is not dividing $p$ and by the hypothesis that 
$G$ has no element of order $p$, for any prime $\tilde u$ of $K_UK_\cyc$ above $v$, the group $D_{H \cap U ,\tilde u}/ D_{\infty ,w}$ is either a finite group of order prime to $p$ or 
isomorphic to a product of $\mathbb{Z}_p$ and a finite group of order prime to $p$. In the former case, the group \eqref{equation:E_l-part} is trivial and there is nothing to discuss.  
In the latter case, the largest finite quotient of the group \eqref{equation:E_l-part} is a quotient of a finite group 
$(A^\ast(1))^{G_{K_{\infty ,w}}}
/\bigl((A^\ast(1))^{G_{K_{\infty ,w}}} \bigr)_{\mathrm{div}}$. By \eqref{equation:E_l-part}, $(E^{A^\ast(1)}_{H \cap U ,\tilde{u}})^\vee_\mathrm{null}$ is a finite group for 
any prime $v$ of $K$ not dividing $p$ and for any prime $\tilde u$ of $K_UK_\cyc$ above $v$. 
Let $v$ be a prime of $K$ not dividing $p$. Then, 
$(E^{A^\ast(1)}_{H \cap U ,\tilde{u}})^\vee_\mathrm{null}$ and $(E^{A^\ast(1)}_{H \cap U ,{\tilde{u}}'})^\vee_\mathrm{null}$ are isomorphic to 
each other for two different primes $\tilde u$ and ${\tilde u}'$ of $K_UK_\cyc$ above $v$. 
We conclude that 
\begin{equation}\label{080921091}
(E^{A^\ast(1)}_{H \cap U ,\tilde{u}})^\vee_\mathrm{null} \text{ is finite and uniformly bounded independently of $U$ and $\tilde{u}$}.
 \end{equation}
Thus, we get the following isomorphism 
by using induced module and by \eqref{080921091} 
\begin{small}{$$
\underset{U}{\varprojlim}\Big(\bigoplus_{\tilde{u}\in P_{H \cap U }, \tilde{u} \nmid p} (E^{A^\ast(1)}_{H \cap U ,\tilde{u}})^\vee\Big)_\mathrm{null} \cong \displaystyle{\bigoplus_v} \La_{\mathcal O}(G) \otimes _{\La_{\mathcal O}(G_v)} M_v
$$}\end{small}
where $v$ runs through primes of $K$ such that the inertia subgroup of $G_v$ is infinite and $M_v$ a  $\La_{\mathcal O}(G_v)$-module of finite cardinality. 
\par Recall the cyclotomic $\Z_p$ extension $K_\cyc$ is finitely decomposed at every prime of $K$; hence $\La(G) \otimes_{\La(G_{v})} -$ is a well-defined map from $\M_{H_{v}}(G_{v}) \lra \M_H(G)$ for $v \in P_G$, $v \nmid p$. Then applying  Proposition \ref{h-level-vanishing}(a), we obtain
\begin{small}{\begin{equation}\label{080921092}
\Big[\underset{U}{\varprojlim}\Big(\bigoplus_{\tilde{u}\in P_{H \cap U }, \tilde{u} \nmid p} (E^{A^\ast(1)}_{H \cap U ,\tilde{u}})^\vee\Big)_\mathrm{null}\Big]  = [\La_{\mathcal O}(G) \otimes _{\La_{\mathcal O}(G_v)} M_v]=0 \text{ in }K_0(\mathfrak M_H(G)) 
\end{equation}}\end{small}
 
\par 
 Next, we discuss $E^{A^\ast(1)}_{H \cap U ,
\tilde{u}}$ for $\tilde{u} \mid p$. Recall that we have 
\begin{equation}\label{sep1901}
E^{A^\ast(1)}_{H \cap U ,
\tilde{u}} \cong H^1 \big(I_{H \cap U, \tilde u}/{I_{\infty, w}}  ,(A^\ast(1)/{F^+_v A^\ast(1)})^{I_{\infty,w}}  \big)^{D_{H \cap U,\tilde u }}
\end{equation} 
by definition. Let us consider the following diagram. 
\begin{equation}\label{cyclotomic-error-estimate-over-p}
\begin{CD}
H^1 \big(D_{H \cap U ,\tilde u} ,A^\ast(1)/{F^+_v A^\ast(1)}\big) @>>> H^1 \big(D_{\infty ,w} , A^\ast(1)/{F^+_v A^\ast(1)}  \big) \\ 
@VVV @VVV \\ 
H^1 \big(I_{H \cap U , \tilde u} ,A^\ast(1)/{F^+_v A^\ast(1)}  \big)^{D_{H \cap U,u }} @>>> H^1 \big(I_{\infty ,w} ,A^\ast(1)/{F^+_v A^\ast(1)}  \big)^{D_{\infty,w }} . \\ 
\end{CD}
\end{equation}
%By the same reason as that we gave on the diagram \eqref{equation:CDoutsidep}

Since the groups $D_{\infty ,w}/I_{\infty ,w}$ and $D_{H \cap U ,\tilde u}/I_{H \cap U , \tilde u}$ are procyclic group and thus have cohomological dimension one,  
the left vertical homomorphism and 
the right vertical homomorphism of the diagram are surjective. 
Since the group $E^{A^\ast(1)}_{H \cap U ,
\tilde{u}}$ is the kernel of the lower horizontal homomorphism of the 
diagram, we have a map 
\begin{equation}\label{equation:p-part_surjection_to_E}
\mathrm{Ker} [H^1 (D_{H \cap U ,\tilde u} ,A^\ast(1)/{F^+_v A^\ast(1)}) \longrightarrow  H^1 (D_{\infty ,w} ,A^\ast(1)/{F^+_v A^\ast(1)}  )]
\rightarrow 
E^{A^\ast(1)}_{H \cap U ,
\tilde{u}}, 
\end{equation}
whose cokernel is a subquotient of the kernel of the right vertical homomorphism of \eqref{cyclotomic-error-estimate-over-p}
by applying the snake lemma to the diagram 
\eqref{cyclotomic-error-estimate-over-p}. 
We note that the kernel of the right vertical homomorphism of \eqref{cyclotomic-error-estimate-over-p} is a finite group 
for any prime $w$ of $K_\infty$ dividing $p$ if the condition $(\mathrm{A}_p)$ is satisfied. 
By the Inflation-Restriction sequence, 
we have  
\begin{small}{\begin{multline}\label{equation:E_p-part}
H^1 (D_{H \cap U ,\tilde u}/ D_{\infty ,w}, (A^\ast(1)/{F^+_v A^\ast(1)})^{G_{K_{\infty ,w}}})
\\ \cong 
\mathrm{Ker} [H^1 (D_{H \cap U ,\tilde u} ,A^\ast(1)/{F^+_v A^\ast(1)} ) \longrightarrow  H^1 (D_{\infty ,w} ,A^\ast(1)/{F^+_v A^\ast(1)}  )]. 
\end{multline}}\end{small}
We note that this group is a finite group for any prime $v$ of $K$ dividing $p$ and for any prime $\tilde u$ (resp. $w$) of 
$K_UK_\cyc$ (resp. $K_\infty$) above $v$ if the condition $(\mathrm{A}_p)$ is satisfied. 
Let $v$ be a prime of $K$ dividing $p$. Then, this group for $\tilde u$ and  $w$ dividing $v$ and group for ${\tilde u}'$ and  $w'$ dividing 
$v$ are isomorphic to each other. This group is finite and uniformly bounded independently of $U$, $\tilde u$ and $w$ 
if the condition $(\mathrm{A}_p)$ is satisfied.
By applying these facts to  \eqref{equation:p-part_surjection_to_E}, we conclude that 
\begin{equation}\label{080920193}
(E^{A^\ast(1)}_{H \cap U ,\tilde{u}})_\mathrm{null} \text{ is finite and uniformly bounded independently of $U$, $\tilde u$ and $w$}
\end{equation}
if the condition $(\mathrm{A}_p)$ is satisfied.
 
\par As in the previous case, using induced modules, we can deduce 
 from \eqref{080920193} that 
 \begin{small}{$$\underset{U}{\varprojlim}~\Big(\bigoplus_{\tilde{u}\in P_{H \cap U }, \tilde{u} \mid p} (E^{A^\ast(1)}_{H \cap U ,\tilde{u}})^\vee\Big)_\mathrm{null} \cong  \displaystyle{\bigoplus_v}  \La_{\mathcal O}(G) \otimes _{\La_{\mathcal O}(G_v)} M_v
 $$}\end{small}
where $v$ runs through primes of $K$ dividing $p$ and $M_v$ is a $\La_{\mathcal O}(G_v)$-module of finite cardinality for each $v$. 
Once again, since $\La(G) \otimes_{\La(G_{v})} -$ is a well-defined map from $\M_{H_{v}}(G_{v}) \lra \M_H(G)$ for $v \mid p$,  it follows  Proposition \ref{h-level-vanishing}(a) that
\begin{small}{\begin{equation}\label{080921095}
\Big[\underset{U}{\varprojlim}\Big(\bigoplus_{\tilde{u}\in P_{H \cap U }, \tilde{u} \mid p} (E^{A^\ast(1)}_{H \cap U ,\tilde{u}})^\vee\Big)_\mathrm{null}\Big]  = 0 \text{ in }K_0(\mathfrak M_H(G)) 
\end{equation}}\end{small}
When the condition $(\mathrm{A}_p)$ is not satisfied, $(E^{A^\ast(1)}_{H \cap U ,\tilde{u}})$ might not be a finite group. 
However, $\Big(\bigoplus_{\tilde{u}\in P_{H \cap U }, \tilde{u} \nmid p}(E^{A^\ast(1)}_{H \cap U ,\tilde{u}})^\vee\Big)_\mathrm{null}$ is finite 
for each $U$, $\tilde u$ and $w$ and  
$
\underset{U}{\varprojlim}\big(\underset{\tilde{u}\in P_{H \cap U }, \tilde{u} \mid v}  \bigoplus(E^{A^\ast(1)}_{H \cap U ,\tilde{u}})^\vee\big)_\mathrm{null}$ is expressed as  $ \displaystyle{\bigoplus_v} \La_{\mathcal O}(G) \otimes _{\La_{\mathcal O}(G_v)} M_v$ 
where $v$ runs through primes of $K$ dividing $p$ and $M_v$ is a finitely generated $\mathbb{Z}_p$-module. Hence by Proposition \ref{h-level-vanishing} and by the condition (Van$_p$), the same conclusion as in \eqref{080921095}  continues to hold. This completes the proof of the lemma.
\end{proof}

\begin{proposition}\label{proposition:vanishing_higher_ext}
Let us assume that  $\mathrm{Sel}^{\mathrm{Gr}}_{A^\ast(1)}(K_\infty )^\vee \in \M_H(G)$  and either the condition $\mathrm{(a)}$ or the condition 
$\mathrm{(b)}$ below holds. 
\begin{enumerate} 
\item[(a)] The following two conditions are satisfied simultaneously.
\begin{enumerate}
\item[(i)]
The condition $\mathrm{(A)}$ in Definition \ref{set2b} is satisfied. 

\item[(ii)]
For each  $n \geq 1 $, the cohomology group $H_n (H\cap U ,\mathrm{Sel}^{\mathrm{Gr}}_{A^\ast (1)} (K_\infty )^\vee  )$ is finite 
for any $U \in \mathcal{U}$ and the order of this cohomology group is bounded independently of $U \in \mathcal{U}$. 

\end{enumerate}
\item[(b)] The following two  conditions are satisfied simultaneously.
\begin{enumerate}
\item[(i)] The assumption  $\mathrm{ (Van)}$ in Definition \ref{definition:setting} holds. 
\item[(ii)] For each  $n \geq 1$, the cohomology group $H_n (H\cap U ,\mathrm{Sel}^{\mathrm{Gr}}_{A^\ast (1)} (K_\infty )^\vee  )$ is finite for any $U \in \mathcal{U}$ and  $\underset{U}{\varprojlim}H_n (H\cap U ,\mathrm{Sel}^{\mathrm{Gr}}_{A^\ast (1)} (K_\infty )^\vee  )$ is a finitely generated $\Z_p$-module.
\end{enumerate}
\end{enumerate}

Then we have $[\mathrm{Ext}^n_{\Lambda_\mathcal O (G)} ( \mathrm{Sel}^{\mathrm{Gr}}_{A^\ast (1)} (K_\infty )^\vee  , \Lambda_\mathcal O (G) )] = 0 $  in 
$K_0 (\mathfrak{M}_H (G))$  for any $n\geq 3$.  

\par
We continue to assume that $\mathrm{Sel}^{\mathrm{Gr}}_{A^\ast(1)}(K_\infty )^\vee $ is in $ \M_H(G)$ and that either the condition $\mathrm{(a)}$ or the condition 
$\mathrm{(b)}$ above holds.  In addition, we also assume  either the condition $(\mathrm{A}_p)$ of Definition \ref{set2b} or the condition  $\mathrm{(Van}_p)$ of Definition \ref{definition:setting} holds.

Then we have $[\mathrm{Ext}^2_{\Lambda_\mathcal O (G)} ( \mathrm{Sel}^{\mathrm{Gr}}_{A^\ast (1)} (K_\infty )^\vee  , \Lambda_\mathcal O (G) )] = 0 $  in 
$K_0 (\mathfrak{M}_H (G))$. 

\end{proposition}

\begin{proof} 

We will first prove that $[\mathrm{Ext}^n_{\Lambda_\mathcal O (G)} ( \mathrm{Sel}^{\mathrm{Gr}}_{A^\ast (1)} (K_\infty )^\vee  , \Lambda_\mathcal O (G) )] = 0 $  in 
$K_0 (\mathfrak{M}_H (G))$  for any $n\geq 3$. 

Let $d$ be the dimension of $G$ as a $p$-adic Lie group. 
Then the global dimension of $\La_\mathcal O(G)$ is equal to $d+1$. 
\par 

Set $H_U = H \cap U$ and $\widetilde{\Gamma}_U = G/H_U= \mathrm{Gal}(K_\infty^{H_U}/K)$. There is a Grothendieck spectral sequence 
\begin{small}{\begin{equation*}
E_2^{s,t} = \text{Ext}^s_{{\mathcal{O}}[[{\widetilde{\Gamma}_U}]]}\big( H_t(H_U, \mathrm{Sel}^{\mathrm{Gr}}_{A^\ast (1)} (K_\infty )^\vee), O
[[{\widetilde{\Gamma}_U}]]\big)
\Rightarrow E_\infty^{s,t}=
\text{Ext}^{s+t}_{\La_{\mathcal{O}}(G)}\big (\mathrm{Sel}^{\mathrm{Gr}}_{A^\ast (1)} (K_\infty )^\vee, {\mathcal{O}}[[{\widetilde{\Gamma}_U}]] \big). 
\end{equation*}}\end{small}
%\begin{multline*}
%E_2^{s,t} = \text{Ext}^s_{{\mathcal{O}}[[{\widetilde{\Gamma}_U}]]}\big( H_t(H_U, \mathrm{Sel}^{\mathrm{Gr}}_{A^\ast (1)} (K_\infty )^\vee), O
%[[{\widetilde{\Gamma}_U}]]\big)
%\\  
%\Rightarrow E_\infty^{s,t}=
%\text{Ext}^{s+t}_{\La_{\mathcal{O}}(G)}\big (\mathrm{Sel}^{\mathrm{Gr}}_{A^\ast (1)} (K_\infty )^\vee, {\mathcal{O}}[[{\widetilde{\Gamma}_U}]] \big). 
%\end{multline*}
Also, for any $n$, we have 
\begin{equation}\label{prtyuvbw}
\text{Ext}^{n}_{\La_{\mathcal{O}}(G)}\big (\mathrm{Sel}^{\mathrm{Gr}}_{A^\ast (1)} (K_\infty )^\vee, {\mathcal{O}}[[{G}]] \big) \cong \underset{U}{\varprojlim}~ \mathrm{Ext}^{n}_{\La_{\mathcal{O}}(G)}\big (\mathrm{Sel}^{\mathrm{Gr}}_{A^\ast (1)} (K_\infty )^\vee, {\mathcal{O}}[[{\widetilde{\Gamma}_U}]] \big).
\end{equation}
Since $K_\infty^{H_U} = K_\infty^U K_\infty ^H = K_UK_\cyc$, we have $[K_\infty^{H_U}: K_\cyc] < \infty $ for any $U$. Thus we have $[\widetilde{\Gamma}_U:\Gamma_U
] < \infty$ where $\Gamma_U: = \text{Gal}(K_UK_\cyc/K_U)$. By Lemma  \ref{jannsen-lemma}, we have,  
$$
\mathrm{Ext}^n_{{\mathcal{O}}[[{\widetilde{\Gamma}_U}]]}\big( M, {\mathcal{O}}[[{\widetilde{\Gamma}_U}]]\big) \cong \mathrm{Ext}^n_{{\mathcal{O}}[[{\Gamma_U}]]}\big( M, {\mathcal{O}}[[\Gamma_U]]\big) 
$$ 
for any ${\mathcal{O}}[[\widetilde{\Gamma}_U]]$-module $M$. 
We have $E_2^{s,t} =0 $ for $s \geq 3$ and for any $t$
since $ {\mathcal{O}}[[\Gamma_U]]$ is a regular local ring of Krull dimension two. 
We have that $H_t (H_U, \mathrm{Sel}^{\mathrm{Gr}}_{A^\ast (1)} (K_\infty )^\vee)$ is a torsion ${\mathcal{O}}[[\Gamma]]$-module for any $t$ 
thanks to \cite[Lemma 3.1]{cfksv} and our assumption that $\mathrm{Sel}^{\mathrm{Gr}}_{A^\ast (1)} (K_\infty )^\vee$ is in $\M_H(G)$. 
Thus we obtain, $E_2^{0,t} =0 $ for any $t$. Putting these together, 
we have  $E_2^{s,t} =0 $ unless $s=1, 2$. 
\par 
Consequently, we obtain  $E_2^{1,t} =E_\infty^{1,t}$ and $E_2^{2,t} =E_\infty^{2,t}$ 
for any $t$.
This implies the following exact sequence for any $n$: 
\begin{equation}\label{breaking-ext-group}
0  \longrightarrow 
E_2^{1,n-1}
 \longrightarrow
\text{Ext}^{n}_{\La_{\mathcal{O}}(G)}\big (\mathrm{Sel}^{\mathrm{Gr}}_{A^\ast (1)} (K_\infty )^\vee, {\mathcal{O}}[[{\widetilde{\Gamma}_U}]] \big)
\longrightarrow E^{2,n-2}_2 \longrightarrow 0 
\end{equation}
Now  we assume the condition (a) of this proposition. Then 
 $\text{Ext}^{n}_{\La_{\mathcal{O}}(G)}\big (\mathrm{Sel}^{\mathrm{Gr}}_{A^\ast (1)} (K_\infty )^\vee, {\mathcal{O}}[[{\widetilde{\Gamma}_U}]] \big)$ is finite and uniformly bounded independent of $U$ for any $n \geq 3$ by our assumption (a)(ii) in this Proposition \ref{proposition:vanishing_higher_ext}. 
Now applying Lemma \ref{h-level-vanishing}(a), from equation \eqref{prtyuvbw}, we obtain $[\text{Ext}^{n}_{\La_{\mathcal{O}}(G)}\big (\mathrm{Sel}^{\mathrm{Gr}}_{A^\ast (1)} (K_\infty )^\vee, \La_{\mathcal{O}}(G) \big)] =0$ in $K_0(\M_H(G))$ for any $n\geq 3$. 
\par 

On the other hand, if we have the assumption (b), then using the hypothesis (b)(ii), we get 
 $\mathrm{Ext}^{n}_{\La_{\mathcal{O}}(G)}\big (\mathrm{Sel}^{\mathrm{Gr}}_{A^\ast (1)} (K_\infty )^\vee, \La_{\mathcal{O}}(G) \big)$ is a finitely generated $\Z_p$-module for $n \geq 2$. Thus we are done by assumption (Van) in hypothesis (b)(i).
 \par 
 
Thus it only remains to show $[\text{Ext}^{2}_{\La_{\mathcal{O}}(G)}\big (\mathrm{Sel}^{\mathrm{Gr}}_{A^\ast (1)} (K_\infty )^\vee, \La_{\mathcal{O}}(G) \big)] =0 $ in $K_0(\M_H(G))$ to complete  the proof of the Proposition \ref{proposition:vanishing_higher_ext}.% under the assumption {\rm (a)}. 
\par
By our assumption,  $H_1(H_U, \mathrm{Sel}^{\mathrm{Gr}}_{A^\ast (1)} (K_\infty )^\vee)$ is finite whether we are in case (a) or (b); and hence  $E_2^{1,1} =0$. 
%Moreover, we have shown that $E_2^{0,2} =0$. 
On the other hand,
\begin{equation}\label{calculationa2}
E_2^{2,0}=\text{Ext}^2_{{\mathcal{O}}[[{\Gamma_U}]]}\big(  \mathrm{Sel}^{\mathrm{Gr}}_{A^\ast(1)} (K_\infty )^\vee_{H_U}, {\mathcal{O}}[[{\Gamma_U}]]\big) \cong  \big( \mathrm{Sel}^{\mathrm{Gr}}_{A^*(1)} (K_\infty )^\vee_{H_U}\big)
_{\mathrm{null}}, 
\end{equation}
 the maximal ${\mathcal{O}}[[{\Gamma_U}]]$ pseudo-null (i.e. finite) submodule of $ \mathrm{Sel}^{\mathrm{Gr}}_{A^*(1)} (K_\infty )^\vee_{H_U}$. 
\par 
By \eqref{prtyuvbw},  \eqref{breaking-ext-group} and \eqref{calculationa2}, 
we obtain  
 $$
 [\text{Ext}^{2}_{\La_{\mathcal{O}}(G)}\big (\mathrm{Sel}^{\mathrm{Gr}}_{A^\ast (1)} (K_\infty )^\vee, \La_{\mathcal{O}}(G) \big)] =[ \underset{U}{\varprojlim}~ 
 \big( \mathrm{Sel}^{\mathrm{Gr}}_{A^*(1)} (K_\infty )^\vee_{H_U}\big)
 _{\mathrm{null}}].
 $$
We fix an open subgroup $U \subset G$ as above for the moment. Then, we have a natural restriction map:  
\begin{equation}\label{cyclotomic-discrete-map}
 \mathrm{Sel}^{\mathrm{Gr}}_{A^*(1)} (K_UK_\cyc ) \lra  \mathrm{Sel}^{\mathrm{Gr}}_{A^*(1)} (K_\infty )^{H\cap U}.
\end{equation} 
Further, by applying the Pontryagin dual functor to \eqref{cyclotomic-discrete-map}, we get the map $\phi^\vee_{U,\cyc}$
\begin{equation}\label{cyc-control-map}
 \mathrm{Sel}^{\mathrm{Gr}}_{A^*(1)} (K_\infty )^\vee_{H_U} \stackrel{\phi^\vee_{U,\cyc}}{\lra}  \mathrm{Sel}^{\mathrm{Gr}}_{A^*(1)} (K_UK_\cyc )^\vee.
 \end{equation}
Next, using \eqref{cyc-control-map}, it suffices to show  
\begin{equation}\label{cyc-control-map-consequence}
[\underset{U}{\varprojlim}~ 
\mathrm{Ker}(\phi^\vee_{U,\cyc})_{\mathrm{null}}]  =0 
\text{ and } [\underset{U}{\varprojlim}~ 
 \mathrm{Sel}^{\mathrm{Gr}}_{A^*(1)} (K_UK_\cyc)^\vee_{\mathrm{null}}] =0 
\text{ in } K_0(\M_H(G)).
\end{equation}
Taking direct limit  of the natural restriction map in the diagram \eqref{equation:control} over all open normal subgroups $W$ of $G$ with  $H \cap U \subset W \subset U$  and then using a snake lemma on that diagram,  we obtain an exact sequence 
\begin{equation}\label{justify-finite-submodule-part}
H^2(H\cap U, A^{G_{K_\infty}})^\vee \lra \mathrm{Ker}(\phi^{\vee} _{U,\cyc}) \lra \underset{W}{\varprojlim}~{\mathrm{res}''_W}^\vee, 
\end{equation}
where $\mathrm{res}''_W:= \mathrm{res}''_{\textbf{1},W}$ is the local restriction map defined in diagram \eqref{equation:control} with trivial $\rho$. 
\par 

We now assume the hypothesis (a) in this proposition. In addition, we also assume either $(\mathrm{A}_p)$  or  $\mathrm{(Van}_p)$ holds.
\par
 First we consider $ \mathrm{Sel}^{\mathrm{Gr}}_{A^*(1)} (K_UK_\cyc)^\vee _{\mathrm{null}}$. By a result of \cite{hm}, 
$$
 \mathrm{Sel}^{\mathrm{Gr}}_{A^*(1)}(K_UK_\cyc)^\vee_{\mathrm{null}} \cong \underset{n}{\varprojlim}~ \text{Ker}\big( \mathrm{Sel}^{\mathrm{Gr}}_{A^*(1)} ((K_UK_\cyc)^{\Gamma^{p^n}_U})\lra  \mathrm{Sel}^{\mathrm{Gr}}_{A^*(1)} (K_UK_\cyc )^{\Gamma^{p^n}_U}\big). 
$$ 
Then by the proof of cyclotomic control theorem (see for example \cite{oc}), 
the cardinality of the finite group $ \mathrm{Sel}^{\mathrm{Gr}}_{A^*(1)}(K_UK_\cyc)^\vee _{\mathrm{null}} $ is uniformly bounded independently of $U$ by the cardinality of the module $\big(A^*(1)\big)^{G_{K_\infty}}$, which is finite by the hypothesis (A) in condition (a)(i) of this proposition. Also, thanks to the hypothesis (A), the cardinality of $H^2(H\cap U, A^{G_{K_\infty}})^\vee$ is finite and uniformly bounded independently of $U$ following the proof of  Theorem 1. 
\par
Using this in equation \eqref{cyc-control-map-consequence}, 
it remains to establish $\underset{U}{\varprojlim}\Big(\underset{W}{\varprojlim}~{\mathrm{res}''_W}^\vee\Big)_\mathrm{null}$ is uniformly bounded independently of $U$ to complete the proof of the proposition in this case.  Recall the notation $\Big(\underset{\tilde{u}\in P_{H \cap U }}  \bigoplus{E^{A^\ast(1)}_{H \cap U ,\tilde{u}}}^\vee\Big)_\mathrm{null}$ from Lemma \ref{ext2-van}. Now we further assume condition (A$_p$) holds. Notice that $\underset{U}{\varprojlim}\Big(\underset{W}{\varprojlim}~{\mathrm{res}''_W}^\vee\Big)_\mathrm{null}$ differs from $\underset{U}{\varprojlim}\Big(\underset{\tilde{u}\in P_{H \cap U }}  \bigoplus{E^{A^\ast(1)}_{H \cap U ,\tilde{u}}}^\vee\Big)_\mathrm{null}$ by a finite module whose order is uniformly bounded independent of $U$.   Then applying Lemma \ref{ext2-van},  it follows that $\Big[\underset{U}{\varprojlim}\Big(\underset{\tilde{u}\in P_{H \cap U }}  \bigoplus{E^{A^\ast(1)}_{H \cap U ,\tilde{u}}}^\vee\Big)_\mathrm{null}\Big]=0 $ in $K_0(\M_H(G))$.

\par 
 The proof for the case with the assumption (b) works  similarly with suitable modifications as follows:
 
 \par  In this case,  we will get $\underset{U}{\varprojlim}~ 
 \mathrm{Sel}^{\mathrm{Gr}}_{A^*(1)} (K_UK_\cyc)^\vee_{\mathrm{null}}$ and $\underset{U}{\varprojlim}~ H^2(H\cap U, A^{G_{K_\infty}})^\vee$ are finitely generated $\Z_p$-modules. Thus their class will vanish in $K_0$ group by the hypothesis (Van) in condition (b)(i). Finally, thanks to the additional hypothesis (Van$_p$), in this case by applying Lemma \ref{ext2-van}, we again deduce $\Big[\underset{U}{\varprojlim}\Big(\underset{W}{\varprojlim}~{E^{A^{\ast}(1)}_W}^\vee\Big)_\mathrm{null}\Big]  =\Big[\underset{U}{\varprojlim}\Big(\underset{\tilde{u}\in P_{H \cap U }}  \bigoplus{E^{A^\ast(1)}_{H \cap U ,\tilde{u}}}^\vee\Big)_\mathrm{null}\Big]=0 $ in $K_0(\M_H(G))$.
 
\par This shows that $[\text{Ext}^{2}_{\La_{\mathcal{O}}(G)}\big (\mathrm{Sel}^{\mathrm{Gr}}_{A^\ast (1)} (K_\infty )^\vee, \La_{\mathcal{O}}(G) \big)] =0 $ and completes the proof of the proposition. 
\end{proof}

\begin{example}\label{rmk33} 
We discuss the hypotheses $\mathrm{(Van)}$ for $K_0(\M_H(G))$ and $\mathrm{(Van}_p)$ for  $K_0(\M_{H_v}(G_v))$ of Definition \ref{definition:setting}. Recall that  $G = \mathrm{Gal}(K_\infty/K)$,  $H = \mathrm{Gal}(K_\infty/K_\cyc)$ and for a prime $v$ in $K$ above $p$, $G_v$ is the decomposition subgroup of $G$ at $p$ and $H_v: = H \cap G_v$.

\begin{enumerate}
\item For the false-Tate curve extension in Example \ref{example:2}, $[\Z_p] \neq 0$ in $K_0(\M_H(G))$ (\cite{zab}). 
 
\item Let $G= H \times \Gamma$ be a $p$-adic Lie group with $H$ as in Proposition \ref{h-level-vanishing}(b). Then we have $[\Z_p] = 0 $ in $K_0(\M_H(G))$. We give a brief explanation of this.  
By our choice of $H$ from Proposition \ref{h-level-vanishing}(b), we have $[\F_p ]=0$ in $K_0(\Omega (H))$.  Note that there is an isomorphism between $K_0(\La(H))$ and $K_0(\Omega(H))$ \cite[Lemma 4.1]{cfksv} via the natural map sending $[M]$ to $[M/pM]$. Thus we deduce that $[\Z_p] =0 $ in $ K_0(\La(H))$. As $H$ has no element of order $p$, $\La(H)$ has finite global dimension. Thus any finitely generated $\La(H)$-module defines a class in $K_0(\La(H))$. As $G = H \times \Gamma$, any finitely generated $\La(H)$-module belongs to the category $\M_H(G)$. Hence there is a natural map from $K_0(\La(H))$ to $K_0(\M_H(G))$ sending the class of a finitely generated $\La(H)$-module $M$ to $[M]$. Using the composite map from $K_0(\Omega(H)) \lra K_0(\M_H(G))$ sending $[\F_p]$ to $[\Z_p]$, the result follows. 
\item Let $G = \mathrm{GL}_2(\Z_p)$. Set $H = \widetilde{\mathrm{SL}_2(\Z_p)} :=\{ g \in  \mathrm{GL}_2(\Z_p) \mid \mathrm{det}(g)^{p-1} =1 \}$. Then we have $[\Z_p] =0 $ in $K_0(\M_H(G))$.
We follow the proof of  \cite[Proposition 4.2]{z2} and provide some details. Note that we have the homomorphism $\mathrm{GL}_2(\Z_p) \stackrel{\theta}{\lra} 1+p\Z_p$ where  $\theta(g ) = \mathrm{det}(g)^{p-1}$. Then $\mathrm{ker}(\theta)=\widetilde{\mathrm{SL}_2(\Z_p)}$. As $p$ is odd, 
any $\lambda \in 1+p\Z_p$ has a unique $2(p-1)$-th root $\sqrt[2(p-1)]{\lambda }$ exists in $1+p\Z_p$.  Now we have  a section of $\theta$ given by $ 1+p\Z_p \stackrel{\psi}{\lra} G$ where  $\psi(\lambda) = \begin{small}{\begin{pmatrix} \sqrt[p-1]{\lambda}  & 0 \\ 0  & \sqrt[p-1]{\lambda}  \end{pmatrix}}\end{small}$. Using this we also see that  $\theta $ is surjective and $G/H=\mathrm{GL}_2(\Z_p)/\widetilde{\mathrm{SL}_2(\Z_p)} \cong 1+p\Z_p$. Note that by definition of $\psi$, $\Gamma := \psi(1+p\Z_p)$ is a normal subgroup of $\mathrm{GL}_2(\Z_p)$ with $\Gamma \cap \widetilde{\mathrm{SL}_2(\Z_p)} = I_{2,2}$. Thus we have shown that $\mathrm{GL}_2(\Z_p)=\{ g \in  \mathrm{GL}_2(\Z_p) \mid \mathrm{det}(g)^{p-1} =1 \} \times \Gamma \cong \widetilde {\mathrm{SL}_2(\Z_p)} \times ~ 1+ p\Z_p.$
\par It is a general fact that for any $p$-adic Lie group $H$ without an element of order $p$, the  dimension of the centralizer (as a $p$-adic Lie group) of an element $x \in H$  is the same as the $\Q_p$-vector space dimension of $\{g\in \mathrm{Lie}(H) \mid \mathrm{Ad}(x)(g) =g \}$ \cite[Proof of Lemma 8.6]{aw}. (cf. {\it loc. cit.} for the definition of  $\mathrm{Lie}(H)$ and $\mathrm{Ad}(-)$). We use this fact for $\widetilde {\mathrm{SL}_2(\Z_p)} $  and keep in mind that $\mathrm{SL}_2(\Z_p)$ is a finite index subgroup of $\widetilde {\mathrm{SL}_2(\Z_p)} $. By Lemma \ref{fp-0-in-sl2zp} and Corollary \ref{fp-0-criterion}, we deduce $[\F_p] =0$ in $K_0(\Omega(\widetilde{\mathrm{SL}_2(\Z_p)}))$. Finally using  Example \ref{rmk33}(2), we deduce that $[\Z_p] =0 $ in $K_0\Big(\M_{\widetilde{\mathrm{SL}_2(\Z_p)}}\big(\mathrm{GL}_2(\Z_p)\big)\Big)$. 

\item
Let us consider the case in Example \ref{example:3}, Case \rm{(O)} with $E$ a non-CM elliptic curve isogenous to $B$ with good ordinary reduction at all primes above $p$ in $K$; for the extension $K_\infty  = K(E_{p^\infty})$ of a number field $K $ with $G = \mathrm{Gal}(K_\infty/K)$ and $H:= \mathrm{Gal}\big(K(E_{p^\infty})/K_\cyc\big)$, the assumption (Van) is satisfied in $K_0(\M_H(G))$. 

In this case, by Serre's open image theorem $G$ is of finite index in $\mathrm{GL}_2(\Z_p)$ and also $H $ has a finite index subgroup which is of finite index in $\widetilde{{\mathrm{SL}_2(\Z_p)}}.$ Thus for any module $M\in \M_{\widetilde{\mathrm{SL}_2(\Z_p)}}\big(\mathrm{GL}_2(\Z_p)\big)$, the map $K_0\Big(\M_{\widetilde{\mathrm{SL}_2(\Z_p)}}\big(\mathrm{GL}_2(\Z_p)\big)\Big) \lra K_0\big(\M_H(G)\big)$ sending $[M] \lra [M]$ is well-defined. Taking $M=\Z_p$, the assumption \rm{(Van)} is satisfied in $K_0(\M_H(G))$ from the discussion in Example \ref{rmk33}(3).

  \iffalse{
  More generally, for   $A= B_{p^{\infty}}$ is associated to a non-CM elliptic curve $B$ over $\Q$  and the extension $K_\infty$ satisfies    $A^{G_{K_\infty}}$ is infinite then it was shown in \cite[Lemma 2.3]{bz}  there is a map from $G =\mathrm{Gal}(K_\infty/K) \lra \mathrm{GL}_2(\Z_p)$ with open image. Using that for any such extension $K_\infty/K$, it was shown  $[\Z_p] = 0 $ in $K_0(\M_H(G))$ \cite[ Proposition 2.4]{bz}.}
  \fi
  \item Let us again consider the case in Example \ref{example:3}, Case (O) with $E$ a non-CM elliptic curve isogenous to $B$ with good ordinary reduction at all primes above $p$ in $K$. We have the extension $K_\infty  = K(E_{p^\infty})$ of a number field $K $ with $G = \mathrm{Gal}(K_\infty/K)$ and $H:= \mathrm{Gal}\big(K(E_{p^\infty})/K_\cyc\big)$. For any prime $v$ in $K$ above $p$ let $G_{v}$ be the decomposition subgroup of $G$ at $v$ and set $H_v : = H \cap G_{v}$. Then $[\Z_p] =0 $ in $K_0\big(\M_{H_{v}}(G_{v})\big)$ for every prime $v$ in $K$ above $p$. This can be explained by a suitable modification of the  Example \ref{rmk33}(3) as follows. 
  
  As $E$ has good ordinary reduction at $v$ by Serre's work (see for example \cite[Page 196]{ch}) $G_{v}$ is a $3$-dimensional $p$-adic Lie group which is a finite index subgroup of the Borel subgroup $B$ of all upper triangular matrices in $\mathrm{GL}_2(\Z_p)$. Thus restricting the map $\theta$ in Example \ref{rmk33}(3) to $B$ and noting that determinant of a matrix in $B$ is given by the cyclotomic character, we see $\theta{|_B} $ is  still surjective on $1+ p\Z_p$. Moreover, the section $\psi: 1+p\Z_p \lra B$ works as $B$ contains all diagonal matrices. As $G_{v}$ and $H_{v}$ are respectively finite index subgroups of $B$ and 
 $\begin{small}{\left\{ \begin{pmatrix} a  & b \\ 0  & d \end{pmatrix}  \in \mathrm{GL}_2(\Z_p) \ \middle| \ \mathrm{det} \begin{pmatrix} a  & b \\ 0  & d \end{pmatrix}^{p-1} =1 \right\}}\end{small}$, we deduce $[\Z_p] =0 $ in $K_0\big(\M_{H_{v}}(G_{v})\big)$ following the discussions in Example \ref{rmk33}(3) and Example \ref{rmk33}(4).
  
  In this case, we further deduce $[\La(G) \otimes_{\La(G_{v})} \Z_p] =0$ in $K_0(\M_H(G))$ for every prime $v$ in $K$ above $p$ as %. Indeed,  the cyclotomic $\Z_p$ extension $K_\cyc$ is finitely decomposed at every prime of $K$ and hence 
  $\La(G) \otimes_{\La(G_{v})} -$ is a well-defined map from $\M_{H_v}(G_v) \lra \M_H(G)$. Thus the statement  follows from the above discussion.
 % \item For any $n \geq 2$, let $G:=\mathrm{GL}_n(\Z_p)$ and $ \widetilde {\mathrm{SL}_n(\Z_p)}:=\{ g \in  \mathrm{GL}_n(\Z_p) \mid \mathrm{det}(g)^{p-1} =1 \}.$ Then $[\Z_p] = 0 $ in $K_0\Big(\M_{\widetilde {\mathrm{SL}_n(\Z_p)}}\big(\mathrm{GL}_n(\Z_p)\big)\Big)$. The proof is essentially the same as in Example \ref{rmk33}(3) once we have the result: for every $n \geq 2$, the centralizer of every element in $\mathrm{SL}_n(\Z_p)$ is infinite. This can be proved in a way similar to Lemma \ref{fp-0-in-sl2zp}. 
\end{enumerate}
\end{example}

\begin{rem}\label{motivation-for-non-vanish}
Proposition  \ref{proposition:vanishing_higher_ext} is  a generalization of  \cite[Proposition 3.6]{bz}. We explain our natural motivation for the assumption in this proposition  which is not mentioned in the existing literature. See also Remark \ref{false-tate-pseudo-null}.

Let $A= E_{p^\infty}$, where $E$ is an elliptic curve over $\Q$ with good ordinary reduction at $p$. The condition that $H_i (H\cap U ,\mathrm{Sel}^{\mathrm{Gr}}_{A^\ast (1)} (K_\infty )^\vee  )$ is finite for any $U \in \mathcal{U}$ reflects the fact that  conjectural $p$-adic $L$-function has no poles. 

More precisely, let $\mathcal{L}_p(V_pE)$ be the conjectural $p$-adic $L$-function of $E$ over $K_\infty$ (see Appendix \ref{appendix} and Section \ref{s6}). Then one would like to ensure the evaluation of the $p$-adic $L$-function at any Artin representation $\eta$ of $\mathrm{Gal}(K_\infty/\Q)$ has the property $\eta(\mathcal{L}_p(V_pE)) \neq \infty$. 

Assume that $\mathrm{Sel}^{\mathrm{Gr}}_{A^\ast (1)} (K_\infty )^\vee \in \M_H(G)$  and let $\xi_{\mathrm{Sel}^{\mathrm{Gr}}_{E_{p^\infty}} (K_\infty )^\vee} \in K_1(\La_{\mathcal{O}}(G)_{S^*})$ be the pre-image of $[\mathrm{Sel}^{\mathrm{Gr}}_{A^\ast (1)} (K_\infty )^\vee] \in K_0(\M_H(G))$ under the natural connecting map $K_1(\La_{\mathcal{O}}(G)_{S^*})\stackrel{\delta}{\lra}  K_0(\M_H(G)) \rightarrow 0.$ (see section \ref{s6}). Then $H_i (H\cap U ,\mathrm{Sel}^{\mathrm{Gr}}_{A^\ast (1)} (K_\infty )^\vee)$ is finite for any $U \in \mathcal{U}$ ensures $\eta(\xi_{\mathrm{Sel}^{\mathrm{Gr}}_{E_{p^\infty}} (K_\infty )^\vee}) \neq \infty$ for any Artin representation $\rho$ as above \cite[Theorem 3.8]{cfksv}.  By Iwasawa Main conjecture,  (see  Conjecture \ref{773}) $\delta(\mathcal{L}_p(V_pE))= [\mathrm{Sel}^{\mathrm{Gr}}_{A^\ast (1)} (K_\infty )^\vee)]$ in $K_0(\M_H(G))$. Thus our hypothesis will conjecturally ensure that  $\eta(\mathcal{L}_p(V_pE)) \neq \infty$, as claimed. 
\end{rem}
\begin{example}\label{example2-2}
We continue here with the set up and notation of Example \ref{example:2} related to the false-Tate curve extension. Let $A$ be as given in Case (1), (2) or (3) of Example \ref{example:firstone}.   If $\mathrm{Sel}^{\mathrm{Gr}}_{A}( \Q(\mu_{p^\infty}))^\vee $ is a 
$($finitely generated$)$ torsion $\La(\Gamma)$-module then $\mathrm{Sel}^{\mathrm{Gr}}_{A}(K_\infty)^\vee $  is a $($finitely generated$)$ torsion $\La(G)$-module $($cf.  \cite[second proof of Theorem 2.8]{hv}$)$.   The  proof of \cite[Lemma 2.5]{css} easily extends in all of our cases and we deduce $H_1(H\cap U,  \mathrm{Sel}^{\mathrm{Gr}}_{A}(K_\infty)^\vee ) =0 $. Also as $H$ has $p$-cohomological dimension $=1$,  $H_i(H\cap U, -)$ also vanish for $i \geq 2$. 

Thus whenever $\mathrm{Sel}^{\mathrm{Gr}}_{A}(K_\infty)^\vee \in \mathfrak M_H(G)$, then all the conditions of  Proposition  \ref{proposition:vanishing_higher_ext}(a)  are satisfied and we deduce $[\mathrm{Ext}^i_{\Lambda (G)} ( \mathrm{Sel}^{\mathrm{Gr}}_{A^\ast (1)} (K_\infty )^\vee  , \Lambda (G) )] = 0 $  in 
$K_0 (\mathfrak{M}_H (G))$  for every $i\geq 2$.

Moreover, if  we assume $\mathrm{Sel}^{\mathrm{Gr}}_{A}(\Q(\mu_{p^\infty}))^\vee$ is a finitely generated $\Z_p$-module, then $\mathrm{Sel}^{\mathrm{Gr}}_{A}(K_\infty)^\vee$ is finitely generated over $\La(H)$ and in particular  $ \in \mathfrak M_H(G)$ (cf. \cite[Theorem 3.1(i)]{hv}). 
\end{example}

\begin{rem}\label{false-tate-pseudo-null}
In the false-Tate situation of Example \ref{example2-2}, if  we assume $\mathrm{Sel}^{\mathrm{Gr}}_{A}(K_\infty)^\vee$ has no non-zero $\La(G)$ pseudo-null submodule, we have $[\mathrm{Ext}^i_{\Lambda (G)} ( \mathrm{Sel}^{\mathrm{Gr}}_{A^\ast (1)} (K_\infty )^\vee  , \Lambda (G) )] =0$ 
in $K_0(\M_H(G))$ even without invoking Proposition  \ref{proposition:vanishing_higher_ext}.  
\par 
To prove this, we set $E^i(-): =\mathrm{Ext}^i_{\Lambda (G)}(-, \Lambda (G))$. 
\begin{enumerate}
\item[\rm{(1)}] 
We have $E^i(-) =0 $ for every $i >3$ since the 
$p$-cohomological dimension of $G$ is $2$. 
\end{enumerate}
Also, the following results are known by \cite{ve}: 
\begin{enumerate}
\item[\rm{(2)}] 
We have $E^j(E^i(-)) =0$ when $j <i$ $($\cite[Proposition 3.5(iii)(a)]{ve}$)$. 
\item[\rm{(3)}] 
For a $\La(G)$-module $N$ with no non-zero pseudo-null submodule, we have 
 $E^i(E^i(N))=0$ for every $i \geq 2$ $($\cite[Proposition 3.5(i)(c)]{ve}$)$. 
\end{enumerate}
Combining these three facts above, for a $\La(G)$-module $N$ with no non-zero pseudo-null submodule, we have $E^i(E^3(N)) =0 $ for every $i \geq 0$ and $E^i(E^2(N)) =0$ for every $i \neq 3$. Hence we deduce $E^3(N) $ 
and  $E^2(N)$ are both finite  \cite[Proposition 2.6(5)]{v1}.  Consequently, we have $ [E^2(N)] =0$ and $ [E^3(N)]=0$  in $K_0(\mathfrak M_H(G))$ by Lemma \ref{h-level-vanishing}(a). 
\par
Thus, whenever the maximal pseudo-null submodule of the 
$\La(G)$-module 
$\mathrm{Sel}^{\mathrm{Gr}}_{A}(K_\infty)^\vee$ is trivial, 
we can apply the above argument to $N =  \mathrm{Sel}^{\mathrm{Gr}}_{A}(K_\infty)^\vee$ to obtain 
$$
[\mathrm{Ext}^i_{\Lambda (G)} ( \mathrm{Sel}^{\mathrm{Gr}}_{A} (K_\infty )^\vee  , \Lambda (G))] = 0
\text{ in $K_0(\mathfrak M_H(G))$ for every $i \geq 2$.}
$$  
We remark that \cite[Theorem 2.6(ii), Theorem 3.1(ii)]{hv} discussess 
a variety of  assumptions under which 
 maximal pseudo-null submodule of the $\La(G)$-module $\mathrm{Sel}^{\mathrm{Gr}}_{A}(K_\infty)^\vee$ is trivial.  
\par 
However, it seems that this argument based on the non-existence of pseudo-null submodules will work only in such  special $p$-adic Lie extensions like false-Tate curve extension and 
it can not be generalized to higher Ext groups of the Selmer group associated to a general $p$-adic Lie group $G$ with a large $p$-cohomological dimension. 
In a general situation, we need Proposition  \ref{proposition:vanishing_higher_ext}. %(also see \cite[Theorem 2.6(ii)]{hv})
\end{rem}

\begin{example}

Next we discuss the vanishing of higher ext groups for the Selmer groups considered in Example \ref{example:3}. We carry the same setting from Example \ref{example:3}. Let us fix $p \geq 5$ and assume we are in the Case {\rm (O)} and take $E = B$, a non-CM elliptic curve over $\Q$ with good ordinary reduction at $p$. Here  $A= E_{p^\infty}$.   When  $\mathrm{Sel}^{\mathrm{Gr}}_{A}(K_\cyc)^\vee$ is torsion over $\La(G)$,  it follows that $H_i(H\cap U,  \mathrm{Sel}^{\mathrm{Gr}}_{A}(K_\infty)^\vee ) =0 $ for every $i \geq 1$ \cite[Lemma 2.5 and Remark 2.6]{css}.  Then all the conditions of  Proposition  \ref{proposition:vanishing_higher_ext}(b)  are satisfied and we deduce $[\mathrm{Ext}^i_{\Lambda (G)} ( \mathrm{Sel}^{\mathrm{Gr}}_{A^\ast (1)} (K_\infty )^\vee  , \Lambda (G) )] = 0 $  in 
$K_0 (\mathfrak{M}_H (G))$  for every $i\geq 2$. 
\par
In this setting, it is conjectured in  \cite[Conjecture 5.1]{cfksv} that  $\mathrm{Sel}^{\mathrm{Gr}}_{A}(K_\infty)^\vee  \in \mathfrak M_H(G)$  and in particular it is a finitely generated torsion $\La(G)$-module.
\par
Whenever \  $\mathrm{Sel}^{\mathrm{Gr}}_{A}(K_\cyc)^\vee$ is a finitely generated $\Z_p$-module, then it follows that \  $\mathrm{Sel}^{\mathrm{Gr}}_{A}(K_\infty)^\vee$ is finitely generated over $\La(H)$ and in particular  $ \in \mathfrak M_H(G)$ \cite[Proposition 5.6]{cfksv}. 

Note that, even in Case {\rm (S)}, the proof of \cite[Proposition 5.6]{cfksv} extends and we can deduce $\mathrm{Sel}^{\mathrm{Gr}}_{A}(K_\infty)^\vee  \in \mathfrak M_H(G)$ if $\mathrm{Sel}^{\mathrm{Gr}}_{A}(K_\cyc)^\vee  $ is finitely generated over $\Z_p$.
\end{example}
\begin{example}
Let us keep the same setting as Example \ref{example:4} where $G$ is isomorphic to $\mathbb{Z}^{d+1}_p$ and 
$A$ is isomorphic to $E_{p^\infty}$ for $E$ an elliptic curve with good ordinary reduction at all primes of $K$ above $p$. 
\par 
For any $M \in \mathfrak M_H(G)$, $\mathrm{Ext}^i_{\La(G)}(M, \La(G))$ is a pseudonull $\La(G)$-module for every $i \geq 2$ $($see \cite[\S 3]{ve}$)$. 
Then, since $\La (G)$ is commutative and $\mathrm{Ext}^i_{\La(G)}(M, \La(G))$ is a pseudonull $\La (G)$-module, we have $[\mathrm{Ext}^i_{\La(G)}(M, \La(G))] =0$ in $K_0(\mathfrak M_H(G))$ for every $i \geq 2$. Recall that we have $A\cong A^\ast(1)$ and hence 
$\mathrm{Sel}^{\mathrm{Gr}}_{A} (K_\infty )^\vee \cong \mathrm{Sel}^{\mathrm{Gr}}_{A^\ast (1)} (K_\infty )^\vee$. 
Thus, if we know $\mathrm{Sel}^{\mathrm{Gr}}_{A} (K_\infty )^\vee$   $\in \mathfrak M_H(G)$, then we deduce $[\mathrm{Ext}^i_{\Lambda (G)} ( \mathrm{Sel}^{\mathrm{Gr}}_{A} (K_\infty )^\vee  , \Lambda (G) )] = 0 $  in 
$K_0 (\mathfrak{M}_H (G))$  for every $i\geq 2$.   
\par 
By the argument of the proof  of Lemma \ref{ext2-van}, we can show that the kernel of the map  
$$\big(\mathrm{Sel}^{\mathrm{Gr}}_{A} (K_\infty )^\vee \big)_H  \lra  \mathrm{Sel}^{\mathrm{Gr}}_{A} (K_\cyc )^\vee$$ is a finite generated $\Z_p$-module. Then, by Nakayama's lemma, we have $\mathrm{Sel}^{\mathrm{Gr}}_{A} (K_\infty )^\vee \in \mathfrak M_H(G)$ 
if $\mathrm{Sel}^{\mathrm{Gr}}_{A} (K_\cyc )^\vee$ is a finitely generated $\Z_p$-module.
\end{example}

%%%%%%%%%%%%%%%%%%%%%%%%%%%%%%%%%%%%%%%%%%%%%%%%%%%%%%%%%%%%%%%
\section{Proof of the algebraic functional equation (Theorem \ref{function-equation-thm})}\label{s3}

In this section, we will prove the functional equation of the Selmer group. We will use the control theorem and the vanishing of the higher extension groups of the Selmer group as discussed respectively in Section \ref{s2} and Section \ref{sext}.  We first discuss a few results which will go into the proof of Theorem \ref{function-equation-thm}.

\par Recall that $\mathfrak M_H(G)$ is the category of finitely generated $\La_{\mathcal{O}}(G)$-module such that $M/M(p)$ is finitely generated over $\La_{\mathcal{O}}(H)$. Since $\mathrm{Sel}^{\mathrm{BK}}_{A_\rho} (K_U)$ is a subgroup of $\mathrm{Sel}^{\mathrm{Gr}}_{A_\rho} (K_U)$, we define a module $C^{A_\rho}_U$
 as follows: 
\begin{small}{\begin{equation}\label{qwer23397684}
C^{A_\rho}_U =\mathrm{Coker}
\left[ 
\mathrm{Sel}^{\mathrm{BK}}_{A_\rho} (K_U ) \lra \mathrm{Sel}^{\mathrm{Gr}}_{A_\rho} (K_U )
\right].
 \end{equation}}\end{small}
\begin{lemma}\label{lemma:triviality_of_C}
Assume either the condition $(\mathrm{A}_p)$ of Definition \ref{set2b}
or the condition {\rm (Van$_p$)} of Definition \ref{definition:setting}. 
Then, we have 
$
\Big[ 
\big(\underset{U}{\varinjlim}C^{A_\rho}_U\big)^\vee 
\Big] 
=0 $ and 
$\Big[ 
\underset{U}{\varprojlim}C^{A_\rho}_U
\Big] 
=0 $ in $K_0(\M_H(G))$.
\end{lemma}
\begin{proof} 
The comparison between the Selmer group of Bloch-Kato and the Selmer group of Greenberg 
is known by \cite[Theorem 3]{flach} and \cite[Proposition 4.2]{oc00}. Recall the sequence 
$$
0 \lra \mathrm{Sel}^{\mathrm{BK}}_{A_\rho} (K_U ) \lra \mathrm{Sel}^{\mathrm{Gr}}_{A_\rho} (K_U ) \lra C^{A_\rho}_U \lra 0.
$$
By the proof of \cite[Proposition 4.2]{oc00}, $C^{A_\rho}_U$ is a finite group which is subquotient of 
\begin{small}{\begin{multline}\label{equation:comparisonerror1}
\displaystyle{
\bigoplus_{u\vert p} 
\Bigl( (
( A^\ast_\rho  (1))^{D_{u}})^\vee \oplus 
(A_\rho / F^+_u A_\rho )^{D_u}
\oplus 
( F^+_u A^\ast_\rho (1))^{D_{u}}
\Bigr) }
\\ 
\displaystyle{
\oplus 
\bigoplus_{u\nmid p} \left( \left.  (
( A^\ast_\rho  (1))^{I_{u}})_{D_u }
\right/ 
((
( A^\ast_\rho  (1))^{I_{u}})_{D_u })_{\mathrm{div}} 
\right) }
\end{multline}}\end{small}
where $u$ runs through the set of primes of $K_U$ over $p$ 
on the first line and $u$ runs through the set of primes of $K_U$ away from $p$ on the second line. 
\eqref{equation:comparisonerror1} is reformulated as follows 
\begin{small}{\begin{multline}\label{equation:comparisonerror2}
\displaystyle{
\bigoplus_{v\vert p} 
\mathrm{Ind}^{{G_{v,U}}}_{G/U}
\Bigl\{ 
\left( 
\left( 
( A^\ast_\rho  (1))^{D_{u(v)}}
\right)^\vee 
\right) 
\oplus 
\left( 
(A_\rho / F^+_{u(v)} A_\rho )^{D_{u(v)}}
\right) 
\oplus 
\left( 
( F^+_{u(v)} A^\ast_\rho (1))^{D_{u(v)}}
\right) 
\Bigr\} }
\\ 
\displaystyle{
\oplus 
\bigoplus_{v\nmid p} 
\left. 
\mathrm{Ind}^{{G_{v,U}}}_{G/U} \left\{ 
(
( A^\ast_\rho  (1))^{I_{u(v)}})_{D_{u(v)} }
\right/ 
((
( A^\ast_\rho  (1))^{I_{u(v)}})_{D_{u(v)} })_{\mathrm{div}} 
\right\} 
}
\end{multline}}\end{small}
where $v$ runs through the set of primes of $K$ over $p$ 
on the first line and $v$ runs through the set of primes of $K$ away from $p$ on the second line. For each $v$, we choose a prime $u(v)$ of 
$K_U$ which is over $v$ and we denote by $G_{v,U}$ 
the image of the decomposition group $G_v \subset G$ into $G/U$. 
The term \eqref{equation:comparisonerror2} is independent of the choices 
of $u(v)$'s. 
Hence, $(C_\infty^{A_\rho})^\vee :=(\underset{U}{\varinjlim} C^{A_\rho}_U)^\vee$ is a subquotient of the module the Pontryagin dual of the injective limit of \eqref{equation:comparisonerror2} with respect to 
open normal subgroups $U$ of $G$. 
Note that the injective limit of 
$\left( 
( A^\ast_\rho  (1))^{D_{u(v)}} 
\right)^\vee $ is the Pontryagin dual of the inverse limit 
of $
( A^\ast_\rho  (1))^{D_{u(v)}}$, which is trivial since the order of the group 
$( A^\ast_\rho  (1))^{D_{u(v)}}$ is finite and uniformly bounded 
by the assumption $(\mathrm{A}_p)$. 
\par 
Hence the module the Pontryagin dual of the injective limit of \eqref{equation:comparisonerror2} with respect to 
open normal subgroups $U$ of $G$ is as follows 
\begin{small}{\begin{multline}\label{equation:injectivelimit_first_assertion}
\displaystyle{
\Big( 
\bigoplus_{v\vert p} 
\mathrm{Ind}^{G_v}_G 
 \Bigl\{ 
\left( 
(A_\rho / F^+_{w(v)} A_\rho )^{D_{w(v)}}
\right) 
\oplus 
\left( 
( F^+_{w(v)} A^\ast_\rho (1))^{D_{w(v)}}
\right) 
\Bigr\} \Big) ^\vee }
\\ 
\displaystyle{
\oplus 
\bigoplus_{v\nmid p} \left( 
 \mathrm{Ind}^{G_v}_G 
 \left\{ 
\left.  (
( A^\ast_\rho  (1))^{I_{w(v)}})_{D_{w(v)} }
\right/ 
((
( A^\ast_\rho  (1))^{I_{w(v)}})_{D_{w(v)} })_{\mathrm{div}} 
\right\} 
\right)^\vee }
\end{multline}}\end{small}
where $v$ runs through the set of primes of $K$ over $p$ 
on the first factor and $v$ runs through the set of primes of $K$ away from $p$ on the second factor. Inside the sums, 
$w=w(v)$ is a prime of $K_\infty$ which is over $v$ for each $v$. 
Note the the groups does not depend on the choices of $w=w(v)$. 
\par 
By the condition $(\mathrm{A}_p)$, the group 
$
(A_\rho / F^+_{w(v)} A_\rho )^{D_{w(v)}}
\oplus 
( F^+_{w(v)} A^\ast_\rho (1))^{D_{w(v)}}$ 
is finite for each place $v \vert p$ of $K$. 
By applying Lemma \ref{h-level-vanishing}(a) to this $G_v$-module of finite 
cardinality, we obtain  
\begin{small}{$$
\left[ 
 \left( 
(A_\rho / F^+_{w(v)} A_\rho )^{D_{w(v)}}
\oplus 
( F^+_{w(v)} A^\ast_\rho (1))^{D_{w(v)}} \right)^\vee 
\right] =0 
$$}\end{small}
in $K_0(\M_{H_v}(G_v))$. Hence we have 
\begin{small}{\begin{equation*}
\Big[ \displaystyle{
\Big( 
\bigoplus_{v\vert p} 
\mathrm{Ind}^{G_v}_G 
 \Bigl\{ 
\left( 
(A_\rho / F^+_{w(v)} A_\rho )^{D_{w(v)}}
\right) 
\oplus 
\left( 
( F^+_{w(v)} A^\ast_\rho (1))^{D_{w(v)}}
\right) 
\Bigr\} \Big) ^\vee } 
\Big] =0 
\end{equation*}}\end{small}
in $K_0(\M_{H}(G))$. 
It is not difficult to see that the group 
$\left.  (
( A^\ast_\rho  (1))^{I_{w(v)}})_{D_{w(v)} }
\right/ 
((
( A^\ast_\rho  (1))^{I_{w(v)}})_{D_{w(v)} })_{\mathrm{div}} $ 
is finite for each place $v \nmid p$ of $K$ and the group is trivial 
if the action of $I_v$ on $A_\rho$ is trivial.  
Hence, by using Lemma \ref{h-level-vanishing}(a) to these $G_v$-modules of finite 
cardinality, we obtain 
\begin{small}{$$
\Big[ 
\Big( 
\bigoplus_{v\nmid p} \mathrm{Ind}^{G_v}_G  
\left\{ \left.  
(
( A^\ast_\rho  (1))^{I_{w(v)}})_{D_{w(v)} }
\right/ 
((
( A^\ast_\rho  (1))^{I_{w(v)}})_{D_{w(v)} })_{\mathrm{div}} 
\right\} \Big)^\vee  
\Big] =0 
$$}\end{small}
in $K_0(\M_{H}(G))$. This complete the proof for 
$
\Big[ 
\big(\underset{U}{\varinjlim}C^{A_\rho}_U\big)^\vee 
\Big] 
=0 $. 
\par 
As for $\underset{U}{\varprojlim} C^{A_\rho}_U$, 
it is isomorphic to 
$\Big( \underset{U}{\varinjlim} \big(C^{A_\rho}_U\big)^\vee 
\Big)^\vee$. 
Note that the module $(C^{A_\rho}_U)^\vee$ is a subquotient of 
\begin{small}{\begin{multline}\label{equation:comparisonerror_dual2}
\displaystyle{
\bigoplus_{v\vert p} 
\mathrm{Ind}^{{G_{v,U}}}_{G/U}
\Bigl\{ 
( A^\ast_\rho  (1))^{D_{u(v)}}
\oplus 
\left( 
(A_\rho / F^+_{u(v)} A_\rho )^{D_{u(v)}}
\right)^\vee  
\oplus 
\left( 
( F^+_{u(v)} A^\ast_\rho (1))^{D_{u(v)}}
\right)^\vee  
\Bigr\} }
\\ 
\displaystyle{
\oplus 
\bigoplus_{v\nmid p} 
\left. 
\mathrm{Ind}^{{G_{v,U}}}_{G/U} 
\left\{ 
\left( 
(
( A^\ast_\rho  (1))^{I_{u(v)}})_{D_{u(v)} }
\right/ 
((
( A^\ast_\rho  (1))^{I_{u(v)}})_{D_{u(v)} })_{\mathrm{div}} 
\right)^\vee 
\right\} 
}
\end{multline}}\end{small}
Under the assumption $(\mathrm{A}_p)$, a similar discussion as 
the discussion given before \eqref{equation:injectivelimit_first_assertion} 
implies that $\underset{U}{\varprojlim} C^{A_\rho}_U$ is a subquotient 
of 
\begin{small}{\begin{equation}\label{equation:injectivelimit_second_assertion}
%\displaystyle{
\Big( 
\bigoplus_{v\vert p} 
\mathrm{Ind}^{G_v}_G 
 \left\{ 
( A^\ast_\rho  (1))^{D_{w(v)}}
\right\} 
\Big)^\vee,
% }
\end{equation}}\end{small}
where $v$ runs through the set of primes of $K$ over $p$ and 
$w=w(v)$ is a prime of $K_\infty$ which is over $v$ for each $v$. 
Since $( A^\ast_\rho  (1))^{D_{w(v)}}$ is finite for each $v$ by 
the assumption $(\mathrm{A}_p)$, the same discussion as 
the case of $\Big[ 
\big(\underset{U}{\varinjlim}C^{A_\rho}_U\big)^\vee 
\Big] $ implies that 
\begin{small}{$$
\Big[
\Big( 
\bigoplus_{v\vert p} 
\mathrm{Ind}^{G_v}_G 
 \left\{ 
( A^\ast_\rho  (1))^{D_{w(v)}}
\right\} 
\Big)^\vee  
\Big] =0 
$$}\end{small}
in $K_0(\M_{H}(G))$. 
This completes the proof for the second statement $\Big[ 
\underset{U}{\varprojlim}C^{A_\rho}_U
\Big] 
=0 $ in $K_0(\M_H(G))$ under the condition $(\mathrm{A}_p)$. 
\par 
Finally we remark that, when $(\mathrm{A}_p)$ does not hold, the groups 
$( A^\ast_\rho  (1))^{D_{w(v)}}$, 
$(A_\rho / F^+_{w(v)} A_\rho )^{D_{w(v)}}$, 
$( F^+_{w(v)} A^\ast_\rho (1))^{D_{w(v)}}$ for primes 
of $K$ over $p$ and 
the groups 
\begin{small}{$\left(
( A^\ast_\rho  (1))^{I_{w(v)}})_{D_{w(v)} }
\right/ 
((
( A^\ast_\rho  (1))^{I_{w(v)}})_{D_{w(v)} })_{\mathrm{div}} $}\end{small}
for primes 
of $K$ away from $p$ are a successive extension of finite groups and 
a copies of $\mathbb{Q}_p/\mathbb{Z}_p$.    
If we assume {\rm (Van$_p$)}, then
$
\Big[ 
\Big( 
\bigoplus_{v\vert p} 
\mathrm{Ind}^{G_v}_G 
 \left\{ 
\mathbb{Q}_p /\mathbb{Z}_p 
\right\} 
\Big)^\vee  
\Big] =0 
$
in $K_0(\M_{H}(G))$. Hence, 
even in the case where $(\mathrm{A}_p)$ does not hold, 
the same argument applies as far as the assumption {\rm (Van$_p$)} 
holds. This completes the proof. 
\end{proof}
\par
Next, we recall  the following theorem from \cite{jo}: 
%%%%
\begin{theorem}[Main Theorem of \cite{jo}]\label{mainin}
Let $G$ be a compact $p$-adic Lie group and let $H$ 
be its closed subgroup such that $G/H$ is isomorphic to $\Gamma$.  
Let $M$ be a $\La_{\mathcal{O}}(G)$-module which is finitely generated over $\La_{\mathcal{O}}(H)$. 

Then there exists a continuous character $\rho : G \lra \Z_p^\times$ such that $M(\rho)_U$ 
is finite for every open normal subgroup $U$ of $G$. 
\end{theorem} 
\par
Next, we recall the definition of $G$-Euler characteristic: 
\begin{defn}
Let $G$ be a compact $p$-adic Lie group without any element of order $p$. For a finitely generated $\Lambda_{\mathcal O}(G)$-module $M$, we say that the $G$-Euler characteristic of $M$ exists if the homology groups $H_i(G,M)$ are finite for every $i \geq 0$. When the $G$-Euler characteristic of $M$ exists, then set $  \chi(G,M) :=  \prod_i (\#H_i(G,M))^{(-1)^i}$
\end{defn}

Extending Theorem \ref{mainin} of \cite{jo}, the following theorem was established in \cite{js}:
\begin{theorem}\label{euler-gen}
Let  $G$ be a compact $p$-adic Lie group without any element of order $p$ and let $H$ 
be its closed subgroup such that $G/H$ is isomorphic to $\Gamma$. Let $M$ be a $\La_\mathcal O(G)$-module which is in $\mathfrak M_H(G)$ i.e. $M/{M(p)}$ is finitely generated over $\La_\mathcal O(G)$. Then there exists a continuous character $\rho: \Gamma \lra \Z_p^\times$, such that  $\chi(U, M (\rho))$ exists for every open normal subgroup $U$ of $G$.\end{theorem}
By applying  Theorem \ref{euler-gen}, we deduce the following corollary.
\begin{corollary}\label{twist-summary}
We assume that $\mathrm{Sel}^{\mathrm{Gr}}_{A^*(1)} (K_\infty )^\vee$ is in $\mathfrak M_H(G)$. Then there exists a continuous character $\rho: \Gamma \lra \Z_p^\times$, such that   both $\Big(\mathrm{Sel}^{\mathrm{Gr}}_{A^*(1)} (K_\infty )^\vee (\rho)\Big)_U$ and $H_1\Big(U, \mathrm{Sel}^{\mathrm{Gr}}_{A^*(1)} (K_\infty )^\vee (\rho)\Big)$ are finite.
 \end{corollary}
Next we prove the following lemma:
\begin{lemma}\label{base-change-ext1}
Let $M$ be a finitely generated $\La_{\mathcal{O}}(G)$-module and assume that $H_1(U,M)$ is finite for an open normal subgroup $U$ of $G$. Then we have a $\mathcal{O}[ G/U]$-module isomorphism $\mathrm{Ext}^1_{\La_{\mathcal{O}}(G)} ( M,  \mathcal{O}[ G/U]) \cong \mathrm{Ext}^1_{\mathcal{O}[ G/U]} ( M_U,  \mathcal{O}[ G/U])$.
\end{lemma}
\begin{proof} To simplify the notation through this proof, we write $\La:= \La_{\mathcal{O}}(G)$. By the definition of Ext functor, 
$M$ fits into an exact sequence $0 \lra R \stackrel{\theta}{\lra} P \lra M \lra 0$ where $P$ is a projective $\La_{\mathcal{O}}(G)$-module and $R$ is a $\La_{\mathcal{O}}(G)$-module. This gives rise to a second exact sequence 
$0 \lra \frac{R_U}{\text{Image}(H_1(U,M))} \stackrel{\theta_U}{\lra} P_U \lra M_U \lra 0.$ 

By our assumption, we have $\mathrm{Hom}_{O[G/U]}({\text{Image}(H_1(U,M))}, O[G/U]) =0$. Then we can compute the extension groups appearing in the statement of this lemma as follows: 
\small 
\begin{align*}
& 
\mathrm{Ext}^1_\La(M, O[G/U]) \cong \frac{\mathrm{Hom}_\La(R, O[G/U])}{\theta\big(\mathrm{Hom}_\La(P, O[G/U])\big)} \cong  \frac{\mathrm{Hom}_{O[G/U]}(R_U, O[G/U])}{\theta_U\big(\mathrm{Hom}_{O[G/U]}(P_U, O[G/U]))},
\\ 
& \mathrm{Ext}^1_{O[G/U]}(M_U, O[G/U]) \cong \frac{\mathrm{Hom}_{O[G/U]}(\frac{R_U}{\text{Image}(H_1(U,M))}, O[G/U])}{\theta_U\big(\mathrm{Hom}_{O[G/U]}(P_U, O[G/U])\big)} \cong  \frac{\mathrm{Hom}_{O[G/U]}(R_U, O[G/U])}{\theta\big(\mathrm{Hom}_{O[G/U]}(P_U, O[G/U]))}. 
\end{align*}
\normalsize
This completes the proof of the lemma.
 \end{proof}

Finally, we recall following  well known result  in homological algebra (see for example \cite[Proposition 6.1]{zab} for the proof): 
\begin{proposition}\label{ext-groups-in-k0}
Let $M$ be a module in $\M_H(G)$. Then $\mathrm{Ext}^i_{\La_{\mathcal{O}}(G)} (M ,  \La_{\mathcal{O}}(G) ) \in \M_H(G)$ for every $i \geq 1$ and we have the following equality in $\ K_0(\M_H(G))$.
\begin{small}{$$ [M] = \underset{1 \leq i \leq \mathrm{dim } G +1}{\sum} (-1)^{i+1} [\mathrm{Ext}^i_{\La_{\mathcal{O}}(G)} (M ,  \La_{\mathcal{O}}(G) )].$$}\end{small}
In particular, if we have $ [\mathrm{Ext}^i_{\La_{\mathcal{O}}(G)} (M ,  \La_{\mathcal{O}}(G) )]  = 0$ in $K_0(\M_H(G))$  for every $i \geq 2$, then we have 
\begin{small}{$$[M] = [\mathrm{Ext}^1_{\La_{\mathcal{O}}(G)} (M ,  \La_{\mathcal{O}}(G) )] \quad \text{in} \quad K_0(\M_H(G)).$$}\end{small}  
\end{proposition} 
We are now ready to prove Theorem 2.

\begin{proof}[Proof of Theorem  \ref{function-equation-thm} (Functional Equation)] We claim the following equality  in $K_0(\M_H(G))$
\begin{equation}\label{duality}
 [\big( \mathrm{Sel}^{\mathrm{Gr}}_A (K_\infty )^\vee \big )^\iota ] + 
[(E_1^{A^\ast(1)})^\iota] = [\mathrm{Ext}^1_{\La_{\mathcal{O}}(G)} (  \mathrm{Sel}^{\mathrm{Gr}}_{A^\ast(1)} (K_\infty )^\vee  ,  \La_{\mathcal{O}}(G) )].
\end{equation}

First, we will complete the proof of  Theorem \ref{function-equation-thm} assuming \eqref{duality}. By applying  
Proposition \ref{ext-groups-in-k0} to $M=\mathrm{Sel}^{\mathrm{Gr}}_{A^\ast(1)} (K_\infty )^\vee$ and by using \eqref{duality}, we deduce the following equality in $K_0(\M_H(G)):$
\begin{multline}\label{general-fnal-eqn}
 [\big(\mathrm{Sel}^{\mathrm{Gr}}_A (K_\infty )^\vee)^\iota] + \big[(E^{A^\ast(1)}_1)^\iota\big]
\\ = \big[\mathrm{Sel}^{\mathrm{Gr}}_{A^\ast(1)} (K_\infty )^\vee \big] + \underset{i \geq 2 }{\sum} (-1)^{i} [\mathrm{Ext}^i_{\La_{\mathcal{O}}(G)} ( \mathrm{Sel}^{\mathrm{Gr}}_{A^\ast(1)} (K_\infty )^\vee   ,  \La_{\mathcal{O}}(G) )].
\end{multline}
By our assumption (3) in Theorem \ref{function-equation-thm}, either the condition \ref{proposition:vanishing_higher_ext}(a) or the condition  \ref{proposition:vanishing_higher_ext}(b) %(A) of Definition  \ref{set2b} 
is satisfied.  Then by Proposition \ref{proposition:vanishing_higher_ext}, $[\mathrm{Ext}^i_{\La_{\mathcal{O}}(G)} ( \mathrm{Sel}^{\mathrm{Gr}}_{A^\ast(1)} (K_\infty )^\vee   ,  \La_{\mathcal{O}}(G) )] =0$ for $i \geq 3$. Moreover by assumption (4) of Theorem 2, either (A$_p$) or (Van$_p$) is satisfied. Then   from Proposition \ref{proposition:vanishing_higher_ext}, we can deduce $[\mathrm{Ext}^2_{\La_{\mathcal{O}}(G)} ( \mathrm{Sel}^{\mathrm{Gr}}_{A^\ast(1)} (K_\infty )^\vee   ,  \La_{\mathcal{O}}(G) )] =0$ as well. From the above discussion,  we deduce from  \eqref{general-fnal-eqn},
\begin{equation}\label{final-fnal-eqn}
 [\big(\mathrm{Sel}^{\mathrm{Gr}}_A (K_\infty )^\vee\big)^\iota] + \big[(E^{A^\ast(1)}_1)^\iota\big]
 = \big[\mathrm{Sel}^{\mathrm{Gr}}_{A^\ast(1)} (K_\infty )^\vee  \big].
\end{equation}
Twisting this by $\iota$, Theorem \ref{function-equation-thm} is an immediate consequence of \eqref{final-fnal-eqn}. \par 
In the rest of the proof, we are left to establish the claim in \eqref{duality}.
By our assumption in Theorem \ref{function-equation-thm}, $\mathrm{Sel}^{\mathrm{Gr}}_{A^*(1)} (K_\infty )^\vee$ is in $\mathfrak M_H(G)$.
Thus by Corollary \ref{twist-summary}, there exists a continuous character 
$\rho: G \lra \Gamma \lra \Z_p^\times $ such that  $\Big(\mathrm{Sel}^{\mathrm{Gr}}_{A^*(1)} (K_\infty )^\vee(\rho)\Big)_U$ and $H_1\Big(U, \mathrm{Sel}^{\mathrm{Gr}}_{A^*(1)} (K_\infty )^\vee(\rho)\Big)$ are both finite for any open normal subgroup $U$ of $G$.
Since the character $\rho$ is trivial on $G_{K_\infty}$, we deduce that $\mathrm{Sel}^{\mathrm{Gr}}_{A^*(1)} (K_\infty )^\vee (\rho) \cong \mathrm{Sel}^{\mathrm{Gr}}_{{A}_\rho^*(1)} (K_\infty )^\vee$. Thus, $\big(\mathrm{Sel}^{\mathrm{Gr}}_{{A}_\rho^*(1)} (K_\infty )^\vee\big)_U$ and $H_1\big(U, \mathrm{Sel}^{\mathrm{Gr}}_{{A}_\rho^*(1)} (K_\infty )^\vee\big)$ are both finite. By our hypothesis (1) in  Theorem \ref{function-equation-thm}, the kernel and the cokernel of the  natural restriction map $\mathrm{Sel}^{\mathrm{Gr}}_{{A}_\rho^*(1)} (K_U)  \stackrel{\text{res}^{A^*(1)}_{\rho, U}} \lra  \mathrm{Sel}^{\mathrm{Gr}}_{{A}_\rho^*(1)} (K_\infty )^U $ are both finite for any $U$ in $\mathcal U$. 
Hence we have 
\begin{equation}\label{twist-invariant-finite}
\mathrm{Sel}^{\mathrm{Gr}}_{{A}_\rho^*(1)} (K_U)  \text{ is finite for any open subgroup $U \subset G$}.
\end{equation}
\par 
\noindent Since the Bloch-Kato Selmer group $\mathrm{Sel}^{\mathrm{BK}}_{{A}_\rho^*(1)}
 (K_U)$ is a subgroup of $\mathrm{Sel}^{\mathrm{Gr}}_{{A}_\rho^*(1)} (K_U)$ 
by definition, \eqref{twist-invariant-finite} 
implies that 
\begin{equation}\label{twist-invariant-finite2}
\mathrm{Sel}^{\mathrm{BK}}_{{A}_\rho^*(1)} (K_U) \text{  is finite for any open subgroup $U \subset G$}. 
\end{equation}
Then, by the generalized Cassels-Tate pairing 
established for Bloch-Kato Selmer groups by Flach (cf. \cite[Theorem 1, Theorem 3]{flach}, \cite[Theorem 3.1.1]{pr}),  we have an isomorphism 
\begin{equation}\label{pairing}
\mathrm{Sel}^{\mathrm{BK}}_{A_\rho} (K_U )^\vee \cong \mathrm{Sel}^{\mathrm{BK}}_{A_\rho^\ast(1)} (K_U ), 
\end{equation}
for any $U$. Using the isomorphism in \eqref{pairing}, together with the natural restriction map and the natural inclusion of a Bloch-Kato Selmer group into a Greenberg Selmer group, we get an $\mathcal{O}[G/U]$-linear map $\mathrm{Sel}^{\mathrm{Gr}}_{A_\rho} (K_U )^\vee  \lra \mathrm{Sel}^{\mathrm{Gr}}_{A_\rho^\ast(1)} (K_\infty)^U$ by the composition as follows: 
\begin{equation}\label{pairing-control}
\mathrm{Sel}^{\mathrm{Gr}}_{A_\rho} (K_U )^\vee \rightarrow \mathrm{Sel}^{\mathrm{BK}}_{A_\rho} (K_U )^\vee 
\overset{\sim}{\rightarrow} \mathrm{Sel}^{\mathrm{BK}}_{A_\rho^\ast(1)} (K_U )% \stackrel{\phi^{\mathrm{BK}}_{\rho,U}} 
\rightarrow  \mathrm{Sel}^{\mathrm{Gr}}_{A_\rho^\ast(1)} (K_U) \rightarrow \mathrm{Sel}^{\mathrm{Gr}}_{A_\rho^\ast(1)} (K_\infty)^U.
\end{equation}

Taking the inverse limit with respect to the norm map over $U \in \mathcal U$, we get a $\La_{\mathcal{O}}(G)$-linear map 
\begin{equation}\label{pairing-limit2}
\mathrm{Sel}^{\mathrm{Gr}}_{A_\rho} (K_\infty )^\vee 
\stackrel{\phi^1_\rho} \lra  \underset{U \in \mathcal U}{\varprojlim} ~ \mathrm{Sel}^{\mathrm{Gr}}_{A_\rho^\ast(1)} (K_\infty)^U.  
\end{equation} where $\phi^1_\rho$ is by the composition as follows:
\begin{multline*}\label{pairing-limit}
\mathrm{Sel}^{\mathrm{Gr}}_{A_\rho} (K_\infty )^\vee \rightarrow \mathrm{Sel}^{\mathrm{BK}}_{A_\rho} (K_\infty )^\vee 
\overset{\sim}{\rightarrow}
 \underset{U \in \mathcal U}{\varprojlim} \mathrm{Sel}^{\mathrm{BK}}_{A_\rho^\ast(1)} (K_U ) 
\\ 
 \rightarrow \underset{U \in \mathcal U}{\varprojlim}  \mathrm{Sel}^{\mathrm{Gr}}_{A_\rho^\ast(1)} (K_U)  \rightarrow \underset{U \in \mathcal U}{\varprojlim} ~\mathrm{Sel}^{\mathrm{Gr}}_{A_\rho^\ast(1)} (K_\infty)^U . 
\end{multline*}
\par 
%{}{Applying Theorem 1 together with the isomorphism in \eqref{pairing}, we deduce that $\mathrm{Sel}^{\mathrm{Gr}}_{A_\rho^\ast(1)} (K_\infty )^U$ is also finite for every  $U$.}  
As $\mathrm{Sel}^{\mathrm{Gr}}_{A_\rho^\ast(1)} (K_\infty )^U$  is finite for any $U$, we have 

\begin{align*}
\underset{U \in \mathcal U}{\varprojlim}~ \big(\mathrm{Sel} ^{\mathrm{Gr}}_{A_\rho^\ast(1)} (K_\infty )^U \big)^\iota 
& 
= \underset{U \in \mathcal U}{\varprojlim}~ \text{Tor}_{\mathcal{O}}\big(\mathrm{Sel} ^{\mathrm{Gr}}_{A_\rho^\ast(1)} (K_\infty )^U \big)^\iota 
 \\ 
& 
\overset{\sim}{\longrightarrow} 
\underset{U \in \mathcal U}{\varprojlim}~ \text{Ext}^1_{\mathcal{O}}\Big(\mathrm{Sel} ^{\mathrm{Gr}}_{A_\rho^\ast(1)} (K_\infty )^\vee_U , 
\mathcal{O} \Big) \text{ (see \cite[Line 05, Page~732]{pr})}
\\ 
& \overset{\sim}{\longrightarrow} 
\underset{U \in \mathcal U}{\varprojlim} ~\mathrm{Ext}^1_{\mathcal{O}[G/U]} \Big( \mathrm{Sel}^{\mathrm{Gr}}_{A_\rho^\ast(1)} (K_\infty )^\vee_U,  \mathcal{O}[ G/U] \Big) \text{ (see Lemma \ref{jannsen-lemma})}  \\ 
%& 
%\overset{\sim}{\longrightarrow} 
%\underset{U \in \mathcal U}{\varprojlim} ~\mathrm{Ext}^1_{\mathcal{O}[G/U]} \Big( \big(\mathrm{Sel}^{\mathrm{Gr}}_{A_\rho^\ast(1)} (K_\infty )^\vee_U\big)^\iota,  \mathcal{O}[ G/U] \Big) \\ 
& 
\overset{\sim}{\longrightarrow} 
\underset{U \in \mathcal U}{\varprojlim} ~\mathrm{Ext}^1_{\La_{\mathcal{O}}(G)} \Big( \mathrm{Sel}^{\mathrm{Gr}}_{A_\rho^\ast(1)} (K_\infty )^\vee ,  \mathcal{O}[ G/U] \Big)  \text{ (by Lemma \ref{base-change-ext1})} \\ 
& 
\overset{\sim}{\longrightarrow}  \mathrm{Ext}^1_{\La_{\mathcal{O}}(G)} \Big( \mathrm{Sel}^{\mathrm{Gr}}_{A_\rho^\ast(1)} (K_\infty )^\vee ,  \La_{\mathcal{O}}(G) \Big). 
\end{align*}
Thus, we obtain a $\La_{\mathcal{O}}$-linear isomorphism
\begin{equation}\label{limit-ext1}
\underset{U \in \mathcal U}{\varprojlim} ~\big( \mathrm{Sel}^{\mathrm{Gr}}_{A_\rho^\ast(1)} (K_\infty )^U\big )^\iota \stackrel{\phi^2_\rho}{\cong} \mathrm{Ext}^1_{\La_{\mathcal{O}}(G)} \Big( \mathrm{Sel}^{\mathrm{Gr}}_{A_\rho^\ast(1)} (K_\infty )^\vee  ,  \La_{\mathcal{O}}(G) \Big).
\end{equation}
By composing $\phi^2_\rho$ in \eqref{limit-ext1} with $\phi^1_\rho$ of  \eqref{pairing-limit2} twisted by $\iota$,  we have a $\La_{\mathcal{O}}(G)$-linear map 
\begin{equation}\label{equation:ext1}
 \big(\mathrm{Sel}^{\mathrm{Gr}}_{A_\rho} (K_\infty )^\vee\big)^\iota 
\stackrel{\phi_\rho}
\longrightarrow \mathrm{Ext}^1_{\La_{\mathcal{O}}(G)} \Big( \mathrm{Sel}^{\mathrm{Gr}}_{A_\rho^\ast (1)} (K_\infty )^\vee ,  \La_{\mathcal{O}}(G) \Big). 
\end{equation}
\par
\noindent Now using the first vanishing assertion of Lemma \ref{lemma:triviality_of_C}, we have the following equality: 
\begin{equation}
[\mathrm{Sel}^{\mathrm{Gr}}_{A_\rho} (K_\infty )^\vee]= [\mathrm{Sel}^{\mathrm{BK}}_{A_\rho} (K_\infty )^\vee]  
\text{ in $K_0(\mathfrak M_H(G))$.}
\end{equation} 
Also, we have the following equality by taking the limit of \eqref{pairing} for $U \in \mathcal{U}$:
%Since the local condition of Bloch-Kato Selmer group has exact annihilation property for the local Tate-duality pairing, we have the following equality: 
\begin{equation}
 [\mathrm{Sel}^{\mathrm{BK}}_{A_\rho} (K_\infty )^\vee]  = [\underset{U \in \mathcal U}{\varprojlim} \mathrm{Sel}^{\mathrm{BK}}_{A_\rho^\ast(1)} (K_U )]
\text{ in $K_0(\mathfrak M_H(G))$.}
\end{equation} 
Moreover, by the second vanishing assertion of Lemma \ref{lemma:triviality_of_C}, we deduce 
\begin{equation}
[\underset{U \in \mathcal U}{\varprojlim} \mathrm{Sel}^{\mathrm{BK}}_{A_\rho^\ast(1)} (K_U )]= [ \underset{U \in \mathcal U}{\varprojlim}  \mathrm{Sel}^{\mathrm{Gr}}_{A_\rho^\ast(1)} (K_U)] 
\text{ in $K_0(\M_H(G))$.} 
\end{equation} 
Using these observations in \eqref{equation:ext1}, we deduce
\begin{align*}\label{equation:ext2}
& [\text{Ker}(\phi_\rho) ]=\Big[  \underset{U \in \mathcal U}{\varprojlim}~\big(\text{Ker}(\text{res}^{A^\ast(1)}_{\rho, U})\big)^\iota\Big],\\
&  [\text{Coker}(\phi_\rho) ]=\Big[\underset{U \in \mathcal U}{\varprojlim}~ \big(\text{Coker}(\text{res}^{A^\ast(1)}_{\rho, U})\big)^\iota\Big]  \text{ in } K_0(\M_H(G)).
\end{align*}
%{}{where $\sim$ above means $\La_{\mathcal{O}}(G)$-module isomorphism up to a module  (coming from the comparison of Bloch-Kato and Greenberg Selmer group) whose class in $K_0(\M_H(G))$ is trivial.}
As $\rho$ is trivial on $G_{K_\infty}$, we obtain by applying $-\otimes_{\mathcal{O}} \rho^{-1}$ to \eqref{equation:ext1}
\begin{equation}\label{equation:ext3}
\big( \mathrm{Sel}^{\mathrm{Gr}}_{A} (K_\infty )^\vee \big)^\iota
\stackrel{\phi}\longrightarrow \mathrm{Ext}^1_{\La_{\mathcal{O}}(G)} \Big( \mathrm{Sel}^{\mathrm{Gr}}_{A^\ast (1)} (K_\infty )^\vee  ,  \La_{\mathcal{O}}(G) \Big). 
\end{equation}
Thus we have 
\begin{equation}\label{equation:ext4}
[\text{Ker}(\phi) ]=\Big[  \underset{U \in \mathcal U}{\varprojlim}~\big(\text{Ker}(\text{res}^{A^\ast(1)}_{U})\big)^\iota\Big]~\text{ and }~ [\text{Coker}(\phi) ]=\Big[\underset{U \in \mathcal U}{\varprojlim}~ \big(\text{Coker}(\text{res}^{A^\ast(1)}_{ U})\big)^\iota\Big]  \text{ in } K_0(\M_H(G)).
\end{equation}

\begin{enumerate}
\item  First, we assume the condition 2(a) in Theorem \ref{function-equation-thm} holds. Then, it follows that $ \underset{U \in \mathcal U}{\varprojlim}~\text{Ker}(\text{res}^{A^\ast(1)}_{U})$ is finite in cardinality. Hence by Lemma \ref{h-level-vanishing}(a), we deduce from \eqref{equation:ext4} that 
\begin{equation}\label{kerzero}
[\text{Ker}(\phi)] =0 ~ \text{ in }~ K_0(\M_H(G)).
\end{equation}
\par 
Moreover, from \eqref{equation:ext4}, the condition 2(a) in Theorem \ref{function-equation-thm} implies that $[\text{Coker}(\phi) ]$ and $\big[(E_0^{A^\ast(1)})^\iota \big]$ in $K_0(\mathfrak M_H(G))$ differ by the class of a module which is finite in cardinality. Thus again using Lemma \ref{h-level-vanishing}(a),  we deduce from \eqref{equation:ext4} that $[\text{Coker}(\phi)] =
\big [(E_0^{A^\ast(1)})^\iota \big]$ in $K_0(\M_H(G))$. 
Further, the  assumption (4) in Theorem \ref{function-equation-thm} says that either (A$_p$) of Definition \ref{set2b} or (Van$_p$) of Definition \ref{definition:setting} holds. By Proposition \ref{cor-error-term-over-p}, 
we have $\big[(E_0^{A^\ast(1)})^\iota\big ] =\big[(E_1^{A^\ast(1)})^\iota\big]$ in $K_0(\M_H(G))$. The claim in \eqref{duality}  now follows  from \eqref{kerzero} and \eqref{equation:ext3}.

\item  Next, we assume the condition 2(b) in Theorem \ref{function-equation-thm} holds. The condition 2(b) says that %we assume that both (Red) and (Van) in Definition \ref{definition:setting}  are satisfied.   Thus from Theorem 1 that 
$ \underset{U \in \mathcal U}{\varprojlim}~\text{Ker}(\text{res}^{A^\ast(1)}_{U})$ is a finite generated $\Z_p$-module. Then due to our assumption (Van) of Definition \ref{definition:setting}, we deduce from \eqref{equation:ext4}, $[\text{Ker}(\phi)] =0$. Also, in this case from \eqref{equation:ext4}, $[\mathrm{Coker}(\phi)]$ and $\big[(E_0^{A^\ast(1)})^\iota \big]$ in $K_0(\mathfrak M_H(G))$ differ by the class of a finitely generated $\Z_p$-module. 
   Thanks to the assumption (Van) of Definition \ref{definition:setting}, which is part of hypothesis 2(b) of Theorem \ref{function-equation-thm}, we get from \eqref{equation:ext4} that $[\text{Coker}(\phi)] = \big[(E_0^{A^\ast(1)})^\iota ]$ in $K_0(\M_H(G))$. 
   
   Finally, we again obtain $\big[(E_0^{A^\ast(1)})^\iota\big ] =\big[(E_1^{A^\ast(1)})^\iota\big]$ in $K_0(\M_H(G))$ by Proposition \ref{cor-error-term-over-p}. Thus the claim in \eqref{duality} holds also in this case.
\end{enumerate}
\par 
This completes the proof of Theorem \ref{function-equation-thm}.
 \end{proof}
\par

\medskip
  
{\it Proof of Remark \ref{remark0.2}:} Note we have   $[\mathrm{Sel}^{\mathrm{Gr}}_{A} (K_\infty )^\vee] = [\mathrm{Sel}^{\mathrm{BK}}_{A} (K_\infty )^\vee] $ in $K_0(\M_H(G))$ due to the first vanishing result, for trivial $\rho$ in Lemma \ref{lemma:triviality_of_C} given by $[ 
(\underset{U}{\varinjlim}C^{A}_U)^\vee ] 
=0 $ in $K_0(\M_H(G))$.  In addition,  by the proof of Lemma \ref{lemma:triviality_of_C}, we also have $[ 
(\underset{U}{\varinjlim}C^{A^*(1)}_U)^\vee ] 
=0 $ in $K_0(\M_H(G))$ due to our assumption that either (A$_p$) or (Van$_p$) holds. As a result of this, the equality $[\mathrm{Sel}^{\mathrm{Gr}}_{A^\ast(1)} (K_\infty )^\vee] = [\mathrm{Sel}^{\mathrm{BK}}_{A^\ast(1)} (K_\infty )^\vee] $ in $K_0(\M_H(G))$ holds true. Applying the involution $\iota$ and using the functional equation of the Greenberg Selmer group, we finally deduce that 
  $$[\mathrm{Sel}^{\mathrm{BK}}_{A} (K_\infty )^\vee] + [E^{A^*(1)}_1] = [(\mathrm{Sel}^{\mathrm{BK}}_{A^\ast(1)} (K_\infty )^\vee)^\iota]  \text{ in } K_0(\M_H(G)). \qed$$ 
 
\section{{}{Compatibility of the algebraic and  the conjectural analytic functional equation}}\label{s6}
In this section, we explain how the functional equation of the conjectural $p$-adic $L$-function looks like and then we check the compatibility of this conjectural functional equation with the algebraic functional equation. 
{}{We will consider a special case of $V=V_{f,p}(\frac{k}{2})$ for a normalized eigen elliptic cuspform $f$ of even weight $k\geq 2$ and level $\Gamma_0 (N)$ whose $p$-th Fourier coefficient $a_p(f)$ is a $p$-adic unit such that  $N$ is square-free and the conductor $N_f$ of $f$ is not divisible by $p$ as discussed in Example 3 of Section \ref{s5}. Here $V_{f,p}$ is the $p$-adic Galois representation for $f$.} We take a lattice $T \subset V$ and we set $A =T\otimes \mathbb{Q}_p /\mathbb{Z}_p$. 
We set $K_\infty$ to be a false-Tate curve extension  in the setting of Example \ref{example:2}. Recall that $K=\Q $, $K_\infty= \Q(\mu_{p^\infty}, a^{1/p^\infty})$, $G= \mathrm{Gal}(K_\infty/\Q) \cong \Z_p^\times \rtimes \Z_p$, $H= \mathrm{Gal}(K_\infty/\Q_\cyc) \cong \Z_p^\times $ and $\Gamma=G/H = \mathrm{Gal}(\Q_\cyc /\Q)$. Recall from Section \ref{s5}, $$P_0:= \{ q \text{ prime in } \Q: q \mid a \text { but } q \nmid p\}.$$ 
Also, {}{we define $P_1$ and $P_2$ to be the subsets of $P_0$ given as in Example 3 of} Section \ref{s5}.

{}{Let $\eta$ be an Artin representation whose representation space over $\mathbb{C}$ is denoted by $W_\eta$}.  
As a special case of the general definition of complex $L$-function  given in  \eqref{most-general-complex-l-function} in Appendix, 
{}{
we have $L(f, \eta,s):= \underset{q }{\prod}P_q(f, \eta, q^{-s})^{-1}$ where we choose any prime $\ell \neq q$ and define $ P_q(f,\eta,T)=\text{det}\Big(1-\text{Frob}^{-1}_qT\ \vert \ 
\big(V_{f,\ell}  \otimes_{\Q_\ell} W_{\eta ,\ell} \big)^{I_q}\Big)$ where $V_{f,\ell}$ is the $\ell$-adic Galois representation for $f$ and $W_{\eta ,\ell}$ is an Artin representation which is isomorphic to $W_\eta$ but defined over an 
$\ell$-adic field.} 
\par 
{}{Let us fix complex periods $\Omega_+(f), \Omega_-(f) \in \mathbb{C}^\times$ 
such that $\dfrac{L(f,\chi,\tfrac{k}{2})}{\Omega_{\mathrm{sgn}(\chi)}(f)} \in \mathbb{Q}_f [\chi]$ 
for any Dirichlet character $\chi$ where $\mathrm{sgn}(\chi)$ is the signature of $\chi (-1)$. Here $\mathbb{Q}_f $ is the Hecke field of $f$ 
which is a finite extension of $\mathbb{Q}$ obtained by adjoining Fourier coefficients of $f$ to $\mathbb{Q}$. 
A priori, the periods $\Omega_+(f), \Omega_-(f) $ are defined only up to multiplication of elements of $\mathbb{Q}_f$. However, we can normalize them so that 
$\Omega_+(f), \Omega_-(f) $ are defined up to multiplication of units of the ring of integers of $\mathbb{Q}_f$ localized at a prime over $p$.}
\par 
As a special case of Conjecture \ref{p-adic-lfun-existence-conjecture} when {}{$V=V_{f,p}(\frac{k}{2})$} and $K_\infty$ as above, the conjectural $p$-adic $L$-function {}{$\mathcal{L}_p (V_{f,p}(\frac{k}{2}) ) =\mathcal{L}_p (V_{f,p}(\frac{k}{2}) ; \{ \Omega_\pm (f) \}) \in K_1(\La_\mathcal O(G)_{S^*})$\footnote{The $p$-adic $L$-function $\mathcal{L}_p (V_{f,p}(\frac{k}{2}) )$ essentially depends on the choice of 
periods $\Omega_+(f), \Omega_-(f)$. However, below, we remove the dependence from the notation to ease the notation.}} 
depending on the fixed periods {}{exists and it is characterized} by 
{}{
\begin{equation}\label{p-adiclfn-ellipticcurve}
\eta^\ast (\mathcal{L}_p (V_{f,p}(\tfrac{k}{2})) )=
 \epsilon_p (\eta) \alpha_p^{-C_p(\eta)} \times 
\frac{P_p(\eta^\ast , \alpha_p^{-1} )}{P_p(\eta , \beta_p^{-1})} 
\times 
\frac{L_{P}(f,\eta,\tfrac{k}{2})}{\Omega_+(f)^{d_+ (\eta)} \Omega _-(f)^{d_-(\eta)}}
\end{equation}
}
{}{where $\eta^\ast$ denotes the contragadient representation of $\eta$}. Here $\mathcal O$ is the ring of integers of some finite extension of $\Q_p$\footnote{Note that, according to the work of Fukaya-Kato \cite[Theorem 2.2.26]{fk} on Tamagawa number conjecture, it is suggested in \cite[Page 203]{cfksv} that the coefficient ring of Iwasawa algebra for the conjectural $p$-adic $L$-function be enlarged to a finite extension $\mathcal O$ even 
{}{when $\mathbb{Q}_f =\Q$}.}, $p^{C_p(\eta)}$ is the $p$-part of the conductor $\eta$ at 
 $p$, $\epsilon_p(\eta)$ is the local $\epsilon$ factor of $\eta$ at $p$ and $P$ is the set of primes of $\Q$ which are 
infinitely ramified in $K_\infty/\Q$ and {}{$L_{P}(f,\eta,\tfrac{k}{2}):= \underset{q \not\in P}{\prod}P_q(f, \eta, q^{-\tfrac{k}{2}})^{-1}$}. %is special value at $1$ of the complex $L$-function of $E$ twisted by an Artin representation $\eta$ of $G$ whose Euler factor at all $q \in P$  are removed.  
One has that $P= P_0 \cup \{p\}$ (see \cite[Lemma 3.9]{hv}).
\par 
{}{
If the weight $k$ is larger than two, the complex $L(f,\eta,s)$ has critical points other than $s=\frac{k}{2}$ and 
the interpolation property of the $p$-adic $L$-function $\mathcal{L}_p (V_{f,p}(\tfrac{k}{2}))$ covers 
more specializations than \eqref{p-adiclfn-ellipticcurve} corresponding to other critical points. However, since we assume that 
$f$ is ordinary, the interpolation property \eqref{p-adiclfn-ellipticcurve} is enough to characterize 
$\mathcal{L}_p (V_{f,p}(\tfrac{k}{2}))$.
}
\par 
On the other hand, recall that the twisted complex $L$-function of {}{$f$} has the following conjectural 
functional equation, 
{}{
\begin{equation}\label{qdi2ud90890-9-0bkgckm}
\Lambda (f, \eta,s ) = \omega(f, \eta)\Lambda (f, \eta^\ast,k-s), 
\end{equation}
where  $\Lambda(f, \eta,s )$ is the completed $L$-function which is obtained by multiplying certain $\Gamma$-factor to $L(f, \eta,s)$. 
% =  L(f, \eta,s) \Gamma(\tfrac{s}{2})^{\mathrm{dim}(\eta)}\Gamma(\tfrac{s+1}{2})^{\mathrm{dim}(\eta)}\left( 
%\tfrac{N(E,\eta)}{\pi^{2 \mathrm{dim}(\eta)}}\right)^{\frac{s}{2}}$.
}
The conjectural functional equation of complex $L$-function in \eqref{qdi2ud90890-9-0bkgckm} is a special case of Conjecture \ref{conjecture:Hasse-Weil}  when {}{$V=V_{f,p}(\tfrac{k}{2})\otimes W_{\eta ,p}$}.

%One can check the equality  $N(E,\eta)= N(E,\eta^\ast)$. 
Then putting {}{$s =\frac{k}{2}$} in  \eqref{qdi2ud90890-9-0bkgckm}  and using the definition of  {}{$\Lambda(f, \eta,\frac{k}{2} )$}, we  deduce the following equation: {}{
\begin{equation}\label{l-value-at-1}
L(f, \eta,\tfrac{k}{2}) = \omega(f, \eta) L(f, \eta^\ast,\tfrac{k}{2}).
\end{equation}
}

Recall that $P$ is the set of primes of $\Q$ which are 
infinitely ramified in $K_\infty/\Q$. By \eqref{p-adiclfn-ellipticcurve}, we have 
{}{
\begin{align*}
\eta^\ast (\mathcal{L}_p (V_{f,p}(\tfrac{k}{2})) )& =
 \epsilon_p (\eta) \alpha_p^{-C_p(\eta)} \times 
\frac{P_p(\eta^\ast , \alpha_p^{-1} )}{P_p(\eta , \beta_p^{-1})} 
\times 
\frac{L(f,\eta,\tfrac{k}{2})}{\Omega_+(f)^{d_+ (\eta)} \Omega _-(f)^{d_-(\eta)}} \underset{q \in P}{\prod} P_q(f,\eta,q^{-\frac{k}{2}}) , 
\\ 
  \eta(\mathcal{L}_p (V_{f,p}(\tfrac{k}{2})) ) & =
 \epsilon_p (\eta^\ast) \alpha_p^{-C_p(\eta^\ast)} \times 
\frac{P_p(\eta , \alpha_p^{-1} )}{P_p(\eta^\ast , \beta_p^{-1})} 
\times 
\frac{L(f,\eta^\ast,\tfrac{k}{2})}{\Omega_+(f)^{d_+ (\eta^\ast)} \Omega _-(f)^{d_-(\eta^\ast)}} \underset{q \in P}{\prod} P_q(f,\eta^\ast,q^{-\frac{k}{2}}) .
\end{align*}
}  
 Since we have $d_\pm(\eta) = d_\pm(\eta^\ast)$ and $C_p(\eta)= C_p(\eta^\ast)$, the following identity is obtained
by using \eqref{l-value-at-1} and the last two expressions above:
{}{
\begin{multline}\label{fn-equation-for-p-adic-l-function1}
\frac{\eta^\ast(\mathcal{L}_p (V_{f,p}(\tfrac{k}{2})) )}{\prod_{q \in P_0} P_q(f,\eta, q^{-\tfrac{k}{2}}) }
\\ = \frac{\eta(\mathcal{L}_p (V_{f,p}(\tfrac{k}{2})) )}{\prod_{q \in P_0} P_q(f,\eta^\ast, q^{-\tfrac{k}{2}}) } \omega(f, \eta) \frac{\epsilon_p(\eta)}{\epsilon_p(\eta^\ast)}\frac{P_p(f,\eta,p^{-\tfrac{k}{2}})}{P_p(f,\eta^\ast,p^{-\tfrac{k}{2}})}\frac{P_p(\eta^\ast , \alpha_p^{-1} )}
{P_p(\eta , \alpha_p^{-1})}\frac{P_p(\eta^\ast , \beta_p^{-1} )}{P_p(\eta , \beta_p^{-1})}
\end{multline}
}
In \eqref{fn-equation-for-p-adic-l-function1}, we have used $P_0=P\setminus \{ p\}$. Further, as we have 
{}{
$\eta(\mathcal{L}_p (V_{f,p}(\tfrac{k}{2})) ^\iota) = \eta^\ast (\mathcal{L}_p (V_{f,p}(\tfrac{k}{2})))$}, 
we get the following identity  from \eqref{fn-equation-for-p-adic-l-function1}: 
{}{
\begin{multline}\label{761}
\frac{\eta(\mathcal{L}_p (V_{f,p}(\tfrac{k}{2}))^\iota )}{\prod_{q \in P_0} P_q(f,\eta, q^{-1}) }
\\ = \frac{\eta(\mathcal{L}_p (V_{f,p}(\tfrac{k}{2})))}{\prod_{q \in P_0} P_q(f,\eta^\ast, q^{-\frac{k}{2}}) } \omega(f, \eta) \frac{\epsilon_p(\eta)}{\epsilon_p(\eta^\ast)}\frac{P_p(f,\eta,p^{-\frac{k}{2}})}{P_p(f,\eta^\ast,p^{-\frac{k}{2}})}\frac{P_p(\eta^\ast , \alpha_p^{-1} )}{P_p(\eta , \alpha_p^{-1})}\frac{P_p(\eta^\ast , \beta_p^{-1} )}{P_p(\eta , \beta_p^{-1})}.
\end{multline}
}
Note that $\frac{\epsilon_p(\eta)}{\epsilon_p(\eta^\ast)}$ is a $p$-adic unit. Now we discuss that the factors at $p$  in \eqref{761}. First of all, note that if $\eta$ is the trivial Artin representation of $G$, then the factors at prime $p$ in \eqref{761} are the same for $\eta$ and $\eta^\ast$.

First, we check that the contributions of factors at $p$ is trivial in this particular case of a false-Tate extension. 
Recall that, by the $\ell$-independence of an Euler factor, we have 
{}{
\begin{align*}
& P_p (f, \eta, p^{-\frac{k}{2}}) = \mathrm{det} (1-\mathrm{Frob}_p^{-1} p^{-\frac{k}{2}} \mid (V_{f,\ell}(\tfrac{k}{2}) \otimes W_{\eta ,\ell} )^{I_p}), \\ 
& P_p (\eta, \alpha_p^{-1})= \mathrm{det} (1-\mathrm{Frob}_p^{-1} \alpha_p^{-1} \mid ( W_{\eta ,\ell} )^{I_p}), \\
& P_p (\eta, \beta_p^{-1})= \mathrm{det} (1-\mathrm{Frob}_p^{-1} \beta_p^{-1} \mid ( W_{\eta ,\ell} )^{I_p}),
\end{align*}
for any prime $\ell \neq p
 $}
 where $W_{\eta ,\ell}$ is defined in the paragraph preceeding \eqref{p-adiclfn-ellipticcurve}. %an Artin representation which is isomorphic to $W_\eta$ but defined over an  $\ell$-adic field. 
 Note that the inertia invariant subspace $W_{\eta ,\ell}^{I_p}$ is trivial for any non-trivial Artin representation 
$\eta$ of $G$.  {}{The representation  $V_{f,\ell}(\tfrac{k}{2})$ is unramified at $p$ since the conductor $N_f$ of $f$ is not divisible 
by $p$. Hence the inertia invariant space 
$\big(V_{f,\ell}(\tfrac{k}{2})\otimes W_{\eta ,\ell} \big)^{I_p}$ is trivial. Thus we have $P_p(f, \eta, p^{-\frac{k}{2}})=P_p(\eta, \alpha_p^{-1}) =P_p(\eta, \beta_p^{-1}) =1$ 
for any non-trivial Artin representation $\eta$ of $G$}. 
 By combining the discussions above with \eqref{761}, we have proved the following result. 
\begin{proposition}\label{772}
{}{Assume Conjecture \ref{conjecture:algebraicity} and Conjecture \ref{p-adic-lfun-existence-conjecture} for $V=V_{f,p}(\tfrac{k}{2})$} and $G=\mathrm{Gal}(K_\infty/\Q)$ where $K_\infty $ is a false-Tate curve extension as considered in Example \ref{example:2} and 
Section \ref{s5}. 
\par 
{}{Then the $p$-adic $L$-function $\mathcal{L}_p (V_{f,p}(\tfrac{k}{2}) ) \in K_1 (\La_\mathcal O(G)_{S^\ast} )$ satisfies the 
following interpolation property: 
\begin{equation}\label{771}
\eta(\mathcal{L}_p (V_{f,p}(\tfrac{k}{2}))^\iota )
= 
\omega(f, \eta)
\eta(\mathcal{L}_p (V_{f,p}(\tfrac{k}{2}))) \prod_{q \in P_0=P\setminus \{ p\}} 
\frac{P_q(f,\eta, q^{-\frac{k}{2}})}{P_q(f,\eta^\ast, q^{-\frac{k}{2}}) } 
\end{equation}
}
for every Artin representation $\eta$ of $G$. Here  $S$ and $S^*$ are given by  \eqref{nimoprtantatal} and $\mathcal O$ is the ring of integers of a certain finite extension of $\Q_p$. 
\end{proposition}

\begin{proposition}\label{proposition:coincidence_evaluation_errorterm}
Let us consider the setting of Section \ref{s5} where $K_\infty $ is a false-Tate curve extension. 
We take a lattice $T$ of $V=V_{f,p}(\tfrac{k}{2})$ and we set $A =T\otimes \mathbb{Q}_p /\mathbb{Z}_p$.  
{}{
Let $\xi_{E_1} \in K_1({\La_\mathcal O(G)}_{{S^\ast}})$ be the characteristic element of the preimage of the class 
$[E_1] \in K_0(\mathfrak M_{H}(G))$ of the error term 
$$
E_1= E_1^{A} \cong \underset{q \in P_1 \cup P_2}{\bigoplus} \mathrm{Ind}^{G}_{G_{q}}  T(-1) 
$$ 
}
of the functional equation of Selmer groups and $\mathcal O$ is the ring of integers of a certain finite extension of $\Q_p$
\par 
Then, for every Artin representation $\eta$ of $G$, we have the equality: 
{}{
$$
 \eta (\xi_{E_1}) = \displaystyle{\prod_{q \in P_0} }
\dfrac{P_q(f,\eta, q^{-\frac{k}{2}})}{P_q(f,\eta^\ast, q^{-\frac{k}{2}})} 
$$
}
modulo multiplication by $p$-adic units. 
\end{proposition}
Before proving Proposition \ref{proposition:coincidence_evaluation_errorterm}, we give some preparations. 
 
Let $\theta_n$ denote the Artin representation of $G= \mathrm{Gal}(K_\infty/\Q)$ is obtained by inducing a character $\phi_n$ of exact order $p^n$ of $\text{Gal}(\Q(\mu_{p^n}, a^{1/p^n})/\Q(\mu_{p^n} ))$ to $\text{Gal}(\Q(\mu_{p^n}, a^{1/p^n})/\Q)$. Then $\theta_n$ is irreducible, and every irreducible Artin representation $\eta$ of $G$ is either of the form $\eta = \psi$ or of the form $\eta= \theta_n\psi$ for an integer $n\geq 1$, 
where $\psi$ is a 1-dimensional character of $\text{Gal}(\Q(\mu_{p^\infty} )/\Q)$ (see \cite[Page 212]{dd}). 
\par 
For $q \in P_2$, Let $C$ be a free $\mathcal{O}$-module of finite rank equipped with a continuous unramified action 
of $G_{\mathbb{Q}_q}$. We recall that $H_q:= H \cap G_q$ is completely ramified at $q$ and $\Gamma_q: = G_q/{H_q}$ is unramified at $q$. We identify $H_q$ with 
inertia subgroup and we note that $G_{q}/{H_q}$ is topologically generated by $\mathrm{Frob}_q $. Hence, 
$G_{q}$ acts naturally on $C$ through the quotient $G_{q}/{H_q}$, which allows us to regard $C$ as a $\La_{\mathcal{O}}(G_{q})$-module. 
With the above explanation, we have a characteristic element $\xi^q_{C}$ of $P$ in $K_1(\Lambda_{\mathcal{O}}(G_{q})_{S_{q}^\ast})$. 
This especially gives us {}{$\xi^q_{T (-1)} \in K_1(\Lambda_{\mathcal{O}}(G_{q})_{S_{q}^\ast})$}
for every $q \in P_2$ since {}{$T (-1)$} is unramified for every $q \in P_2$. Here $\mathcal{O}$ can be taken to be the ring of integers of a $p$-adic field which contains the eigenvalues of the action of $\mathrm{Frob}_q$ on {}{$T$}.  
When $q \in P_1$, the action of $G_{\mathbb{Q}_q}$ on {}{$T (-1)$} factors through the quotient $G_{\mathbb{Q}_q} 
\twoheadrightarrow G_{q}$. This allows us to regard {}{$T (-1)$} as a $\Lambda (G_{q})$-module and we have a  
characteristic element {}{$\xi^q_{T (-1)}$ of $T (-1)$} in $K_1(\Lambda(G_{q})_{S_{q}^\ast})$ 
also for $q \in P_1$.

\par 
Let $q \in P_2$. Let us denote by 
$\Lambda (G_{q})^\sharp$ a free $\Lambda_{\mathcal{O}} (G_{q})$-module of rank one on which the decomposition group 
$G_{\mathbb{Q}_q}$ acts on $\Lambda_{\mathcal{O}} (G_{q})$ by the tautological action through the surjection $G_{\mathbb{Q}_q}\twoheadrightarrow G_{q}$. 
Let $h_q$ be a topological generator of $H_q$. 
We have the exact sequence of $\Lambda_{\mathcal{O}} (G_{q}) [G_{\mathbb{Q}_q}] $-modules: 
\begin{multline}\label{ex3001}
0 \lra  (h_q -1)
\left(\Lambda_{\mathcal{O}} (G_{q})^\sharp /(\mathrm{Frob}_q -1)\Lambda_{\mathcal{O}} (G_{q})^\sharp \right) 
\otimes_{\Lambda_{\mathcal{O}} (G_{q})}C 
\\ 
\lra \left(\Lambda_{\mathcal{O}} (G_{q})^\sharp /(\mathrm{Frob}_q -1)\Lambda_{\mathcal{O}} (G_{q})^\sharp \right) 
\otimes_{\Lambda _{\mathcal{O}}(G_{q})}C \lra 
C \lra 0.
\end{multline}
Let us denote the first term and the second term of the above exact sequence by $M$ and $N$ respectively. 
Then $M, N $ are in $K_0(\mathfrak M_{H}(G))$ and $\xi^q_{C} \in K_1(\Lambda_{\mathcal{O}}(G_{q})_{S_{q}^\ast})$ 
is equal to $\dfrac{\xi^q_M}{\xi^q_N} \in K_1(\Lambda_{\mathcal{O}}(G_{q})_{S_{q}^\ast})$ modulo multiplication by an element of 
$K_1(\Lambda_{\mathcal{O}}(G_{q}))$. Let us calculate the evaluations of $\xi^q_M$ and $\xi^q_N$ at characters $\psi$ and $\theta_n \psi$.

We need the following lemma.
%%%%%
\begin{lemma}\label{lemma3003}
Let $C$ be a free $\mathcal{O}$-module of rank one equipped with a continuous unramified action 
of $G_{\mathbb{Q}_q}$. Assume that the eigenvalue of $\mathrm{Frob}_q $ acting on $C$ is $x$.  
\begin{enumerate}
\item 
Let $M= (h_q -1)
\left(\Lambda (G_{q})^\sharp /(\mathrm{Frob}_q -1)\Lambda (G_{q})^\sharp \right) 
\otimes_{\Lambda (G_{q})}C 
$. Then we have 
$\psi ( \xi^q_M ) = \psi (\mathrm{Frob}_q ) -qx$
modulo a multiplication by unit in $\mathcal{O}^\times$. 
\item 
Let $N= 
\left(\Lambda (G_{q})^\sharp /(\mathrm{Frob}_q -1)\Lambda (G_{q})^\sharp \right) 
\otimes_{\Lambda (G_{q})}C 
$. Then we have 
$\psi ( \xi^q_N) = \psi (\mathrm{Frob}_q ) -x $ modulo a multiplication by a unit in $\mathcal{O}^\times$. 
\end{enumerate}
%Let  $P$  be a torsion $\La_{\mathcal{O}}(\Gamma_q)$-module which is free $\mathcal{O}$-module of rank $1$. 
%Assume that $P$ is an unramified $G_{0,q}$-module whose action is given through $G_{0,q} \twoheadrightarrow \Gamma_q$. If $\mathrm{Frob}_q $ acts on  $P$ with eigenvalue $e \in \mathcal{O}$, then the characteristic element $\xi^q_P \in K_1(\La(G_{0,q})_{S_q^\ast})$ 
%of $P$ is the class of the module $
%\frac{\big(e\widetilde{\mathrm{Frob}_q}-1\big)}
%Z{\big({qe\widetilde{\mathrm{Frob}_q}-1}\big)}$ in $ K_1(\La(G_{0,q})_{S_q^\ast})$ where $\widetilde{\mathrm{Frob}_q} \in G_{0,q}$ is a lift 
%of ${\mathrm{Frob}_q} \in \Gamma_q$. 
\end{lemma}

\begin{proof}[Proof of Lemma \ref{lemma3003}]
Since the second statement is easier, we start from the proof of the assertion (2). In this case $N$ is a cyclic 
$\Lambda (G_{q})$-module whose annihilator ideal is a principal ideal generated by $\mathrm{Frob}_q  -x \in \Lambda_{\mathcal{O}} (G_{q})$.      
So $\xi^q_N \in K_1 (\Lambda_{\mathcal{O}} (G_{q})_{S^\ast_q})$ has a representative in $\mathrm{Frob}_q  -x \in  (\Lambda_{\mathcal{O}} (G_{q})_{S^\ast_q})^\times$. This implies that 
$$
\psi (\xi^q_N ) = \psi (\mathrm{Frob}_q  -x ) = \psi (\mathrm{Frob}_q ) -x ,  
$$ 
modulo a multiplication by a unit in $\mathcal{O}^\times$, which proves the assertion (2).
\par 
For the assertion (1), be simple calculation, we can check that the annihilator ideal of $M$ is a principal ideal generated by 
$\mathrm{Frob}_q  -\frac{(h_q)^q -1 }{h_q -1}x \in \Lambda_{\mathcal{O}} (G_{q})$.  Since $\psi (h_q) =1$, we have 
$\psi (\frac{(h_q)^q -1 }{h_q -1}) =q$ which implies the equality of the assertion (1) modulo a multiplication by a unit in $\mathcal{O}^\times$.  
\end{proof}

\begin{proof}[Proof of Proposition \ref{proposition:coincidence_evaluation_errorterm}]
In order to show the desired equality, we start by calculating the analytic local terms {}{$\displaystyle{\prod_{q \in P_0} }
\dfrac{P_q(f,\eta, q^{-\frac{k}{2}})}{P_q(f,\eta^\ast, q^{-\frac{k}{2}})}$}. 
\par 
We can check that, when $\eta$ is of the form 
$\eta= \theta_n\psi$ for an integer $n\geq 1$, $(W_\eta )^{I_q}$ is trivial for any $q \in P_0$. Since $p \not=2$, 
we have {}{$\displaystyle{\prod_{q \in P_0} }
\dfrac{P_q(f,\theta_n\psi , q^{-\frac{k}{2}})}{P_q(f,(\theta_n\psi )^\ast, q^{-\frac{k}{2}})}=1$}. When $\eta$ is of the form 
$\eta= \psi$, we have 
{}{
\begin{equation}\label{equation:eulerfactor_at_q}
\dfrac{P_q(f,\psi, q^{-\frac{k}{2}})}{P_q(f,\psi^\ast, q^{-\frac{k}{2}})}=
\begin{cases}
\frac{1- \delta_q (\mathrm{Frob}_q )
\psi^\ast (\mathrm{Frob}_q ) q^{-1}}{1-\delta_q (\mathrm{Frob}_q )\psi (\mathrm{Frob}_q ) q^{-1}} & q \in P_1 , \\
\frac{(1-\alpha_q \psi^\ast (\mathrm{Frob}_q ) q^{-1})(1-\beta_q \psi^\ast (\mathrm{Frob}_q ) q^{-1})}{(1-\alpha_q \psi (\mathrm{Frob}_q ) q^{-1})(1-\beta_q  \psi (\mathrm{Frob}_q ) q^{-1})} & q \in P_2 , \\ 
1 & q \in  P_0 \setminus (P_1 \cup P_2)
\end{cases}
\end{equation}
}
where $\alpha_q$ and $\beta_q$ are eigenvalues of the action of $\mathrm{Frob}_q$ on {}{$V_{f,p}(\tfrac{k}{2})$} 
and {}{$\delta_q $ is a unramified quadratic character of $\Q^\times_q$ which is attached to the special representation 
$\pi_{f,q}$ at $q$ (see Example 3 of Section \ref{s5})}.  
\par 
Now we pass to the algebraic side. 
We have {}{$E_1 \cong \underset{q \in P_1 \cup P_2 }{\bigoplus} \mathrm{Ind}^{G_0}_{G_{0,q}}  T (-1)
$}. 
Recall that, $[E_1]$ denotes the image of $E_1$ in $K_0(\mathfrak M_{H_0}(G_0))$ and the characteristic element $\xi_{E_1}$ denotes any 
preimage of $[E_1]$ in $K_1(\La_\mathcal{O} (G)_{{S^*}})$ via the surjection from $K_1(\La_\mathcal{O} (G)_{{S^*}}) \lra K_0(\mathfrak M_{H}(G))$ (cf. \eqref{fundamental-exact-k-theory3}). Here $\mathcal O$ is the ring of integers of a finite extension of $\Q_p$ containing the eigenvalues of the action of $\mathrm{Frob}_q$ on {}{$V_{f,p}(\tfrac{k}{2})$}, for all $q \in P_2$. Thus $\xi_{E_1}$ is defined up to an element of $K_1(\La_\mathcal{O} (G))$. We have $
 K_1(\La_\mathcal{O} (G)_{{S}}) \cong  \frac{\La_\mathcal{O} (G)^\times_{{S}}}{[\La_\mathcal{O} (G)^\times_{{S}}, \La_\mathcal{O} (G_0)^\times_{{S}}]} $ 
  by \cite[Theorem 4.4]{cfksv} and there is a surjective map  
 \begin{equation}\label{equationexampleb211}
  \frac{\La_\mathcal{O} (G)^\times_{{S^*}}}{[\La_\mathcal{O} (G)^\times_{{S^*}}, \La_\mathcal{O} (G)^\times_{{S^*}}]} \rightarrow K_1(\La_\mathcal{O} (G)_{{S^*}}) \rightarrow 1.
  \end{equation}  
Let us recall that the following diagram, where vertical maps are induced by natural inclusion, commute: $$ 
\begin{CD}
K_1(\La_\mathcal{O} (G_{q})_{S_{q}^\ast}) @>>> K_0(\mathfrak M_{H_{q}}(G_{q})) \\ 
@VVV @VVV \\ 
K_1(\La_\mathcal{O} ( G)_{S^\ast}) @>>>  K_0(\mathfrak M_{H}(G)). \\ 
\end{CD}
$$ 
Hence, the evaluation at $\eta$ of the characteristic element $\xi_{E_1} \in K_1(\La_\mathcal{O} (G)_{{S^\ast}})$ is the product of the evaluations at $\eta$ of the characteristic elements {}{$\xi^q_{T (-1)} \in K_1(\La_\mathcal{O}(G_{q})_{{S^\ast_{q}}})$} at primes in $P_1$ and $P_2$.   
In order to complete the proof of the proposition, it suffices to show the equalities  
{}{
\begin{equation}\label{equation:evaluation_local_algebraic_exceptionaldivisor}
\begin{cases}
\theta_n\psi (\xi^q_{T (-1)}) = 1 & q \in P_1\cup P_2, \\
%\theta_n\psi (\xi^q_{T (-1)} )= 1 & q \in P_2, \\ 
\psi (\xi^q_{T (-1)}) = \frac{1-\delta_q (\mathrm{Frob}_q )\psi^\ast (\mathrm{Frob}_q ) q^{-1}}{1-\delta_q (\mathrm{Frob}_q )\psi (\mathrm{Frob}_q ) q^{-1}} & q \in P_1, \\
\psi (\xi^q_{T (-1)} )= \frac{(1-\alpha_q \psi^\ast (\mathrm{Frob}_q ) q^{-1})(1-\beta_q \psi^\ast (\mathrm{Frob}_q ) q^{-1})}{(1-\alpha_q \psi (\mathrm{Frob}_q ) q^{-1})(1-\beta_q  \psi (\mathrm{Frob}_q ) q^{-1})} & q \in P_2,
\end{cases}
\end{equation}
}
modulo multiplications by a $p$-adic unit where $\delta_q$ is as defined in 
\eqref{equation:eulerfactor_at_q}. % in $\mathbb{Z}^\times_p$.
In the rest of the proof, we prove the equation \eqref{equation:evaluation_local_algebraic_exceptionaldivisor}. 
\par 
Let $q \in P_2$. Note that we have $\alpha_q \beta_q =q^{-1}$.
Since $\mathrm{Frob}_q$ acts on $\mathbb{Z}_p (-1)$ by $q^{-1}$, the action of $\mathrm{Frob}_q$ on {}{$V_{f,p}(\tfrac{k}{2})$} has eigenvalues $ \beta^{-1}_q=\frac{\alpha_q}{q},\alpha^{-1}_q=\frac{\beta_q} {q}$. For a prime $q \in P_2$,  $\mathcal{O}$ contains $\alpha_q$ and $\beta_q$, by definition.  
Then the unramified $G_{\mathbb{Q}_q}$-module {}{$T$ has a decomposition 
$T \otimes_{\mathbb{Z}_p} \mathcal{O} \cong C_{\alpha_q} \oplus C_{\beta_q}$} where $ C_{\alpha_q}$ and $C_{\beta_q}$ are rank one $\mathcal{O}$-modules 
on which $\mathrm{Frob}_q$ acts by multiplication of $\beta^{-1}_q $ and $\alpha^{-1}_q$ respectively. 
By definition, the characteristic element $\xi^q_{T_pE(-1)}$ is equal to $\xi^q_{C_{\alpha_q}} \xi^q_{C_{\beta_q}}$. 
When $\eta$ is of the form $\eta= \psi$, we obtain:
{}{ 
\begin{align*}
\eta (\xi^q_{T (-1)}) & = \psi (\xi^q_{C_{\alpha_q}}) \psi (\xi^q_{C_{\beta_q}}) \\ 
& = 
\frac{(\psi (\mathrm{Frob}_q)-\beta^{-1}_q)
(\psi (\mathrm{Frob}_q)-\alpha^{-1}_q)}{(\psi(\mathrm{Frob}_q)-q\beta^{-1}_q)
(\psi(\mathrm{Frob}_q)-q\alpha^{-1}_q)}\\ 
& = 
\frac{(1-\alpha_q \psi^\ast (\mathrm{Frob}_q)q^{-1})
(1-\beta_q\psi^\ast (\mathrm{Frob}_q)q^{-1})}{(1- \alpha_q\psi(\mathrm{Frob}_q)q^{-1})
(1- \beta_q\psi(\mathrm{Frob}_q)q^{-1})} \times q^{-1} 
\end{align*}
}
by applying Lemma \ref{lemma3003}. 
Let us consider the case where $\eta$ is of the form $\eta= \theta_n \psi$, In this case, we also have 
{}{
$$
\eta (\xi^q_{T (-1)})  = \theta_n \psi (\xi^q_{C_{\alpha_q}}) \theta_n \psi (\xi^q_{C_{\beta_q}}) .
$$
}
Hence it suffices to prove that $\theta_n \psi (\xi^q_{C_{\alpha_q}}) =1$ and $\theta_n \psi (\xi^q_{C_{\beta_q}}) 
=1$. 
By the definition of $\theta_n$ explained after the statement of Proposition \ref{proposition:coincidence_evaluation_errorterm}, 
$\theta_n$ is an induced representation 
of a character $\phi_n$ of 
$\mathrm{Gal}(\Q_q (\mu_{p^n}, m^{1/p^n})/\Q_q (\mu_{p^n} ))$ to $\mathrm{Gal}(\Q_q (\mu_{p^n}, m^{1/p^n})/\Q_q )$ which is realized on 
the space 
\begin{equation}\label{equation:expression_induced_representation}
W=V_{\phi_n} \otimes_{\mathcal{K}} V_\psi \otimes_{\mathcal{K}} \mathcal{K} [\mathrm{Gal}(\Q_q (\mu_{p^n})/\Q_q )]^\sharp 
\end{equation}
where $V_{\phi_n}$ is a representation space of $\phi_n$ which is a one-dimensional vector space over the coefficient field 
$\mathcal{K}$ on which $\mathrm{Gal}(\Q_q (\mu_{p^n}, m^{1/p^n})/\Q_q (\mu_{p^n} ))$ acts and 
$V_\psi$ is a representation space of $\phi_n$ which is a one-dimensional vector space over $\mathcal{K}$ on which $\mathrm{Gal}(\Q_q (\mu_{p^n} )/\Q_q)$ acts.  
Let us calculate $\theta_n \psi (\xi^q_{C_{\alpha_q}}) $. In this situation, 
an element $g\in G_{q}$, $\theta_n \psi (g)$ is represented by a matrix of 
degree $[\mathbb{Q}_q (\mu_{p^n}):\mathbb{Q}_q]$ which is regarded as an element of $\mathrm{End}_{\mathcal{K}}(W \otimes_{\mathcal{O}}
C_{\alpha_q})$. 
Note also that the evaluation of elements in $K_1(\La_\mathcal{O}(G_{q})_{S_q^\ast})$ at Artin representation $\eta$ of $G$
is defined \cite[\S 3]{cfksv} through the Morita equivalence $K_1(M_n ( - )) =K_1( -  ) $ which is known to 
be obtained by taking the determinant of the matrix. 
As in the proof of Lemma \ref{lemma3003}, we calculate 
the ratio of the determinant of the matrix in $\mathrm{End}_{\mathcal{K}}(W\otimes_{\mathcal{O}}
C_{\alpha_q})$ 
induced by an annihilator $\mathrm{Frob}_q -\frac{(h_q)^q-1}{h_q -1} \in \Lambda_\mathcal{O} (G_{q})$  
of 
the cyclic module 
$$(h_q -1)
\left(\Lambda_\mathcal{O} (G_{q})^\sharp \Big/ (\mathrm{Frob}_q -1)\Lambda_\mathcal{O} (G_{q})^\sharp \right) 
=\left(\Lambda_\mathcal{O} (G_{q})^\sharp \Big/
(\mathrm{Frob}_q -\frac{(h_q)^q-1}{h_q -1})\Lambda_\mathcal{O} (G_{q})^\sharp   \right)
$$ and the determinant of the matrix in $\mathrm{End}_{\mathcal{K}}(W\otimes_{\mathcal{O}}
C_{\alpha_q})$
induced by an annihilator $\mathrm{Frob}_q -1 \in \Lambda_\mathcal{O} (G_{q})$ 
of the cyclic module
$ 
\left(\Lambda_\mathcal{O} (G_{q})^\sharp /(\mathrm{Frob}_q -1)\Lambda_\mathcal{O} (G_{q})^\sharp \right) $. 
The matrix representing $\theta_n \psi (\mathrm{Frob}_q)$ is a permuting matrix the basis labeled by elements of 
$\mathrm{Gal}(\mathbb{Q}_q (\mu_{p^n})/\mathbb{Q}_q)$ twisted by the eigenvalues of these elements acting on $P_\alpha$.  
The matrix representing $\theta_n (\frac{(h_q)^q -1 }{h_q -1})$ is a diagonal matrix whose diagonal entries are given by $\phi_n  (\frac{(h^g_q)^q -1 }{h^g_q -1})$ when $g$ runs through 
elements of $\mathrm{Gal}(\mathbb{Q}_q (\mu_{p^n})/\mathbb{Q}_q)$. Note that the product of   
$\phi_n (\frac{(h^g_q)^q -1 }{h^g_q -1})$ for $g \in \mathrm{Gal}(\mathbb{Q}_q (\mu_{p^n})/\mathbb{Q}_q)$ equals to $1$. 
Now the determinants of two matrices which appear in the calculation of $\theta_n \psi (\xi^q_{C_{\alpha_q}})$ are 
both equal to $(\psi (\mathrm{Frob}_q) \alpha_q )^{[\Q_q (\mu_{p^n} ) :\Q_q] }-1$,  
which implies that the ratio is equal to $1$, hence we have $\theta_n \psi (\xi^q_{C_{\alpha_q}})=1$. 
Since we prove $\theta_n \psi (\xi^q_{C_{\beta_q}})=1$ exactly in the same way, we finally obtain:
{}{ 
$$
\eta (\xi^q_{T (-1)})  = \theta_n \psi (\xi^q_{C_{\alpha_q}}) \theta_n \psi (\xi^q_{C_{\beta_q}}) =1.
$$ 
}
\par 
Let $q \in P_1$. Then the action of $G_{\mathbb{Q}_q}$ on {}{$T (-1)$} is not unramified, but 
we can define {}{$\xi^q_{T (-1)} \in K_1(\Lambda_{\mathcal{O}}(G_{q})_{S_{q}^\ast})$} as explained earlier, with $\mathcal{O}$ as before. 
Since we have an exact sequence 
{}{
\begin{multline*} 
0 \longrightarrow \left(\Lambda_\mathcal{O} (G_{q})^\sharp /(\mathrm{Frob}_q -1)\Lambda_\mathcal{O} (G_{q})^\sharp \right) 
\otimes_{\Lambda_\mathcal{O} (G_{q})}\mathbb{Z}_p \otimes \delta_q 
\\ 
\longrightarrow \left(\Lambda_\mathcal{O} (G_{q})^\sharp /(\mathrm{Frob}_q -1)\Lambda_\mathcal{O} (G_{q})^\sharp \right) 
\otimes_{\Lambda_\mathcal{O} (G_{q})}T (-1)
\\ 
\longrightarrow \left(\Lambda_\mathcal{O} (G_{q})^\sharp /(\mathrm{Frob}_q -1)\Lambda_\mathcal{O} (G_{q})^\sharp \right) 
\otimes_{\Lambda_\mathcal{O} (G_{q})}\mathbb{Z}_p (-1) \otimes \delta_q 
\longrightarrow 
0 
\end{multline*}
}
in the category $\mathfrak M_{H}(G)$, we calculate 
{}{$\eta (\xi^q_{T (-1)})$} to be $\eta (\xi^q_{\mathbb{Z}_p \otimes \delta_q })\eta (\xi^q_{\mathbb{Z}_p(-1)\otimes \delta_q })$.  
The calculation of $\eta (\xi^q_{\mathbb{Z}_p\otimes \delta_q })$ and $\eta (\xi^q_{\mathbb{Z}_p(-1)\otimes \delta_q })$ for 
$q \in P_1$ goes exactly in the same way as that of $\eta (\xi^q_{C_{\alpha_q}})$ and  $\eta (\xi^q_{C_{\beta_q}})$ 
for $q \in P_2$ and this completes the proof of Proposition \ref{proposition:coincidence_evaluation_errorterm}. 
\end{proof}

For any $M$ in $\mathfrak M_H(G)$ and for any Artin representation $\eta$ of $G$, we introduce a notation  $\eta([M]):= \eta(\xi_M)$, which is well defined up to a $p$-adic unit.

\begin{theorem}\label{1100}
{}{Let us assume the setting above of 
an ordinary normalized eigen elliptic cuspform $f$ of even weight $k\geq 2$ and square-free level $\Gamma_0 (N)$ 
with a false-Tate extension $K_\infty /\mathbb{Q}$}. 
{}{We take a lattice $T$ of $V=V_{f,p}(\tfrac{k}{2})$ and we set $A =T\otimes \mathbb{Q}_p /\mathbb{Z}_p$}. 
Assume that a conjectural $p$-adic $L$-function {}{$\mathcal{L}_p(V_{f,p}(\tfrac{k}{2})) \in K_1(\La_\mathcal{O}(G)_{{S^*}})$} 
with the interpolation property \eqref{p-adiclfn-ellipticcurve} %(also see \eqref{equation:interpolation_property}) 
exists. Recall from Section \ref{s5} that  
{}{$E_1^{A} \cong \underset{q \in P_1 \cup P_2}{\bigoplus} \mathrm{Ind}^{G}_{G_{q}}  T (-1) $} 
is the error term of the algebraic functional equation.
Then, for any Artin representation $\eta$ of $\mathrm{Gal}(K_\infty /\mathbb{Q})$, 
{}{
$$
\eta (\xi_{E_1^{A}}) = \displaystyle{\prod_{q \in P_0} }
\dfrac{P_q(f,\eta, q^{-\frac{k}{2}})}{P_q(f,\eta^\ast, q^{-\frac{k}{2}})} 
$$
}
modulo multiplication by a  $p$-adic unit.  
Hence, for any Artin representation $\eta$ of $\mathrm{Gal}(K_\infty /\mathbb{Q})$, we have the equality: 
{}{
\begin{equation}\label{equation:evaluated_functional_equation}
\frac{\eta \left( \left[ \mathrm{Sel}^{\mathrm{BK}}_A (K_\infty )^\vee  \right]\right) }{\eta \left( 
\left[ \left( \mathrm{Sel}^{\mathrm{BK}}_{A} (K_\infty )^\vee \right) ^\iota \right] \right)} 
= \frac{\eta \left( \mathcal{L}_p(V_{f,p}(\tfrac{k}{2})) \right)}{\eta \left( \mathcal{L}_p(V_{f,p}(\tfrac{k}{2}))^\iota\right)}
\end{equation}
}
modulo multiplication by a $p$-adic unit.%unit in the ring of inter of a finite extension of $\Q_p$. 
\end{theorem}
\begin{proof}
The first assertion of the theorem  has been established in the proof of Proposition \ref{proposition:coincidence_evaluation_errorterm}. So we prove only the second part. 
Recall our second main theorem (Theorem 2) applied to this setting of an ordinary normalized eigen elliptic cuspform over the false-Tate curve extension  
implies the following 
{}{
 \begin{equation}\label{776}
\left[ \mathrm{Sel}^{\mathrm{Gr}}_{A} (K_\infty )^\vee  \right]  + \Bigl[ 
E_1^{A}\Bigr] = 
\left[\left( \mathrm{Sel}^{\mathrm{Gr}}_{A} (K_\infty )^\vee \right) ^\iota \right] \text{ in }  K_0 (\mathfrak{M}_H (G)).
\end{equation}
}
 Clearly, \eqref{776} can be reformulated  as follows:
 {}{
\begin{equation}\label{7751}
 \xi_{\mathrm{Sel}^{\mathrm{Gr}}_{A} (K_\infty )^\vee} \xi_{E_1^{A}}= \xi_{( \mathrm{Sel}^{\mathrm{Gr}}_{A} (K_\infty )^\vee)^\iota}  
\end{equation}
}
as elements of {}{$K_1(\La_\mathcal{O}(G)_{{S^*}})$} up to an element of  {}{$K_1(\La_\mathcal{O}(G))$}.  On the other hand, from Proposition \ref{772}, we have
{}{
$$
\eta(\mathcal{L}_p (V_{f,p}(\tfrac{k}{2})) ^\iota ) =\eta (\mathcal{L}_p (V_{f,p}(\tfrac{k}{2}))) \prod_{q \in P_0} 
\frac{P_p(f,\eta, p^{-\frac{k}{2}})}{P_p(f,\eta^\ast, p^{-\frac{k}{2}})}
$$
}
modulo multiplication by a $p$-adic unit.  Now by Proposition \ref{proposition:coincidence_evaluation_errorterm}, we know 
{}{$
\eta (\xi_{E_1^{A}})=
 \prod_{q \in P_0} \dfrac{P_q(f,\eta, q^{-\frac{k}{2}})}{P_q(f,\eta^\ast, q^{-\frac{k}{2}})},
$} 
up to a $p$-adic unit. Thus we have 
{}{\begin{equation}\label{last-equ-p}
\eta(\mathcal{L}_p (V_{f,p}(\tfrac{k}{2})) )= \eta(\mathcal{L}_p (V_{f,p}(\tfrac{k}{2}))^\iota)  \eta (\xi_{E_1^{A}})
\end{equation}
}
up to a $p$-adic unit. Evaluating \eqref{7751} at $\eta$ and then comparing it with \eqref{last-equ-p}, we obtain \eqref{equation:evaluated_functional_equation} up to a unit in a ring of integers of a finite extension of $\Q_p$. This completes the proof of the theorem. \end{proof}

\begin{rem}
{}{
The results of this section are a generalization of the corresponding results of Z\'abr\'adi. More precisely, the compatibility between the algebraic and the conjectural analytic functional equation over the false-Tate curve extension for an elliptic curve $E$ over $\Q$ with good, ordinary reduction at $p$ was established in \cite[Prop. 7.3]{zab}. 
With the same assumption on $E$, the compatibility over $\Q(E_{p^\infty})$ is discussed  in \cite[Proposition 7.2]{z2}.
}
\end{rem}
\begin{rem}
{}{
 Fukaya-Kato has a formulation of the conjectural functional equation of analytic $p$-adic zeta function \cite[Theorem 4.4.7]{fk}. 
Though they do not study the  functional equation on the algebraic side, their  functional equation of $p$-adic zeta function combined with 
their main conjecture in \cite[Conjecture 4.2.2]{fk}, which generalizes the main conjecture of \cite{cfksv} (see Conjecture \ref{773}), implies  
a functional equation in the algebraic side.  \par 
We remark that the algebraic object in the main conjecture of Fukaya-Kato is Selmer complex $SC(T,T^0)$ rather than the Selmer group. However, 
as remarked in \cite[4.5.3, Page 82]{fk}, we have the equality  $[SC(T,T^0)]= [\mathrm{Sel}^{\mathrm{Gr}}_{E_{p^\infty}} (K_\infty )^\vee]$ in  $K_0(\mathfrak M_H(G))$ in the setting and hypotheses of Theorem \ref{1100}. %This functional equation on the algebraic side is compatible with our result as well as 
%\cite[Equation 7.12]{zab}.
%%f analytic $p$-adic zeta function
% Fukaya-Kato has a formulation of the conjectural functional equation of $p$-adic zeta function \cite[Theorem 4.4.7]{fk}.  In their  formulation of the main conjecture, the $p$-adic Zeta function is related to a Selmer complex, denoted by $[SC(T,T^0)]$. Thus from the conjectural formulation of functional equation of Fukaya-Kato, one would get a functional equation of the Selmer complex.
%%We keep the setting and hypotheses of Theorem \ref{1100}. Then as remarked in \cite[4.5.3, Page 82]{fk}, in this situation, we have the equality  $[SC(T,T^0)]= [\mathrm{Sel}^{\mathrm{Gr}}_{E_{p^\infty}} (K_\infty )^\vee]$ in  $K_0(\mathfrak M_H(G))$ and   formulation of the main conjecture in \cite[Conjecture 4.2.2]{fk} coincides with the formulation of the main conjecture of \cite{cfksv} (see Conjecture \ref{773}).  
Thus the conjectural  functional equation of Fukata-Kato %\cite[Theorem 4.4.7]{fk} 
coincides with the conjectural  functional equation in Proposition \ref{772} or with \cite[Equation 7.12]{zab} and via the Iwasawa main conjecture in \ref{773}, shows the compatibility of our algebraic functional equation (as well as of \cite{zab}) with that of  \cite[Theorem 4.2.7]{fk}}.
\end{rem}

\appendix 
%%%%%%%%%%%%%%%%%%%%%%%%%%%%%%%
\section{Conjectural existence of analytic $p$-adic $L$-functions}\label{appendix}
In this section, we state a conjectural existence of analytic $p$-adic $L$-function for a given 
Galois representation $V$. For technical reasons, we restrict ourselves to the situation $K=\mathbb{Q}$ 
in the analytic side. Also, in order to talk about the algebraic part of special values of Hasse--Weil $L$-functions, 
we fix embeddings $\overline{\mathbb{Q}}\hookrightarrow \mathbb{C}$ and $\overline{\mathbb{Q}}\hookrightarrow 
\overline{\mathbb{Q}}_p$ simultaneously. 
\par 
Let $V$ be a $p$-adic Galois representation of $G_\mathbb{Q}$ which satisfies the condition (Geom) stated in 
Introduction. Then the $L$-function $L(s,V)$ given by the following Euler product is convergent on $\mathrm{Re}(s) > \dfrac{w}{2} +1$: 
\begin{equation}\label{most-general-complex-l-function}
L(s,V) = \prod_{q \not=p } \mathrm{det} (1-\mathrm{Frob}^{-1}_q q^{-s} ;V^{I_q })^{-1} 
\times \mathrm{det} (1-\varphi p^{-s} ; D_{\mathrm{pst}} (V)^{^{I_p}})^{-1} 
\end{equation}
where $w$ is the weight of motive whose $p$-adic \'{e}tale realization gives $V$, $D_{\mathrm{pst}} (V)$ is a potentially 
stable filtered module of Fontaine on which the operator $\varphi$ is acting. 
It is known that each Euler factor is a polynomial whose coefficients are in $\overline{\mathbb{Q}}$ and 
the Euler product is absolutely convergent on $\mathrm{Re}(s) >1+\dfrac{w}{2}$ due to Deligne. 
\par 
We recall the following well-known conjecture which is a folklore. 
%%%%%%%
\begin{conj}\label{conjecture:Hasse-Weil} 
Let $V \cong \mathcal{K}^d$ be a $p$-adic Galois representation of $G_\mathbb{Q}$ which satisfies the condition \rm{(Geom)} stated in  
Introduction. 
Then the following statements hold. 
\begin{enumerate}
\item[(1)]
The $L$-function $L(s,V)$ is meromorphically continued to the whole $\mathbb{C}$-plane with at most 
finitely many poles.  
Further, if $V$ does not contain any direct summand which is isomorphic to Tate twists of the trivial representation $\mathbb{Q}_p$, 
the $L$-function $L(s,V)$ is holomorphic on the whole $\mathbb{C}$-plane.  
\item[(2)] 
We have the following functional equation: 
$$
L(s, V) = a( V )^{\frac{w+1}{2} -s}
\epsilon (V ) L(1-s ,V^\ast (1) )
$$
where $a(V)$ (resp. $\epsilon (V)$) is the Artin conductor (resp. epsilon factor) for $V$ respectively, which we do not 
explain here. 
\end{enumerate}
\end{conj}
%%%%%%%%%%%%%%%%%%%%%%%%%%%%%%%%%%%%

%%%%%%%%%%%%%%%%%%%%%%%%%%%%%%%%%%%
\begin{conj}[Deligne]\label{conjecture:algebraicity} 
Let $V\cong \mathcal{K}^{2d}$ be a $p$-adic Galois representation of $G_\mathbb{Q}$ of even dimension which satisfies the condition \rm{(Geom)} stated in  
Introduction. 
For simplicity, we assume that $V$ does not contain any Artin representations as its subquotient. 
Further we assume the following conditions 
\begin{enumerate}
\item[(i)] The motive which corresponds to $V$ is critical in the sense of Deligne\cite{deligne}. 
\item[(ii)] Conjecture \ref{conjecture:Hasse-Weil} holds true for $V$. 
%\item[(iii)] For any Artin representation $W$ of $G_{\mathbb{Q}}$ with values in $\mathcal{K}$ whose Galois action factors through a finite quotient 
%of $G$, $V \otimes_{\mathcal{K}} W$ does not contain any direct summand which is isomorphic to Tate twists of the trivial representation $\mathbb{Q}_p$
\end{enumerate}

Then there exist two constants $\Omega_+ (V) , \Omega_- (V)$ such that, for any Artin representation $W$ of $G_{\mathbb{Q}}$ with values in 
$\mathcal{K}$ such that the motive which corresponds to $V \otimes W$ is critical in the sense of Deligne and 
Conjecture \ref{conjecture:Hasse-Weil} holds true for $V\otimes W$, we have 
$$
\dfrac{L(0 ,V\otimes_{\mathcal{K}} W)}{\Omega_+(V)^{{dd_+ (W)} } \Omega _-(V)^{{d d_- (W)}}} 
\in \overline{\mathbb{Q}} 
$$  
where $d_+ (W)$ (resp. $d_- (W)$) is the dimension of the eigenspace of complex conjugate with eigenvalue $+1$ (resp. $-1$). 
\end{conj}
%%%%%%%%%%%%%%%%%%%%%%%%%%%%%%%%%%%%

Based on the conjectural existence of $L$-values and periods, the conjectural existence of 
analytic $p$-adic $L$-function was formulated by \cite[Conj. 5.7]{cfksv} (for elliptic curves) and \cite[Thm 4.2.26]{fk} (for general motives). 
%%%%%%%%%%%%%%%%%%%%%%%%%%%%%%%%%%%
\begin{conj}\label{p-adic-lfun-existence-conjecture}
Let $V\cong \mathcal{K}^{2d}$ be a $p$-adic Galois representation of $G_\mathbb{Q}$ of even dimension which satisfies the condition \rm{(Geom)} stated in  
Introduction. For simplicity, we assume that $V$ does not contain any Artin representations as its subquotient. Let $\mathcal O$  be the  ring of integers of a finite extension of $\Q_p$ and define
\begin{equation}\label{nimoprtantatal}
S = \left\{ 
f \in \La_{\mathcal{O}}(G) \ \middle| \ \frac{\La_{\mathcal{O}}(G)}{\La_{\mathcal{O}}(G)f} \text{ is a finitely generated } 
\La_{\mathcal{O}}(H) \text{-module} \right\}.
\end{equation}
Set $S^* = \underset{n \geq 0}{\cup}p^n S$. Then $S$ and $S^*$ are left and right Ore set in $\La_{\mathcal{O}}(G)$ \cite[\S2, \S3]{cfksv}. In particular the localization $\La_{\mathcal{O}}(G)_{S^*}$ is well-defined and we have an exact sequence of $K$-groups
\begin{equation}\label{fundamental-exact-k-theory3}
K_1(\La_{\mathcal{O}}(G)) \lra K_1(\La_{\mathcal{O}}(G)_{S^*}) \stackrel{\delta_G}{\lra} K_0(\La_{\mathcal{O}}(G), \La_{\mathcal{O}}(G)_{S^*})= K_0(\mathfrak M_H(G)) \lra 0.
\end{equation}
Given $M \in \mathfrak M_H(G)$, we denote by $[M]$ its class in $K_0(\mathfrak M_H(G))$ and the preimage of $[M]$ in $K_1(\La_{\mathcal{O}}(G)_{S^*})$ by $\xi_M$. Note that, under the assumption that $G$ has no element of order $p$, the surjectivity of $\delta_G$ was proved in \cite[Proposition 3.4]{cfksv}.

 Also given an element $\xi \in K_1(\La_{\mathcal{O}}(G)_{S^*})$ and an Artin representation $\eta$ of $G$, \cite[\S3]{cfksv} associates a canonical evaluation map which gives rise to an element $\eta(\xi) \in \bar{\Q}_p \cup \{\infty\}$.

Further we assume the following conditions 
\begin{enumerate}
\item[(i)] The motive which corresponds to $V$ is critical in the sense of Deligne. 
\item[(ii)] Conjecture \ref{conjecture:Hasse-Weil} and Conjecture \ref{conjecture:algebraicity} hold true for $V$. 
\item[(iii)] The representation $V$ is crystalline at $p$ in the sense of Fontaine.  
\item[(iv)] The representation $V$ is of Panchishkin type at $p$ in the sense that 
$V$ has a $G_{\mathbb{Q}_p}$-stable filtration $\mathrm{F}^+_p V \subset V$ of dimension $d$ 
such that the Hodge-Tate weights of $\mathrm{F}^+_p V$ (resp. $V/\mathrm{F}^+_p V$) are all negative 
(resp. non-negative). 
%\item[(iii)] For any Artin representation $W$ of $G_{\mathbb{Q}}$ with values in $\mathcal{K}$ whose Galois action factors through a finite quotient 
%of $G$, $V \otimes_{\mathcal{K}} W$ does not contain any direct summand which is isomorphic to Tate twists of the trivial representation $\mathbb{Q}_p$
\end{enumerate}
For any non-trivial Artin representation $\eta$ of $G$ and for any prime number $q$, we define the 
Euler factor at $q$ as follows:  
 $$
 P_q (\eta, X) = 
 \begin{cases}\mathrm{det} (1-\mathrm{Frob}_q^{-1} X \mid (W_\eta )^{I_q}) & q\not=p \\  
 \mathrm{det} (1-\varphi X \mid D_{\mathrm{crys}}(W_\eta )) & q =p \\ 
 \end{cases}
 $$ 

Then, we have an analytic function $\mathcal{L}_p (V) \in K_1 (\La_{\mathcal{O}} (G)_{S^\ast} )$, where $\mathcal O$ is the ring of integers of a certain finite extension of $\Q_p$,  such that 
\begin{multline}\label{equation:interpolation_property}
\eta^\ast (\mathcal{L}_p (V )) \\ = 
\epsilon_p (W^\ast_\eta )^{-d}
( \alpha^{(1)} \cdots \alpha^{(d)} )^{-C_p(\eta) }\times 
\prod^d_{i=1} \frac{P_p(\eta^\ast , (\alpha^{(i)} )^{-1} )}{P_p(\eta , (\beta^{(i)})^{-1})}\times 
\dfrac{L_{P}(0 ,V\otimes_{\mathcal{K}} W_\eta )}{\Omega_+(V)^{{dd_+ (W_\eta )} } \Omega _-(V)^{{d d_- (W_\eta )}}} 
\end{multline}
for any non-trivial Artin representation $\eta$ of $G$, where 
$\epsilon_p (W^\ast_\eta )$ is the local epsilon factor at $p$ for 
$W^\ast_\eta$, 
$C_p(\eta)$ is the $p$-order of the conductor of $\eta$, 
$\alpha^{(1)} ,\ldots ,\alpha^{(d)}$ (resp. $\beta^{(1)} ,\ldots ,\beta^{(d)}$) is the eigenvalues of $\varphi$-operator acting on $D_{\mathrm{cris}} (F^+_p V)$ (resp. $D_{\mathrm{cris}} (V/F^+_p V)$).  {}{Here $\eta^\ast$ denotes the contragadient representation of $\eta$. Also } $P$ denotes the set of primes $q$ of $\Q$  such that the image of $I_{\mathbb{Q}_q}$ in $G$ is infinite and 
$L_{P} (s ,V\otimes_{\mathcal{K}} W^\ast_\eta )$ means the $L$-function which is obtained by removing  the Euler factors at every  prime $q \in P$ from the $L$-function $L (s ,V\otimes_{\mathcal{K}} W^\ast_\eta )$.
\end{conj}
%The conjectural functional equation for $\mathcal{L}_p (V) $ in our case should take the following shape:
%\begin{conj}[conjectural functional] Assume Conjectures A.1, A.2 and A.3. Then for any Artin representation $\eta $ of $G$,
%$\frac{\mathcal{L}_p (V)(\eta)}{\prod_{q \in P\setminus \{ p\}} P_q(V,\eta, q^{-1}) }$ and  $\frac{\mathcal{L}_p (V^\ast )^\iota(\eta)}{\prod_{q \in P\setminus \{ p\}} P_q(V^\ast,\eta^\ast, q^{-1}) }$ differ by $p$-adic units.
 %\end{conj}
We can also state the main conjecture in our setting assuming Conjectures  \ref{conjecture:Hasse-Weil},  \ref{conjecture:algebraicity} and \ref{p-adic-lfun-existence-conjecture}. 
\begin{conj}[Iwasawa Main Conjecture]\label{773} 
Assume Conjectures \ref{conjecture:Hasse-Weil},  \ref{conjecture:algebraicity} and \ref{p-adic-lfun-existence-conjecture}. Then  via the natural map
from $ K_1 (\La_{\mathcal{O}} (G)_{S^\ast} ) \stackrel{\delta}{\lra} K_0(\M_H(G))$, the image of $\mathcal{L}_p (V) $ in $K_0(\M_H(G))$, $\delta(\mathcal{L}_p (V))$ has the property, $$\delta(\mathcal{L}_p (V)) =[\mathrm{Sel}^{\mathrm{BK}}_{A} (K_\infty )^\vee] \in ~ K_0(\M_H(G)) $$\end{conj}

%$$\mathcal{L}_p (V)(\eta)  =a( V )^{\frac{w+1}{2} -s}\epsilon (V ) (\mathcal{L}_p (V^*(1)))^\iota(\eta) \underset{q \in P\setminus {p}}{\prod} \frac{\mathrm{det} (1-\mathrm{Frob}_\ell \ell^{-s} ;(V\otimes \eta)^{I_\ell })^{-1}}{ \mathrm{det} (1-\mathrm{Frob}_\ell \ell^{-s} ;(V^*(1)\otimes \eta^*)^{I_\ell })^{-1}}$$

\end{document}